\documentclass[12pt,reqno]{amsart}

\usepackage{amsmath,amssymb,amsthm}
\usepackage{amsaddr}
\usepackage{bm}
\usepackage{natbib}
\usepackage{graphicx,here,comment}
\usepackage{algorithm,algorithmic}
\usepackage{xcolor}
\usepackage{fullpage}
\usepackage{booktabs}
\usepackage{bbm}

\newtheorem{theorem}{Theorem}
\newtheorem{corollary}[theorem]{Corollary}
\newtheorem{lemma}[theorem]{Lemma}
\newtheorem{proposition}[theorem]{Proposition}
\newtheorem{assumption}{Assumption}
\newtheorem{example}{Example}
\theoremstyle{definition}
\newtheorem{definition}{Definition}
\newtheorem{remark}{Remark}

\newcommand{\R}{\mathbb{R}}
\newcommand{\N}{\mathbb{N}}

\newcommand{\mK}{\mathcal{K}}

\newcommand{\Ep}{\mathbb{E}}
\renewcommand{\Pr}{\mathbb{P}}

\renewcommand{\hat}{\widehat}
\renewcommand{\tilde}{\widetilde}
\newcommand{\pr}{\mbox{Pr}}
\newcommand{\argmin}{\operatornamewithlimits{argmin}}

\newcommand{\mone}{\textbf{1}}
\newcommand{\corrcoef}{\omega}
\newcommand{\orderomega}{\upsilon}
\newcommand{\RiskboundEuclidcoef}{\gamma}
\newcommand{\interpolatorboundEuclidcoef}{\varepsilon}
\newcommand{\RiskboundEuclidfunction}{g}

\newcommand{\trace}{\mathrm{tr}}

\DeclareMathOperator{\rank}{rank}
\DeclareMathOperator{\rad}{rad}

\usepackage{color}

\usepackage{newtxtext,newtxmath}

\title{Benign Overfitting of Non-Sparse High-Dimensional \\ Linear Regression with Correlated Noise}

\author{Toshiki Tsuda$^\dagger$ \and Masaaki Imaizumi$^{\dagger \ddagger}$}

\address{$^\dagger$The University of Tokyo, $^\ddagger$RIKEN Center for Advanced Intelligence Project}

\date{\today, \textit{Contact}: \textit{tsuda-toshiki@g.ecc.u-tokyo.ac.jp, imaizumi@g.ecc.u-tokyo.ac.jp}}

\allowdisplaybreaks

\begin{document}
\maketitle

\begin{abstract}
    We investigate the high-dimensional linear regression problem in the presence of noise correlated with Gaussian covariates. This correlation, known as endogeneity in regression models, often arises from unobserved variables and other factors. It has been a major challenge in causal inference and econometrics. When the covariates are high-dimensional, it has been common to assume sparsity on the true parameters and estimate them using regularization, even with the endogeneity. However, when sparsity does not hold, it has not been well understood to control the endogeneity and high dimensionality simultaneously. This study demonstrates that an estimator without regularization can achieve consistency, that is, benign overfitting, under certain assumptions on the covariance matrix. Specifically, our results show that the error of this estimator converges to zero when the covariance matrices of correlated noise and instrumental variables satisfy a condition on their eigenvalues. We consider several extensions relaxing these conditions and conduct experiments to support our theoretical findings. As a technical contribution, we utilize the convex Gaussian minimax theorem (CGMT) in our dual problem and extend CGMT itself.
\end{abstract}

\section{Introduction}

We consider a high-dimensional linear regression model with correlated noise and the $p$-dimensional true parameter $\theta_0$:
\begin{align*}
    Y_{i}=\langle X_{i}, \theta_{0}\rangle+\xi_{i}, ~~\Ep[X_i \xi_i] \neq 0, ~~ i= 1,...,n,
\end{align*}
where $n \in \N$ is the number of observations, $X_i$ is a $p$-dimensional centered Gaussian vector as observed covariates, $\xi_i$ is a centered Gaussian noise variable, and $Y_i$ is a response variable.
In this model, the noise variable $\xi_i$ is correlated with the covariate $X_i$.
We assume that the dimension $p$ is much larger than the number of observations $n$ ($p \gg n$) and also the true parameter $\theta_0$ does not have sparsity, that is, the $p$ coordinates of $\theta_0$ are not restricted to zero.
In this high-dimensional setting, we adopt an \textit{instrumental variable} framework, that is, assuming that there exists variable $Z_i$ such that $\Ep[Z_i \xi_i] = 0$, we investigate risk of a \textit{ridgeless estimator} under certain conditions.

We can find many real situations where covariates and noise variables are correlated.
For example, when part of covariates is not observed and the effects are included in a noise variable $\xi_i$, the observed covariates $X_i$ and the noise $\xi_i$ are often correlated.
In this situation, several statistical methods are biased as they require the independent property between the covariates $X_i$ and the noise variables $\xi_i$.
This situation is referred to as \textit{endogeneity}, especially in the domain of econometrics.
One of the most well-known solutions for endogeneity is the method of \textit{instrumental variables} by \citep{stock2002survey}, which utilizes a variable $Z_i$ that is uncorrelated with the noise variable but approximates the covariates $X_i$.
Two properties in this regard are called \textit{exclusion restriction} and \textit{relevance restriction} and are fundamental conditions for instrumental variables $Z_{i}$.
This method has been actively studied in a great deal of literature \citep{soderstrom2002instrumental,newey2003instrumental,baiocchi2014instrumental,isaiah2018weak}.

Estimation with instrumental variables has been extensively investigated in the high-dimensional setting as well, associated with the \textit{sparse} setting. 
As data become high-dimensional, the dimension of covariates of instrumental variables becomes larger than the number of observations $n$. 
To handle this situation, one can utilize the sparsity, which assumes that most of $p$ coordinates of the true parameters $\theta_0$ are zero, then estimate a small number of nonzero parameters using lasso-type regularization and its variants.
\citet{fan2014endogeneity} formulate a new generalized method of moments estimator for estimation and model selection with sparse parameters.
\citet{belloni2014high} and \citet{gautier2021high} focus on the case $\theta_{0}$ is (approximately) sparse and utilize a lasso-type regularization or the Dantzig selector in the instrumental variable framework.
\citet{gold2020inference} consider the one-step update approach and provide sufficient conditions  for inference. 
Several works \citep{belloni2010lasso,belloni2012sparse,belloni2017program,chernozhukov2015valid,chernozhukov2018double,belloni2022high,gautier2013pivotal} estimate nuisance parameters to deal with their high dimensionality by introducing a new instrumental variable orthogonal to nuisance parameters. 

In recent years, high-dimensional statistics with non-sparse parameters have been emerging rapidly. 
The establishment of methods for large-scale data, such as modern machine learning, has led to the emergence of many non-sparse data sets and models. 
Among several existing methods, a \textit{ridgeless estimator}, which perfectly fits the observed data without any regularization, has attracted much attention.
As a theoretical analysis for the setup, \citet{belkin2019reconciling} and \citet{hastie2022surprises} analyze the ridgeless estimator of high-dimensional linear regression models without sparsity using random matrix theory.
\citet{bartlett2020benign} utilize the notion of effective ranks of a covariance matrix to show convergence of the ridgeless estimator in high-dimensional linear regression.
These studies have shown that the ridgeless estimator has several advantages over regularized estimators in the high-dimensional setting \citep{dobriban2018high,tsigler2020benign}.
These results have been extended in various applications \citep{bunea2022interpolating,li2022benign,frei2022benign,nakakita2022benign}.
However, despite successive developments, these theories are still restrictive and require independence on noise variables; hence, they are not flexible enough to analyze the instrumental variable framework.

This study investigates a ridgeless estimator in non-sparse high-dimensional regression models with correlated noise.
As a setup, we assume that the data follow a centered Gaussian distribution and model the correlation between covariates and noise variables using instrumental variables. 
We then assess the estimation error of the ridgeless estimator using a projected residual mean squared error (projected RMSE). 
Consequently, we achieve the following two results.
(i) We show that the estimation error possesses an upper bound that is independent of the dimension $p$ of the covariates.
Specifically, this bound can be expressed by (normalized) correlation coefficients and the effective rank of the covariance matrix of the instrumental variable. 
(ii) We specify sufficient conditions on data distributions for which the derived upper bound converges to zero. 
Specifically, the sufficient condition is that covariance matrices of both instrumental variables and auxiliary variables for covariates must have appropriate effective ranks. 
These sufficient conditions are satisfied by several specific covariance matrices.
We derive these results for each case in which instrumental and auxiliary variables comprising the covariates are orthogonal or not.

Our above theoretical results suggest the following implications.
(i) In the correlated noise setting, the error of the ridgeless estimator is independent of the dimension $p$ under the conditions of instrumental variables. 
This means that we can estimate non-sparse high-dimensional parameters under the instrumental variable setting; in other words, benign overfitting occurs.
(ii) In this setting, the covariance of instrumental variables has a critical role in the risk.
Specifically, a covariance matrix of instrumental variables should have a certain number of ranks but also decay to some extent so that an eigenvalue sum does not diverge too quickly. 
Hence, this result aligns with the common idea that instrumental variables should not be weak.

On the technical side, we develop a proof technique for the evaluation of risks using Gaussian comparison inequalities. 
The most-related study \citep{bartlett2020benign} on non-sparse high-dimensional regression relies on matrix concentration inequalities and the leave-one-out method, but these approaches cannot handle the correlated noise in our setting. 
Therefore, we developed a proof using the convex Gaussian minimax theorem (CGMT) \citep{thrampoulidis2015regularized,thrampoulidis2018precise}, which allows a wider range of models. 
Rigorously, we  rewrite the risk of the ridgeless estimator in a minimax optimization problem using a dual form, then analyze it by CGMT. 
This approach was developed by \citet{koehler2022uniform}, and we applied it to the instrumental variable case. 
In addition, we derive a new extended CGMT and develop a method to analyze risk in situations where the covariance matrix is not orthogonal.

\subsection{Notation}
We denote $\|\cdot\|_{p}$ as $\|x\|_{p}=(\sum_{i}|x_{i}|^{p})^{1/p}$.
For a square matrix $A$, we define $\|A\|_{op}$ as an operator norm of $A$.  
For a positive semidefinite matrix $A$, $\|x\|_{A}^{2}:=\langle x,Ax\rangle$ denotes the Mahalanobis (semi-)norm. 
For a set $S \subset \R^p$, we define its radius as $\rad(S):=\sup_{s\in S}\|s\|_{2}$.
$\Sigma^{+}$ denotes the generalized inverse matrix of $\Sigma$.
$N(\mu, \Sigma)$ denotes a multivariate normal distribution with a mean $\mu \in \R^d$ and a symmetric positive definite matrix $\Sigma \in \R^{d \times d}$.
If $\Sigma$ is positive semi-definite with rank $k < d$, $N(0, \Sigma)$ denotes a distribution of $A X'$ where $A \in \R^{d \times k}$, $\Sigma = A A^\top$, and $X' \sim N(0, I)$ is a $k$-dimensional normal variable. 
$\mathbbm{1}\{\cdot\}$ denotes an indicator function.
For $x,x' \in \R, x \vee x' := \max\{x,x'\}.$
For $a,b \in \R$, $\lesssim$ and $ \gtrsim$ mean $a\leq Cb$ and $Ca\geq b$ for some absolute constant $C$, respectively.
For real-valued sequences $\{a_{n}\}_{n \in \N}$ and $\{b_{n}\}_{n \in \N}$, $a_{n}=O(b_{n})$ means $a_{n}/b_{n} \leq C$ for any sufficiently large $n$, $a_{n}=o(b_{n})$ means $a_{n}/b_{n}$ converges to zero as $n \to \infty$,  $a_{n}=\orderomega(b_{n})$ means $a_{n}/b_{n}$ diverges to $\infty$  as $n \to \infty$, and  $a_{n}=\Theta(b_{n})$ means $C_{1}b_{n}\leq a_{n}\leq C_{2}b_{n}$ holds for any sufficiently large $n$, where  $C, C_{1}$, and $C_{2}$ are some absolute constants. 
Let $\overset{\bold{p}}{\to} $ denote the convergence in probability.

\subsection{Paper Organization}
Section \ref{sec:preliminary} presents the problem setup and various definitions.
Section \ref{sec:orthogonal} provides an error analysis of the ridgeless estimator under the assumption that the covariance matrices of the noise and instrumental variables are orthogonal.
Section \ref{sec:non-orthogonal} provides an error analysis under a relaxation of orthogonality.
Section \ref{sec:general_norm} offers additional error analysis with a generalized norm.
Section \ref{sec:proof_outline} outlines the proof and explains the technical contributions.
Section \ref{Experiment} relates the experiments.
Section \ref{sec:discussion} presents the discussion and conclusion.

\section{Preliminary} \label{sec:preliminary}

\subsection{Setting}

We consider a linear regression problem with dependent noise and instrumental variables.
Let $n \in \N$ be the number of data points, $p,k \in \N$ be dimensions of variables, and $\Theta \subset \R^p$ be the parameter space.
Suppose that there exist $n$ i.i.d. variables $(X_{i},Z_{i},Y_{i})\in\mathbb{R}^{p}\times\mathbb{R}^{k}\times\mathbb{R}$ of the centered variables for $i=1,\cdots,n$ from the following data generating process
\begin{align}
    Y_{i}=\langle X_{i}, \theta_{0}\rangle+\xi_{i}, \mbox{~and~}
    X_{i}=\Pi_{0}Z_{i}+ u_{i}, \label{model:reg}
\end{align}
where $\theta_{0}\in\mathbb{R}^{p}$ is a true unknown parameter such that $\|\theta_{0}\|_{2}<\infty$, $\xi_{i}$ is a Gaussian variable from $N(0, \sigma^2)$, $\Pi_{0} \in \R^{p \times k}$ is an unknown matrix, and $u_i \in \R^p$ is a (potentially correlated) latent noise vector such that $\Ep[u_i|Z_{i}] = 0$.
Here, we refer to $X_i$ as a covariate and $Z_i$ as an instrumental variable.
We assume that $\Ep[X_i|Z_i]$ always exists. 
We define covariance matrices $\Sigma_x = \Ep[X_i X_i^\top]$, $\Sigma_z = \Ep[Z_i Z_i^\top]$, and $\Sigma_u = \Ep[u_i u_i^\top]$. 
Let $(\mathbf{X},\mathbf{Y},\mathbf{Z},\mathbf{\xi})$ denote design matrices and vectors $\mathbf{X}=(X_1,...,X_n)^\top , \mathbf{Y}=(Y_1,...,Y_n)^\top, \mathbf{\xi} = (\xi_1,...,\xi_n)^\top,$ and $ \mathbf{Z}=(Z_1,...,Z_n)^\top$.
Note that the covariance matrices $\Sigma_z$ and $\Sigma_u$ need not be positive definite, that is, positive semi-definite is sufficient for our analysis.

We describe how these variables are related.
We define $\corrcoef \in \R^p$ as the correlation between the covariate $X_i$ and the noise $\xi_i$:
\begin{align*}
    &\corrcoef := \Ep[X_{i}\xi_{i}]\neq 0.
\end{align*}
Further, we assume that the instrument $Z_i$ satisfies the following moment condition:
\begin{align*}
    &\Ep[\xi_{i}|Z_{i}]=0,
\end{align*}
which implies the instrument $Z_i$ and its noise $\xi_i$ are uncorrelated, that is, $\Ep[Z_{i}\xi_{i}]=0$.

\begin{remark}[Modeling with $\Pi_0$] \label{remark:pi0}
We employ the modeling \eqref{model:reg}, because of the following two reasons.
First, this model is often used in applied fields (e.g., econometrics and psychostatistics) \cite{newey2003instrumental,chen2012estimation}, that study a specific interpretation of instrumental variables. 
Second, the usage of the coefficient $\Pi_0$ yields the property $\mathbb{E}[u_i|Z_i] = 0$, which simplifies theoretical analysis for an estimation error. 
\end{remark}

We make an assumption concerning the problem.
\begin{assumption}[Gaussianity] \label{asmp:Gaussianity}
Assume $X_{i}$ and $\xi_{i}$ are normally distributed, that is,
\begin{align*}
    &X_{i}\sim N(0,\Sigma_x),\quad \xi_{i}\sim N(0,\sigma^{2}). 
\end{align*}
\end{assumption}
For $X_{i}$, Assumption \ref{asmp:Gaussianity} enables us to use the convex Gaussian minimax theorem (CGMT), which is a central tool to derive the upper bound for the risk. 
A possible way to mitigate Gaussianity includes the application of universality \cite{montanari2022universality,han2022universality}.
As long as $X_{i}$ is Gaussian, $Z_{i}$ and $u_{i}$ do not have to be Gaussian.

\subsection{Measure for Estimation Error}

A goal of the setting is to estimate the true parameter $\theta_0$ in a high-dimensional setting, that is,  $p,k\gg n$, without the sparse setting. 
Specifically, for $\theta \in \Theta$, we consider a residual mean squared error (RMSE) projected on a space of $Z$:
\begin{align}
    {E\left[\left(E[\langle {\theta}, X \rangle-\langle\theta_{0}, X \rangle|Z]\right)^{2}\right]}, \label{def:projected_RMSE}
\end{align}
where the random element $(X,Z)$ is an i.i.d. copy of $(X_i,Z_i)$ that follows \eqref{model:reg}.
In the literature of nonparametric instrumental variables, the projected RMSE is often used to evaluate the convergence rate of the estimators \citep{ai2003efficient,chen2012estimation,dikkala2020minimax}. It is because we need to deal with ill-posedness in nonparametric instrumental variable estimators. By controlling the dependence of instrumental variables, it is also possible to evaluate the non-projected RMSE. For more details, see \citet{chen2012estimation}.
In our setting, we use this useful evaluation criterion because we face difficulty evaluating RMSE in non-sparse high-dimensional settings. 
Furthermore, it always holds that the projected RMSE is equal or small than the RMSE. Hence, our results are necessary conditions for the convergence of the RMSE.

Note that the projected RMSE can be expressed as a weighted norm $\|{\theta}-\theta_{0}\|^{2}_{\Xi_z}$ with a transformed covariance matrix $\Xi_z:=\Pi_{0}E[ZZ^\top ]\Pi_{0}^\top $: 
\begin{align*}
    \eqref{def:projected_RMSE} &= ({\theta}-\theta_{0})^\top E\left[E[X|Z]E[X^\top|Z]\right]({\theta}-\theta_{0})\\
    &=({\theta}-\theta_{0})^\top\Pi_{0} E\left[ZZ^\top\right]\Pi_{0}^\top({\theta}-\theta_{0}) \\
    &=\|{\theta}-\theta_{0}\|^{2}_{\Xi_z}.
\end{align*}
The second equation follows the property $\Ep[u_i|Z_i] = 0$, which follows the modeling \eqref{model:reg}.
The use of norms weighted by covariance matrices is common in non-sparse high-dimensional statistics. 
For example, in the usual linear regression setting, \citet{hastie2022surprises} and \citet{bartlett2020benign} study an estimation error in terms of a norm weighted by a covariance matrix of covariates $X$. 
As our setting utilizes the projected RMSE, it is natural to use a similar norm with $\Xi_z$.

\subsection{Ridgeless Estimator}
We consider an estimator with interpolation, that is, a prediction by an estimator perfectly corresponds to the response in the observed set of data, which always appears when $p \geq  n$ holds.
Rigorously, with an empirical squared risk
\begin{equation}
    \hat{L}(\theta)= \frac{1}{n}\sum_{i=1}^{n}(Y_{i}-\langle X_{i}, \theta\rangle)^{2}, \label{def:empirical_risk}
\end{equation}
the estimator with interpolation is a parameter $\Theta \subset \R^p$ that satisfies $\hat{L}(\theta)= 0$.
As there may be an infinite number of interpolators, we define a ridgeless estimator, also known as a minimal norm interpolator, as
\begin{align*}
    \hat{\theta}&=\argmin_{\theta \in \Theta : \hat{L}(\theta)=0}\|\theta\|_{2} =\mathbf{X}^{\top}(\mathbf{X}\mathbf{X}^{\top})^{+}\mathbf{Y}.
\end{align*}
Note that we can calculate the minimum norm interpolator only from $(\mathbf{X},\mathbf{Y})$.

Such estimators have been examined frequently in the context of the linear regression problem.
In particular, the motivation for examining the ridgeless estimator (the minimum norm interpolator) is that the gradient descent algorithm for learning parameters converges to a parameter with the smallest norm among parameters that minimize the loss (see Lemma 1 in \citet{hastie2022surprises}).

\section{Error Analysis: Orthogonal Case} \label{sec:orthogonal}

\subsection{Orthogonality Assumption}

In this section, we consider a setting in which there is orthogonality between the transformed covariance matrix of instrumental variables $\Xi_{z}=\Pi_{0}E[ZZ^\top ]\Pi_{0}^\top $ and the covariance matrix of the latent noise $\Sigma_{u}=E[uu^\top ]$. 
This situation simplifies our error analysis and is therefore an appropriate first step. 
This assumption will be relaxed in the next section.

Specifically, we consider the following assumption.
\begin{assumption}[Orthogonality Condition] \label{asmp:orthogonal}
$\Sigma_{u}, and \Xi_z$ are orthogonal, that is,  their sets of eigenvectors $\{\varphi_j\}_{j=1}^{J_u}, \{\varphi'_j\}_{j = 1}^{J_z} \subset \R^p$ are such that there exists the decompositions $\Sigma_u = \sum_{j=1}^{J_u} \lambda^u_j \varphi_j \varphi_j^\top$ and $\Xi_z = \sum_{j=1}^{J_z} \lambda^z_j \varphi'_j (\varphi'_j)^\top$ with $J_u + J_z = p$ and positive eigenvalues $\{\lambda_j^u\}_j$ and $\{\lambda_j^z\}_j$ satisfying $\varphi_j^\top \varphi'_{\ell} = 0$ for every $j$ and $\ell$.
\end{assumption}

Intuitively, the $p$-dimensional eigenspaces of $\Sigma_x$ are divided into $J_u$-dimensional eigenspaces of $\Sigma_u$ and $J_z$-dimensional eigenspaces of $\Xi_z$, which are orthogonal.
We note two points.
In this setting, the ranks of $\Sigma_u$ and $\Xi_z$ are $J_u$ and $J_z$, respectively; hence they are not full-rank. 
Consequently, we obtain the following equality: 
\vskip 0.1in
\noindent
\textbf{Lemma.}
\textit{
Assume Assumption \ref{asmp:orthogonal} holds. 
Then, the positive semidefinite matrices $\Xi_z$ and $\Sigma_{u}$ whose eigenspaces are orthogonal satisfy the following covariance splitting:
\begin{equation*}
    \Sigma_x=\Xi_z + \Sigma_{u}.
\end{equation*}
}
\vskip 0.1in
We will restate this result as Lemma \ref{lem:Covariance_IV} in the supplementary material and offer its proof.
This property is essential for our error analysis below, which uses the speed of decay of the eigenvalues.

\subsection{Result 1: Upper Bound on Projected RMSE}

Here, as the first primary result, we derive an upper bound for the projected RMSE of the ridgeless estimator.
As preparation, we introduce a notion of the effective rank for the  upper bound.
\begin{definition}[Effective Rank]\label{def:effrank}
For a positive semidefine matrix $\Sigma$, two types of the effective rank are defined as
\begin{equation*}
    r(\Sigma)=\frac{\trace(\Sigma)}{\|\Sigma\|_{\mathrm{op}}}\ \ \textit{and}\ \ R(\Sigma)=\frac{\trace(\Sigma)^{2}}{\trace(\Sigma^{2})}.
\end{equation*}
\end{definition}
This notion is a more elaborate version of the notion of matrix ranks, which uses the decay speed of the eigenvalues of a matrix to express the complexity of the matrix. 
Specifically, $r(\Sigma)$ denotes a trace of $\Sigma$ normalized by its largest eigenvalue, and $R(\Sigma)$ denotes the intrinsic complexity of $\Sigma$ considering the decay rate of the eigenvalues of $\Sigma$. 
As these effective ranks fully utilize the information of eigenvalues of $\Sigma$, they are useful in measuring the complexity of $\Sigma$ and the stable quantity compared with the usual rank, especially in the high-dimensional setting.
This has been used in dealing with concentration of random matrices \citep{koltchinskii2017concentration} and has also been applied to the analysis of over-parameterized linear regression with independent noise \citep{bartlett2020benign,koehler2022uniform,tsigler2020benign}.

Using the notion of effective rank, we define an auxiliary coefficient as follows.
For $\delta \in (0,1)$, we define 
\begin{align*}
    \eta(\delta) := \sqrt{\log(1/\delta)} \left( \frac{1}{\sqrt{r(\Xi_z)}} + \sqrt{\frac{ \mathrm{rank}(\Sigma_u)}{n}} + \frac{n }{R(\Xi_z)} \right).
\end{align*}
This coefficient $\eta(\delta)$ becomes asymptotically negligible under appropriate conditions, which will be presented in the latter half of this section.

We develop a generic bound for the projected RMSE $\|\hat{\theta}-\theta\|_{\Xi_z}^{2}$.
With the result of Corollary \ref{Variant_Corollary2} and Theorem \ref{Variant_Theorem2}, we obtain the following sufficient conditions for benign overfitting.
Recall that we define $\Sigma_u^{+}$ as the generalized inverse matrix of $\Sigma_{u}$.

\begin{theorem}[Projected-RMSE Bound]\label{Variant_Theorem3}
Fix any $\delta\leq1/2$. Under Assumptions \ref{asmp:Gaussianity}-\ref{asmp:orthogonal} with covariance splitting $\Sigma_x=\Xi_z + \Sigma_{u}$,
suppose that $n$ and the effective ranks are such that $R(\Xi_z)\gtrsim \log(1/\delta)^{2}$ and $\eta(\delta)\leq1$.
Define $\psi(t) = t + t^2$ and  $\tilde{\sigma}^2 := \sigma^2 - \|\corrcoef\|_{\Sigma_u^+}^2\geq 0$. 
Then, with probability at least $1-\delta$, it holds that
\begin{align}\label{Benign}
\|\hat{\theta}-\theta_{0}\|_{\Xi_z}^{2}&\lesssim (1+\eta(\delta))(1 \vee \tilde{\sigma}) \psi \left( (\|\Sigma_u^+ \corrcoef\|_{2}+\|\theta_{0}\|_{2})\sqrt{\frac{\trace(\Xi_z)}{n}}\right).
\end{align}
\end{theorem}
This upper bound consists of the following two parts: (i) the coefficient part $(1+\eta(\delta))(1 \vee \tilde{\sigma})$ reflects the asymptotically negligible eigenvalues and noises, and (ii) the principal part $\psi((\|\Sigma_u^+ \corrcoef\|_{2}+\|\theta_{0}\|_{2})\sqrt{{\trace(\Xi_z)} / {n}})$ describes a complexity of the  true parameter and the distribution of the data.
With this upper bound, an appropriate assumption on $\Sigma_u$ and $\Xi_z$ guarantees that the projected RMSE converges to zero as $n \to \infty $, which will be explained below.
Note that $\tilde{\sigma}^2 \geq 0$ follows from Lemma \ref{rho-sigma_inq}.

\begin{remark}[Comparison with the independent noise case]
We compare Theorem \ref{Variant_Theorem3} with the endogeneity to the result without the endogeneity.
Particularly, \citet{koehler2022uniform} develop an upper bound of the mean squared error of the ridgeless estimator as
\begin{align}\label{Comparison}
     \|\hat{\theta}-\theta_{0}\|_{\Sigma_x}^{2} \lesssim (1+\eta'(\delta)) (1 \vee \sigma) \psi \left( \|\theta_0\|_{2} \sqrt{\frac{\mathrm{tr}(\Sigma_2)}{n}} \right),
\end{align}
where $\Sigma_{1}$ and $\Sigma_{2}$ are some matrices such that $\Sigma_{x}=\Sigma_{1} + \Sigma_{2}$, and 
        $\eta'(\delta)= \sqrt{\log(1/\delta)}( {1} / {\sqrt{r(\Sigma_{2})}} + \sqrt{{\mathrm{rank}(\Sigma_1)}/ {n}} + {n } / {R(\Sigma_{2})} )$.
This result suggests several implications. 
First, our decomposition of $\Sigma_x$ in Theorem \ref{Variant_Theorem3} can be regarded as a specific case of the decomposition of $\Sigma_x$ by \citet{koehler2022uniform}.
Second, our bound in Theorem \ref{Variant_Theorem3} pays an additional cost to handle covariate correlations, such as the replacement of $\sigma$ with $\tilde{\sigma}$ and introducing a correlation coefficient $\|\Sigma_u^+ \corrcoef\|_{2}$ in \eqref{Benign}.

\end{remark}

\subsection{Result 2: Benign Condition for Consistency}
In this section, we further investigate the upper bound in Theorem \ref{Variant_Theorem3} and derive sufficient conditions for the upper bound to converge to zero. 
We also provide several examples of distributions satisfying the condition.

We first provide a basic condition that is widely used for over-parameterized models (e.g., \citet{bartlett2020benign}).
\begin{definition}[Basic condition] \label{cond:basic}
This condition requires that the value of the following three limits be zero:
    \begin{align}
        \lim_{n\rightarrow\infty}\frac{\rank(\Sigma_{u})}{n}=\lim_{n\rightarrow\infty}\frac{n}{R(\Xi_z)} = \lim_{n\rightarrow\infty}\|\theta_{0}\|_{2}\sqrt{\frac{\trace(\Xi_z)}{n}}= 0. \label{eq:benign_basic}
    \end{align}
\end{definition}
Their details are as follows:
\begin{itemize}
    \item[(i)] (Small  latent noise) The first term, ${\rank(\Sigma_{u})} / {n}$, describes the size of the  latent noise vector relative to $n$, and the condition requires that the  latent noise is small.
    \item[(ii)] (Large effective dimension) The second term, ${n} / {R(\Xi_z)}$, decreases as the effective rank $R(\Xi_z)$ is larger than $n$, which plays the role of an effective dimension in the over-parameterized model.
    \item[(iii)] (No aliasing) The third term, $\|\theta_{0}\|_{2}\sqrt{{\trace(\Xi_z)} / {n}}$, represents the magnitude of the error in a noiseless situation and intuitively plays a role similar to bias.
\end{itemize}
These assumptions are commonly used in the over-parameterized linear regression problem without endogeneity \citep{bartlett2020benign,koehler2022uniform,tsigler2020benign}.
We will provide examples of covariance matrices that satisfy these assumptions in Section \ref{sec:example_orthogonal}.

We derive a result where the projected RMSE converges to zero. 
We achieve this result by introducing new assumptions corresponding to the endogeneity in addition to the basic assumptions in Definition \ref{cond:basic}.
\begin{theorem}[Sufficient conditions]\label{Variant_Theorem12}
Under Assumptions \ref{asmp:Gaussianity} and \ref{asmp:orthogonal} with $\Sigma_x=\Xi_z + \Sigma_{u}$, let $\hat{\theta}$ be the ridgeless estimator.
Suppose that the basic condition in Definition \ref{cond:basic} holds, and the following condition is also satisfied:
\begin{align}
        \lim_{n\rightarrow\infty}\|\Sigma_{u}^{+}\corrcoef\|_{2}\sqrt{\frac{\trace(\Xi_z)}{n}} = 0. \label{eq:benign_cond_orthogonal}
    \end{align}
Then, the following holds:
\begin{align*}
    \|\hat{\theta}-\theta_{0}\|_{\Xi_z}^{2} \overset{\bold{p}}{\to} 0, ~(n \to \infty).
\end{align*}
\end{theorem}
This result states that condition \eqref{eq:benign_cond_orthogonal} is a key factor of the convergence of the projected RMSE to zero in the setting with endogeneity because the basic assumption in Definition \ref{cond:basic} is also needed in ordinary regression without endogeneity.
Intuitively, condition \eqref{eq:benign_cond_orthogonal} means that the replacement of $\sigma^{2}$ with $\tilde{\sigma}^{2}$ in Theorem \ref{Variant_Theorem3} is asymptotically negligible.
For condition \eqref{eq:benign_cond_orthogonal}, the structure of $\corrcoef$ plays an essential role because it is challenging to satisfy \eqref{eq:benign_cond_orthogonal} with only the property of $\Sigma_u^+$.
However, we have $\|(\Sigma_u^+)^{1/2} \corrcoef\|_{2} \leq \sigma^2$ (Lemma \ref{rho-sigma_inq}), which implies a slow increase of $\|\Sigma_u^+ \corrcoef\|_{2}$.
Another implication is about the first term in \eqref{eq:benign_basic}: a strong correlation between $X_i$ and $Z_i$ is necessary for benign overfitting.
This is suggested by the fact that $\rank(\Sigma_u) \geq  p - \min\{\mathrm{rank}(\Sigma_{z}),\mathrm{rank}(\Pi_{0})\}$ (see Proposition \ref{prop:pi_0}).

\begin{remark}[Relation to weakness of instrumental variables]
    Here, we discuss the relation of our results to the study of \textit{weak instrumental variables}. 
    It is known that having many instrumental variables with weak correlations reduces the efficiency of estimation \citep{stock2002survey}.
    In our theory, from the result in Theorem \ref{Variant_Theorem12}, one can also claim that the weak instrumental variables reduce the validity of the estimation in the over-parameterized setting.
    Specifically, the instrumental variable $Z_i$ with weak correlations will decrease the rank of $\Pi_0$ which increases the rank of $\Sigma_u$, and also decreases the effective rank $R(\Xi_z)$.
    These effects makes the assumptions \eqref{eq:benign_basic} in Definition \ref{cond:basic} less likely to hold. 
    Hence, our result in the over-parameterized setting implies almost the same claim on the weak instrumental variables, while our approach is different from the previous studies.

    One can also consider that the independent setting can be recovered by setting $\Pi_0=I_p$, $Z_i=X_i$, $u_i=0$, in which case $X_i$ and $Z_i$ are perfectly correlated.
    However, this setting does not satisfy our sufficient condition, specifically (iii) in Definition \ref{cond:basic}, and hence it is out of the purview of our theoretical framework.
\end{remark}

\begin{remark}[Necessary condition]
We discuss a necessary condition for the benign overfitting.
When the noise $\xi_i$ is independent of $X_i$, there is a necessary condition (or rather a necessary and sufficient condition) for the benign overfitting that the eigenvalue decay of $\Sigma_x$ has a specific rate, which is shown in Theorem 6 in \cite{bartlett2020benign}.
In contrast, when the noise is dependent as in our setting,  no necessary condition is clarified. 
This is because the correlation coefficient $\corrcoef$ increases the flexibility of the estimation error, and thus the eigenvalues of $\Sigma_x$ alone cannot describe the necessary condition. 
\end{remark}

\subsubsection{Examples} \label{sec:example_orthogonal}

In this section, we provide examples that satisfy the condition in Theorem \ref{Variant_Theorem12}.
The example here uses a matrix derived by \cite{bartlett2020benign} as a base matrix $\overline{\Sigma}$, then constructs a  latent noise covariance matrix $\Sigma_u$ and of the instrumental variable $\Xi_z$ based on the base matrix $\overline{\Sigma}$.
Throughout this section, we assume that $\|\theta_0\|_2 = o(\sqrt{n})$.

\begin{example} \label{example1}
    Consider the dimension $p \in \N \cup \{\infty\}$ and  a base matrix $\overline{\Sigma}$ whose  $i$-th largest eigenvalue has the form
    \begin{align*}
        \lambda_i = Ci^{-1}\log^{-\beta}(i+1), ~i=1,...,p,
    \end{align*}
    with some constant $C>0$ and $\beta>1$, and also assume condition \eqref{eq:benign_cond_orthogonal} holds.
    We further define a truncated version of $\overline{\Sigma}$ with a truncation level $k \leq p$ as
        $\overline{\Sigma}_{1:k} = U^\top \mathrm{diag}(\lambda_1,...,\lambda_k, 0,...,0) U$,
    where $U \in \R^{p \times p}$ is an orthogonal matrix generated from a singular value decomposition $\overline{\Sigma} = U^\top \mathrm{diag}(\lambda_1,...,\lambda_p) U$.
    Using the notion, we define our truncation level $k^{*}_{n}$ as
\begin{align}
    k^{*}_{n}:=\min\{k\geq 0:r(\overline{\Sigma}-\overline{\Sigma}_{1:k})>n\}, \label{def:truncation_lebel}
\end{align}
    which balances the complexities of the  latent noise and the instrumental variable.
    Then, we define the (transformed) covariance matrices of $u$ and $z$ as
\begin{align}  
    \Sigma_{u}=\overline{\Sigma}_{1: k^{*}_{n}},   \quad  
    \Xi_{z}=\overline{\Sigma}-\overline{\Sigma}_{1: k^{*}_{n}}. \label{def:covs_exp1}
\end{align}
\end{example}

The example is adapted to our setting with endogeneity by considering the example of a covariance matrix by \citet{bartlett2020benign}.
Rigorously, we set the covariance matrix by \citet{bartlett2020benign} as the base matrix $\overline{\Sigma}$ and decompose it under the appropriate cutoff level $k_n^*$ to the (transformed) covariance matrices. 
Importantly, this example can freely choose the dimension $p$ (even infinite is possible).
The following proposition shows that this example yields benign overfitting.

\begin{proposition}\label{example_1}
Consider Example \ref{example1}.
Assume $\|\theta_{0}\|_{2}=o(\sqrt{n})$. If $\overline{\Sigma}$ and $\corrcoef$ satisfy
\begin{align*}  
    \lambda_i=C i^{-1}\log^{-\beta}(i+1),\quad (U\corrcoef)_{i}=\Theta(i^{-1}\log^{-\beta}(i+1)),
\end{align*}
where $\beta>1$ and $C>0$, then $\Sigma_{u}$ and $\Xi_{z}$ defined in \eqref{def:covs_exp1} and associated $\corrcoef$ as $\|\Sigma_u^+ \corrcoef\|_2 = o(\sqrt{n})$ satisfy all the conditions in Definition \ref{cond:basic} and  Theorem \ref{Variant_Theorem12}.
\end{proposition}

\begin{example} \label{example2}
    We consider the dimension $p = p_n$, which increases faster than $n$, that is, $\forall c>0, \exists \Bar{n} \in \N, \forall n \geq \Bar{n}, p \geq cn $ holds. 
    Furthermore, consider a base matrix $\overline{\Sigma}$ whose  $i$-th largest eigenvalue has the form
    \begin{align*}
        \lambda_i = \gamma_{i}+\varepsilon_{n}, ~i=1,...,p,
    \end{align*}
    where $\{\gamma_i\}_i$ and $\{\varepsilon_n\}_n$ are sequences such that
\begin{align*}  
    \gamma_{i}=\Theta(\exp(-i/\tau)), \quad ne^{-o(n)}=\varepsilon_{n}p=o(n), 
\end{align*}
with some $\tau>0$.
We further assume condition \eqref{eq:benign_cond_orthogonal} holds.
Similar to Example \ref{example1}, we use the truncation level $k^{*}_{n}$ as \eqref{def:truncation_lebel}
and define the (transformed) covariance matrices of $u_{i}$ and $Z_{i}$ as
\begin{align}  
    \Sigma_{u}=\overline{\Sigma}_{1: k^{*}_{n}},   \quad  
    \Xi_{z}=\overline{\Sigma}-\overline{\Sigma}_{1: k^{*}_{n}}. \label{def:covs_exp2}
\end{align}
\end{example}
In the example, we consider the case where $p$ diverges faster than $n$. 
In this case, the eigenvalues consist of two terms: an exponentially decaying term, and a term that behaves like noise. 
The next proposition shows benign overfitting in this setting.

\begin{proposition}\label{example_2}
Consider Example \ref{example2}. Set eigenvalues of $\overline{\Sigma}$ as follows:
\begin{align*}  
    \lambda_i=\gamma_{i}+\varepsilon_{n},
\end{align*}
where $\gamma_{i}=\Theta(\exp(-i/\tau))$ and $\tau>0$. Assume $\|\theta_{0}\|_{2}=o(\sqrt{n})$.  If $p$ and $\corrcoef$ satisfy
\begin{align*}  
    p=\orderomega(n), \quad ne^{-o(n)}=\varepsilon_{n}p=o(n), \quad (U\corrcoef)_{i}=\Theta(\exp(-i/\tau)),
\end{align*}
then $\Sigma_{u}$ and $\Xi_{z}$ defined in \eqref{def:covs_exp2} and associated $\corrcoef$ as $\|\Sigma_u^+ \corrcoef\|_2 = o(\sqrt{n})$ satisfy all the conditions in Definition \ref{cond:basic} and Theorem \ref{Variant_Theorem12}.
\end{proposition}

\section{Error Analysis: Non-Orthogonal Case} \label{sec:non-orthogonal}

In this section, we relax the orthogonality condition of Assumption \ref{asmp:orthogonal} and study the sufficient conditions for benign overfitting when the covariance matrices $\Sigma_u$ and $\Xi_{z}$ are not orthogonal. 
The approach to derive the conditions is almost the same as in Section \ref{sec:orthogonal}; we first derive an upper bound for the projected RMSE, then use it to reveal sufficient conditions. 
To simplify the presentation, we defer the upper bounds to a later section and present only a theorem on the sufficient conditions.

\begin{theorem}[Sufficient conditions: Non-Orthogonal Case]\label{Non_ortho_Variant_Theorem12}
Under Assumption \ref{asmp:Gaussianity}, let $\hat{\theta}$ be the ridgeless estimator. Further, assume $\tilde{\sigma}^2 := \sigma^2 - \|\corrcoef\|_{\Sigma_u^+}^2>0$.  Suppose that the basic condition in Definition \ref{cond:basic} holds, and the covariance splitting $\Sigma_x=\Xi_z + \Sigma_{u}$ satisfies  the following conditions:
\begin{align}
    \lim_{n\rightarrow\infty}\frac{\|\Sigma_{u}^{+}\corrcoef\|_{2}}{\tilde{\sigma}}\sqrt{\frac{\trace(\Xi_{z})}{n}}=\lim_{n\rightarrow\infty}\frac{n}{R(\Xi_z)}\frac{\trace(\Sigma_{u}\Xi_{z})}{\trace(\Xi_{z}^{2})}=\lim_{n\rightarrow\infty}\corrcoef^\top \Sigma_{u}^{+}\Xi_{z} \Sigma_{u}^{+}\corrcoef=0. \label{eq:benign_cond_non_orthogonal}
\end{align}
Then, the following holds:
\begin{align*}
    \|\hat{\theta}-\theta_{0}\|_{\Xi_z}^{2} \overset{\bold{p}}{\to} 0, ~(n \to \infty).
\end{align*}
\end{theorem}

In this non-orthogonal case, the above three conditions  \eqref{eq:benign_cond_non_orthogonal} play a critical role, in addition to Definition \ref{cond:basic}. 
We provide explanations of the terms in \eqref{eq:benign_cond_non_orthogonal} one by one below.
\begin{enumerate}
    \item[(i)] (Non-degenerated noise) The first condition on $({\|\Sigma_{u}^{+}\corrcoef\|_{2}} / {\tilde{\sigma}})(\sqrt{{\trace(\Xi_{z})} / {n}})$ requires that the variance $\tilde{\sigma}^2$ be non-degenerated and condition \eqref{eq:benign_cond_orthogonal} hold. Therefore, this condition is a sufficient condition for the condition \eqref{eq:benign_cond_orthogonal}. 
    \item[(ii)] (Effective rank with non-orthogonality) The second condition on the term $({n} / {R(\Xi_z)}) ({\trace(\Sigma_{u}\Xi_{z})} / {\trace(\Xi_{z}^{2})})$ takes into account the effect of non-orthogonality on the effective rank $R(\Xi_z)$, which already appears in the basic condition in Definition \ref{cond:basic}. 
    This means that non-orthogonality term $({\trace(\Sigma_{u}\Xi_{z})} / {\trace(\Xi_{z}^{2})})$ has a role in reducing the effective rank $R(\Xi_z)$.
    \item[(iii)] (Mixed effect) The third condition on $\corrcoef^\top \Sigma_{u}^{+}\Xi_{z} \Sigma_{u}^{+}\corrcoef$ includes both the effects of the non-orthogonality and the correlation $\corrcoef$. This condition is asymptotically satisfied as $\Sigma_u$ and $\Xi_z$ gradually approach orthogonality.
\end{enumerate}

Of the conditions in \eqref{eq:benign_cond_non_orthogonal}, (i) and (iii) are necessary to handle the endogeneity. 
In other words, (i) and (iii) are always satisfied when $\corrcoef=0$ holds. 
However, condition (ii) is required to achieve benign overfitting under non-orthogonality even in the absence of endogeneity.
To make it clear, we reveal a sufficient condition for benign overfitting with non-orthogonality in the setting of ordinary linear regression without endogeneity ($\corrcoef = 0$).

\begin{theorem}\label{Exogenous_Variant_Theorem12}{(Sufficient conditions: Non-Orthogonal Case when $X_{i}$ and $\xi_{i}$ are independent)}
Under Assumption \ref{asmp:Gaussianity}, let $\hat{\theta}$ be the ridgeless estimator.
Suppose that $\corrcoef = 0$ holds.
Suppose that the basic condition in Definition \ref{cond:basic} holds, and there exists a sequence of covariance $\Sigma_x=\Sigma_{1}+\Sigma_{2}$ such that the following conditions hold:
\begin{align}\label{eq:benign_cond_non_orthogonal_indep}
    \lim_{n\rightarrow\infty}\frac{n}{R(\Sigma_{2})}\left(\frac{\trace(\Sigma_{1}\Sigma_{2})}{\trace(\Sigma_{2}^{2})}\right)=0.
\end{align}
Then, $ L(\hat{\theta})$ converges to $\sigma^{2}$ in probability where $L(\theta)=\Ep(y-\langle \theta,x\rangle)^{2}$.
\end{theorem}

Theorem \ref{Exogenous_Variant_Theorem12} states that condition \eqref{eq:benign_cond_non_orthogonal_indep} is a key factor in RMSE converging to zero in the setting without orthogonality. 
When we set $\Sigma_{1}=\Sigma_{u}$ and $\Sigma_{2}=\Xi_{z}$, condition (10) is exactly equal to condition (ii) in the above discussion. Intuitively, $\trace(\Sigma_{u}\Xi_{z})$ is the degree of non-orthogonality between $\Sigma_{u}$ and $\Xi_{z}$, and Theorem \ref{Exogenous_Variant_Theorem12} requires the degree to be small.

\subsection{Example} \label{sec:exmpale_non_orthogonal}

We provide an example, similar to those provided in Section \ref{sec:example_orthogonal}. 
That is, we first specify the base matrix $\overline{\Sigma}$, then construct (transformed) covariance matrices based on it. 
Note that the definition of the dimension and the way of decomposition are slightly different.
Throughout this section, we also assume that $\|\theta_0\|_2 = o(\sqrt{n})$.

\begin{example}[Non-orthogonal version of Example \ref{example1}]
\label{example3}
    Consider the dimension $p=qn$ with some $q>1$, and a base matrix $\overline{\Sigma}$ whose  $i$-th largest eigenvalue has the form
    \begin{align*}
        \lambda_i = Ci^{-1}\log^{-\beta}(i+1), ~i=1,...,p,
    \end{align*}
    with some constant $C>0$ and $\beta>0$.
    We also assume that $\|\Sigma_u^+ \corrcoef\|_2 = o(\sqrt{n})$, $\lim_{n\rightarrow\infty}  ( \sigma^{2}-\|\corrcoef\|^{2}_{\Sigma_{u}^{+}})>0$, and consider the truncation level as \eqref{def:truncation_lebel}.
    Then, we define the (transformed) covariance matrices with $\alpha > 1$:
\begin{align*}  
    \Sigma_{u}:=\left(1-\frac{1}{n^\alpha}\right)\overline{\Sigma}_{1: k^{*}_{n}},   \quad  
    \Xi_{z}:=\overline{\Sigma}-\left(1-\frac{1}{n^\alpha}\right)\overline{\Sigma}_{1: k^{*}_{n}}.
\end{align*}
\end{example}

The base matrix $\overline{\Sigma}$ used in this example is identical to that in Example \ref{example1}. 
In contrast, the decomposition to construct the (transformed) covariance matrices is different. 
The following result demonstrates the validity of this example.

\begin{proposition}\label{non-ortho_ex}
Consider Example \ref{example3}. Suppose $\lim_{n\rightarrow\infty}  ( \sigma^{2}-\|\corrcoef\|^{2}_{\Sigma_{u}^{+}})>0$ does hold.
Under the assumptions $\|\theta_{0}\|_{2}=o(\sqrt{n})$ and $ (U\corrcoef)_{i}=\Theta(i^{-1}\log^{-\beta}(i+1))$ as in Proposition \ref{example_1}, $\Sigma_{u}$, $\Xi_{z}$, and $\corrcoef$ defined above satisfy all the conditions in Definition \ref{cond:basic} and Theorem \ref{Non_ortho_Variant_Theorem12}.
\end{proposition}

Note that Theorem \ref{Exogenous_Variant_Theorem12} immediately holds from this proposition by setting $\Sigma_1 = \Sigma_u$ and $\Sigma_2 = \Xi_z$.

\section{Extension to General Norm} \label{sec:general_norm}

We extend the result of Theorem \ref{Variant_Theorem12} to the case when $\theta$ is measured in terms of a general norm. 
Let $\|\cdot\|$ be an arbitrary norm.
To achieve our aim, we introduce two definitions, the dual norm and effective $\|\cdot\|$-ranks.
\begin{definition}[Dual Norm]
The dual norm of norm $\|\cdot\|$ on $\mathbb{R}^{d}$ is $\|u\|_{*}:=\max_{\|v\|=1}\langle v,u \rangle$, and the set of all its
sub-gradients with respect to $u$ is $\partial \|u\|_{*}=\{v:\|v\|=1, \langle v, u\rangle=\|u\|_{*}\}$.
\end{definition}

\begin{definition}[Effective $\|\cdot\|$-rank] \label{def:effrank_gennorm}
The effective $\|\cdot\|$-ranks of a covariance matrix $\Sigma$ are listed as follows. Let $H$ be normally distributed with mean zero and variance $I_{d}$, that is, $H\sim N(0,I_{d})$. Denote $v^{*}$ as $\arg \min_{v\in\partial\|\Sigma^{1/2}H\|_{*}}\|v\|_{\Sigma}$. Then, we define
\begin{equation*}
    r_{\|\cdot\|}(\Sigma):=\left(\frac{E\|\Sigma^{1/2}H\|_{*}}{\sup_{\|u\|\leq1}\|u\|_{\Sigma}}\right)^{2}~\text{and}~ R_{\|\cdot\|}(\Sigma):=\left(\frac{E\|\Sigma^{1/2}H\|_{*}}{E\|v^{*}\|_{\Sigma}}\right)^{2}.
\end{equation*}
\end{definition}
Effective $\|\cdot\|$-ranks is a generalization of the effective rank in Definition \ref{cond:basic}, and the dual norm is necessary to define the general effective rank.

We provide basic conditions for general norm $\|\cdot\|$, which corresponds to Definition \ref{cond:basic} and advanced conditions \citet{koehler2022uniform} established.
\begin{definition}[Basic condition with general norm] \label{cond:basic_general_norm}
This condition requires that the value of the three limits be zero with respect to a general norm $\|\cdot\|$:
  \begin{align}
        \lim_{n\rightarrow\infty}\frac{\rank(\Sigma_{u})}{n}=\lim_{n\rightarrow\infty}\frac{n}{R_{\|\cdot\|}(\Xi_z)}=\lim_{n\rightarrow\infty}\frac{\|\theta_{0}\|\mathbb{E}\|\Xi_{z}^{1/2}H\|_{*}}{\sqrt{n}}=0. \label{eq:benign_basic_general}
    \end{align}
\end{definition}

Each condition in \eqref{eq:benign_basic_general} corresponds to conditions in \eqref{eq:benign_basic}. The first condition for small  latent noise remains unchanged. 
For the large effective dimension condition, we replace $R(\Xi_{z})$ with the general norm counterpart, $R_{\|\cdot\|}(\Xi_z)$. For the no aliasing condition, $\theta_{0}$ is measured in terms of any norm $\|\cdot\|$ and $\sqrt{\trace(\Xi_{z})}$ is replaced with $\mathbb{E}\|\Xi_{z}^{1/2}H\|_{*}$.

\begin{definition}[Advanced condition] \label{cond:advanced}
In addition to Definition \ref{cond:basic_general_norm}, we require the following two conditions: 
    \begin{align}\label{cond:advanced_eq}
        \lim_{n\rightarrow\infty}\frac{1}{r_{\|\cdot\|}(\Xi_z)}=\lim_{n\rightarrow\infty}\Pr(\|P_{u}v^{*}\|^{2}>1+\eta)=0,
    \end{align}
    for any $\eta>0$.
\end{definition}
We provide details of the terms in \eqref{cond:advanced_eq} below.
\begin{enumerate}
    \item[(i)] (Large effective dimension) The first term $1/r_{\|\cdot\|}(\Xi_z)$ decreases as the effective rank $r_{\|\cdot\|}(\Xi_z)$ becomes large as with the second condition in \eqref{eq:benign_basic_general}. In the Euclidean norm case, $1/r(\Xi_z)$ converges to zero as $n/R(\Xi_z)$ goes toward zero by definition.
    
    \item[(ii)] (Contracting $\ell_{2}$ projection condition) This condition implies the projected $v^{*}$ onto the space spanned by $\Sigma_{u}$ is asymptotically smaller than or equal to 1. This condition always holds in the Euclidean norm case because $\|P_{u}v^{*}\|^{2}_{2}\leq \|v^{*}\|^{2}_{2}=1$ holds. 
\end{enumerate}

For the projected RMSE to converge to zero, we introduce a new assumption corresponding to condition \eqref{eq:benign_cond_orthogonal} in Theorem 2 in addition to the conditions in Definitions \ref{cond:basic_general_norm} and \ref{cond:advanced}.
\begin{theorem}[Sufficient conditions]\label{Variant_Theorem11}
Under Assumptions \ref{asmp:Gaussianity} and \ref{asmp:orthogonal}, let $\hat{\theta}$ be the ridgeless estimator. Let $\|\cdot\|$ denote an arbitrary norm. Suppose that the basic conditions in Definitions \ref{cond:basic_general_norm} and \ref{cond:advanced} hold, and the covariance splitting $\Sigma_x=\Xi_z + \Sigma_{u}$ satisfies  the following conditions:
    \begin{equation*}
        \lim_{n\rightarrow\infty}\frac{\|\Sigma_u^+ \corrcoef\|\mathbb{E}\|\Xi_{z}^{1/2}H\|_{*}}{\sqrt{n}}=0.
    \end{equation*}
Then, the following holds:
\begin{align*}
    \|\hat{\theta}-\theta_{0}\|_{\Xi_z}^{2} \overset{\bold{p}}{\to} 0, ~(n \to \infty).
\end{align*}
\end{theorem}
As in the condition in \eqref{eq:benign_basic_general}, $\Sigma_u^+ \corrcoef$ is measured in terms of any norm $\|\cdot\|$ and $\sqrt{\trace(\Xi_{z})}$ is replaced with $\mathbb{E}\|\Xi_{z}^{1/2}H\|_{*}$.
If we consider the general norm, compared to the Euclidean norm case, it is possible we can relax some of the sufficient conditions for benign overfitting, especially the condition in Theorem \ref{Variant_Theorem12}. However, as we must incorporate additional advanced conditions outlined in Definition \ref{cond:advanced} in conjunction with the basic conditions presented in Definition \ref{cond:basic_general_norm}, it remains uncertain whether benign overfitting is more probable.

\section{Proof Outline} \label{sec:proof_outline}

\subsection{Approach with CGMT}

Our proof relies on two techniques: (i) describing the ridgeless estimator as a solution to an optimization problem and bounding the projected RMSE, and (ii) evaluating the solution by an extended version of the convex Gaussian minimax theorem (CGMT).
CGMT was introduced into high-dimensional statistics by \citet{thrampoulidis2015regularized,thrampoulidis2018precise}. 
Furthermore, \cite{koehler2022uniform} discussed that CGMT can describe benign overfitting by \citet{bartlett2020benign} in the ordinary regression setting.
In this section, we deal with the non-orthogonal case results given in Section \ref{sec:non-orthogonal}, which can be easily applied to the orthogonal case in Section \ref{sec:orthogonal}.

We prepare some notations.
We define a normalized correlation coefficient $\rho = (\Sigma_u^{1/2})^+ \corrcoef$, which guarantees that $\|\rho\|_{2}^{2} \leq\sigma^{2}$ (see Lemma \ref{rho-sigma_inq}).
We also define $\mathbf{X} = (X_1,...,X_n)^\top$ as an $\R^{n \times p}$-valued random matrix, which has the form 
\begin{equation}
   \mathbf{X}\overset{\mathcal{D}}{=}\mathbf{W}_{1}\Xi_z^{1/2}+\mathbf{W}_{2}\Sigma_{u}^{1/2}, \label{eq:X_matrix}
\end{equation}
where $\mathbf{W}_{1}$ and $\mathbf{W}_{2}$ are $n\times p$ random matrices whose $i$-th row identically follows a joint distribution of $\xi_i$ such that
\begin{equation} \label{def:cov_wwxi}
\begin{pmatrix}
W_{1,i} \\
W_{2,i} \\
\xi_{i}
\end{pmatrix}
\sim N\left(
\begin{pmatrix}
\mathbf{0}_{p\times 1} \\
\mathbf{0}_{p\times 1} \\
0
\end{pmatrix},
\begin{pmatrix}
I_{p\times p} & \mathbf{0}_{p\times p}  & \mathbf{0}_{p\times 1}\\
\mathbf{0}_{p\times p} & I_{p\times p} & \mathbf{\rho} \\
\mathbf{0}_{p\times 1}^\top  & \mathbf{\rho}^\top  & \sigma^{2}
\end{pmatrix}
\right)
\end{equation}
for $i=1,...,n$.
Note that this form follows the  Gaussianity from Assumption \ref{asmp:Gaussianity}.

\subsection{Step (i): Bound Projected RMSE by Optimization Form}
First, we consider a uniform upper bound for the projected RMSE $E[(E[\langle {\theta}, X \rangle-\langle\theta_{0}, X \rangle|Z])^{2}]$ under the constraint that the estimator $\hat{\theta}$ is the ridgeless estimator (i.e., $\hat{L}(\hat{\theta}) = 0$).
Then, we transform it to a maximization problem with a constraint with some compact parameter space $\mK \subset \R^p$:
\begin{align*}
    \max_{\substack{\theta \in \mathcal{K}, \hat{L}(\theta)=0}}E\left[\left(E[\langle {\theta}, X \rangle-\langle\theta_{0}, X \rangle|Z]\right)^{2}\right]
    &=\max_{\theta\in\mathcal{K},\mathbf{X}\theta=\mathbf{Y}}\|\theta-\theta_{0}\|_{\Xi_z}^{2} \\
    &=\max_{\theta\in\mathcal{K},\mathbf{X}(\theta-\theta_{0})=\xi}\|\theta-\theta_{0}\|_{\Xi_z}^{2}.
\end{align*}
Using the surrogate Gaussians in \eqref{eq:X_matrix}, the upper bound above has the same distribution as the following term:
\begin{equation}\label{Variant_40}
    \Phi:=\max_{\substack{(\theta_{1},\theta_{2})\in S, \\ \mathbf{W}_{1}\theta_{1}+\mathbf{W}_{2}\theta_{2}=\xi}}\|\theta_{1}\|_{2}^{2},
\end{equation}
where we define
    $S:=\{(\theta_{1},\theta_{2}):\exists\theta\in\mathcal{K}\ s.t.\ \theta_{1}=\Xi_z^{1/2}(\theta-\theta_{0})\ and\ \theta_{2}=\Sigma_{u}^{1/2}(\theta-\theta_{0}) )\}$.
The details of the derivation are described in the proof of Lemma \ref{Variant_Lemma3} in the appendix.

Second, we approximate the distribution of the optimization problem \eqref{Variant_40} using CGMT.
CGMT approximates minimax optimization problems by a distribution of their simpler auxiliary problems.
Here, we present our variant of CGMT that can deal with correlation between variables, though we also use classical CGMT depending on the situation. 

\begin{theorem}[Extended CGMT]\label{VariantCGMT}
Let $\mathbf{W}:n\times d$ be a matrix with i.i.d. $N(0, 1)$ entries and suppose $G\sim N(0,I_{n})$ and $H\sim N(0,I_{d})$ are independent of $\mathbf{W}$ and each other. Let $S_{W}$ and $S_{U}$ be non-empty compact sets in $\mathbb{R}^{d}\times\mathbb{R}^{d'}$ and $\mathbb{R}^{n}\times\mathbb{R}^{n'}$, respectively, and let $\psi:S_{W}\times S_{U}\mapsto\mathbb{R}$ be an arbitrary continuous function. Define the Primary Optimization (PO) problem
\begin{equation}\label{POCGMT}
    \Phi(\mathbf{W}):=\min_{(\omega,\omega')\in S_{W}}\max_{(u,u')\in S_{U}}\langle u,\mathbf{W}\omega \rangle +\psi((\omega,\omega'),(u,u'))
\end{equation}
and the Auxiliary Optimization (AO) problem
\begin{equation}\label{AOCGMT}
    \phi(G,H):=\min_{(\omega,\omega')\in S_{W}}\max_{(u,u')\in S_{U}}\|\omega\|_{2}\langle G, u \rangle+ \|u\|_{2}\langle H, \omega \rangle +\psi((\omega,\omega'),(u,u')).
\end{equation}
If we suppose that $S_{W}$ and $S_{u}$ are convex sets and $\psi((\omega,\omega'),(u,u'))$ is convex in $(\omega,\omega')$ and concave in $(u,u')$, then $\Pr(\Phi(\mathbf{W}) > c)\leq2\Pr(\phi(G, H) \geq c)$ for any $c\in\mathbb{R}$.
\end{theorem}
This theorem is an extension of the original CGMT to split the variables to be optimized so that it can handle our regression model \eqref{model:reg} with the endogeneity.
Rigorously, this theorem allows correlation between the covariates and the error terms.

Using the extended CGMT in Theorem \ref{VariantCGMT}, we approximate the distribution of the problem \eqref{Variant_40} by
\begin{equation*}
\phi:=\max_{{(\theta_{1},\theta_{2})\in S:
\|\xi-\mathbf{W}_{2}\theta_{2}-G\|\theta_{1}\|_{2}\|_{2}\leq\langle\theta_{1},H\rangle}} \|\theta_{1}\|_{2}^{2},
\end{equation*}
where $G\sim N(0,I_{n})$ and $H\sim N(0,I_{d})$ are Gaussian vectors independent of $\mathbf{W}_{1}, \mathbf{W}_{2}, \xi$, and each other.
A distribution of this term is tractable because of the relatively simple form.
Namely, we obtain the following result.
In the case of a Euclidean norm ball, we set $\mathcal{K}:=\{\theta\in\mathbb{R}^{p}|\|\theta\|_{2}\leq B\}$. By combining the upper bound of $\mathcal{K}$, we can derive a simpler upper bound for the Euclidean norm.

\begin{corollary}\label{Variant_Corollary2}
There exists an absolute constant $C_{1}\leq 64$ such that the following is true. Assume Assumptions \ref{asmp:Gaussianity} and \ref{asmp:orthogonal} hold. Pick $\Sigma_x=\Xi_z + \Sigma_{u}$ and fix $\delta\leq 1/4$. Define $\tilde{\sigma}^2 := \sigma^2 - \|\corrcoef\|_{\Sigma_u^+}^2\geq 0$ and $\RiskboundEuclidfunction(t_{1},t_{2})=t_{1}^{2}-t_{2}^{2}$.  If $B\geq\|\theta_{0}\|_{2}$ and $n$ is large enough that $\RiskboundEuclidcoef(\delta)\leq1$, the following holds with probability at least $1-\delta$:
\begin{align*}
    \max_{\|\theta\|_{2}\leq B,  \mathbf{Y}=\mathbf{X}\theta}\|\theta-\theta_{0}\|_{\Xi_z}^{2}&\lesssim (1+\RiskboundEuclidcoef(\delta))\RiskboundEuclidfunction\left(B\sqrt{\frac{\trace(\Xi_{z})}{n}},\tilde{\sigma}\right),
\end{align*}
where we define
\begin{equation*}
    \RiskboundEuclidcoef(\delta):=\sqrt{\log(1/\delta)} \left( \frac{1}{\sqrt{r(\Xi_z)}} + \sqrt{\frac{ \mathrm{rank}(\Sigma_u)}{n}}\right).
\end{equation*}

\end{corollary}

In this corollary, the radius $B$ of $\mK$ plays an important role. 
That is, the bound in Corollary \ref{Variant_Corollary2} is valid only when the norm $\|\theta\|_2$ is no more than $B$. 
Here, our remaining task is to show that such a $B$ exists. 
In the next step, we will examine the norm $\|\hat{\theta}\|_2$ to show the existence of such $B$.

\subsection{Step (ii): Bound Norm of Estimator}

As the next step, we specify an upper bound on the norm of the solution, which is equivalent to deriving an upper bound of $B$ that appears in the constraint in Corollary \ref{Variant_Corollary2}.
To show the consistency of ridgeless estimators, we need to specify the value of $B$ so that $\mathcal{K}$ includes some parameters.

In the following theorem, we obtain the Euclidean norm bound for the ridgeless estimator.
To achieve this result, we again use CGMT from Theorem \ref{VariantCGMT}.
\begin{theorem}[Euclidean norm bound; special case of Theorem \ref{Non_ortho_Variant_Theorem4}]\label{Non_ortho_Variant_Theorem2}
Fix any $\delta\leq1/4$.  Suppose $\Sigma_x=\Xi_z+\Sigma_{u}$ and $\tilde{\sigma}^2 := \sigma^2 - \|\corrcoef\|_{\Sigma_u^+}^2>0$. If $n$ and the effective ranks are such that $\interpolatorboundEuclidcoef(\delta)\leq1$ and $R(\Xi_z)\gtrsim \log(1/\delta)^{2}$, then with probability at least $1-\delta$, it holds that
\begin{equation*}
\|\hat{\theta}\|_{2}\lesssim(1+\interpolatorboundEuclidcoef(\delta))^{1/2}\left(\|\theta_{0}\|_{2}+\|\Sigma_{u}^{+}\corrcoef\|_{2}+(2\eta_{1}+\tilde{\sigma}+\eta_{2})\sqrt{\frac{n}{\trace({\Xi_{z}})}}\right),
\end{equation*}
where $\eta_1, \eta_2, \varepsilon \in \R$ are sequences depending on $n$ and $\delta$ satisfying
\begin{align*}
    \eta_{1}&\lesssim\sqrt{\frac{n}{R(\Xi_{z})}}\|\Xi_{z}^{1/2}\Sigma_{u}^{+}\corrcoef\|_{2},  \\
    \eta_{2}&\lesssim\sqrt{\left(1+\sqrt{\frac{2\log(8/\delta)}{r(\Xi_z)}}\right)}\sqrt{\frac{(\mathbb{E}\|\Xi_z^{1/2}H\|_{2})^{2}}{n}\|\Sigma_{u}^{+}\corrcoef\|^{2}_{2}+\|\Xi_{z}^{1/2} \Sigma_{u}^{+}\corrcoef\|^{2}_{2}}, \\
    \interpolatorboundEuclidcoef&:=\sqrt{\log(1/\delta)} \left(\sqrt{\frac{\rank(\Sigma_{u})}{n}}+\left(1+\frac{\trace(\Sigma_{u}\Xi_{z})}{\trace(\Xi_{z}^{2})}\right)\left(\frac{n}{R(\Xi_z)}\right)+\frac{\|\Sigma_{u}^{+}\corrcoef\|_{2}}{\tilde{\sigma}}\sqrt{\frac{\trace(\Xi_{z})}{n}}\right).
\end{align*}
\end{theorem}
The rigorous definitions of $\eta_1, \eta_2,$ and $ \varepsilon$ will be provided in the appendix.
To derive this upper bound, we again use the common uniform upper bound argument. 
Combining this result with Corollary \ref{Variant_Corollary2}, we derive our primary result on the upper bound of $\|\theta-\theta_{0}\|_{\Xi_z}^{2}$.

\section{Experiment}\label{Experiment}

We conduct experiments to justify our theoretical results. 
Specifically, we test whether our derived sufficient conditions in Theorems \ref{Variant_Theorem12} and \ref{Non_ortho_Variant_Theorem12}  lead to benign overfitting.
This section contains two experiments: (i) measuring the projected RMSE of the ridgeless estimator, and (ii) comparing the ridgeless estimator to existing high-dimensional operating variable methods.

\subsection{Projected RMSE of Ridgeless Estimator} \label{sec:experiment1}

\subsubsection{Setups}

We generate $n \in \{200,300,...,1000\}$ independent samples $(X_1,Y_1,Z_1),...,(X_n,Y_n,Z_n)$ from the regression model \eqref{model:reg},
and the covariate $X_i$, noise variable $\xi_i$, and  latent noise $u_i$  follow the distribution
\begin{align}
\begin{pmatrix}
X_{i} \\
\xi_{i}
\end{pmatrix}
&\sim N\left(
\begin{pmatrix}
\mathbf{0}_{p\times 1} \\
0 
\end{pmatrix},
\begin{pmatrix}
\Sigma_{x} & \Sigma_{u}^{1/2}\rho \\
(\Sigma_{u}^{1/2}\rho)^\top  & \sigma^{2}
\end{pmatrix}
\right), 
 \quad u_{i}\sim N(
\mathbf{0}_{p\times 1} ,\Sigma_{u}). \label{dgp_x_xi}
\end{align}
The covariance/coefficient matrices $\Sigma_x, \Sigma_u, \Pi_{0}$, and $\Sigma_{z}$ are determined separately 
for the following four setups. 
Through experiments, truncation level $k^{*}_{n}$ is determined in the same way as (\ref{def:truncation_lebel}).

\begin{enumerate}
    \item[Setup (i)] \textbf{Example \ref{example1} (Orthogonal Case)}: 
    This setting follows Example \ref{example1} with the orthogonal case in Section \ref{sec:orthogonal}.
    We set the parameter dimension as $p=5n$, and set a base matrix $\overline{\Sigma} \in \R^{p \times p}$ such that its $i$-th largest eigenvalue $\lambda_{i}$ is $300 i^{-1}(\log(i+1)\exp/2)^{-2}$ for $i=1,...,p$. 
    We set the true parameter $\theta_{0} \in \R^p$ whose $i$-th element is $20/\sqrt{i}$ and set $\corrcoef$ as $(\Sigma_{u}^{1/2})^{+}\corrcoef:=U\rho$, where $\rho \in \R^p$ has its $i$-th element $2/i$ and $U \in \R^{p \times p}$ is an orthogonalized version of $P \in \R^{p \times p}$ such that $P_{j,j'} = \mone\{ |j - j'| \neq p-2\}$ for $j,j'= 1,...,p$.
Then, we define $\Sigma_{x}, \Sigma_{u}$, and $\Xi_{z}$ as $\Sigma_{x}:=\Sigma_{u}+\Xi_{z}, \Sigma_{u}:=\overline{\Sigma}_{1: k^{*}_{n}}$, and $\Xi_{z}:=(\overline{\Sigma}-\overline{\Sigma}_{1: k^{*}_{n}})$ as in Example \ref{example1}.
This setting satisfies the sufficient conditions in Theorem \ref{Variant_Theorem12}, and also $\Xi_{z}$ and $\Sigma_{u}$ are orthogonal.

    \item[Setup (ii)] \textbf{Example \ref{example2} (Orthogonal Case)}: This setting follows Example \ref{example2} with the orthogonal case in Section \ref{sec:orthogonal}.
    We set the dimension $p=n^{3/2}$ 
    and a base matrix $\overline{\Sigma}$ as its $i$-th eigenvalue $\lambda_i$ being $\lambda_{i}=\gamma_{i}+\varepsilon_{n}$, where $\gamma_{i}=10 \exp(-(i/2))$ and $\varepsilon_{n}=\exp\left(-\sqrt{n}\right)/\sqrt{n}$. 
    We also set the true parameter $\theta_{0} \in \R^p$ whose $i$-th element is $20/\sqrt{i}$, and the correlation coefficient 
    $\corrcoef$ is defined to satisfy $(\Sigma_{u}^{1/2})^{+}\corrcoef:=U\rho$ where $\rho \in \R^p$ has $3\exp(-i/4)$ as its $i$-th element. 
    This setting
    satisfies the sufficient conditions in Theorem \ref{Variant_Theorem12}, and $\Xi_{z}$ and $\Sigma_{u}$ are orthogonal.
    
    \item[Setup (iii)] \textbf{Example 1 (Non-orthogonal Case)}: 
    We consider an extension of Example \ref{example1} to the non-orthogonal case in Section \ref{sec:non-orthogonal}. 
    This is identical to that treated in Example \ref{example_1}. Specifically, $p, \overline{\Sigma}$, and $\corrcoef$ are determined as in Setup (i) above. However, $\Sigma_{u}$ and $\Xi_{z}$ are the same as
    \begin{align}
    &\Sigma_{u}:=\left(1-\frac{1}{n^{1.01}}\right)\overline{\Sigma}_{1: k^{*}_{n}}, \mbox{~and~}\notag \\
    &\Xi_{z}:=(\overline{\Sigma}-\overline{\Sigma}_{1: k^{*}_{n}})+\frac{1}{n^{1.01}}\overline{\Sigma}_{1: k^{*}_{n}}, \label{def:non-ortho_matrix}
\end{align}

and $\Sigma_{x}:=\Sigma_{u}+\Xi_{z}$.
In this setting, $\Xi_{z}$ and $\Sigma_{u}$ are non-orthogonal.
    
    \item[Setup (iv)] \textbf{Example 2 (Non-orthogonal Case)}:
    We consider an extension of Example \ref{example2} to the non-orthogonal case in Section \ref{sec:non-orthogonal}.
      In this setting, $p, \overline{\Sigma}$, and $\corrcoef$ are determined as in Setup (ii) above, and $\Sigma_{u},\Xi_{z},$ and $\Sigma_x$ are set as \eqref{def:non-ortho_matrix}.
Here, $\Xi_{z}$ and $\Sigma_{u}$ are non-orthogonal.

    \item[Setup (v)] \textbf{Example \ref{example1} (Sparse and Orthogonal Case)}: 
   We consider Setup (i) under the sparse setting. The parameters are identical to Setup (i) except the setting of 
   $\theta_{0}$. When $i$ is not more than 100 and there exists a natural number $k$ such that $i+4=5k$, set the $i$-th element of $\theta_{0}$ as $20/\sqrt{i}$. Otherwise, the elements of $\theta_{0}$ are equal to zero.

       \item[Setup (vi)] \textbf{Example \ref{example1} (Sparse and Non-Orthogonal Case)}: 
   We study Setup (iii) in the sparse setting. All the settings are identical to Setup (iii) except the setting of 
   $\theta_{0}$. We impose the sparsity on $\theta_{0}$ as in Setup (v).
\end{enumerate}

In addition, beyond our theoretical framework, we also examine the situation when the variable data are non-Gaussian. 
Specifically, we study the situation where the vector of instrumental variable $Z_i$ follows the multivariate $t$-distribution with $5$ degrees of freedom, and the mean and variance are common.

\subsubsection{Results}

Figure \ref{Experiment_result} summarizes the results of each of the setups. 
The values are means of $50$ repetitions.
The red line shows the projected RMSE of the ridgeless estimator.
The blue lines show the case with the $t$-distribution. 

These results carry several implications:
(a) Despite the increase in dimension $p$ being related to $n$, that is, $p=5n$ or $p=n^{3/2}$, the projected RMSEs converge to zero.
This implies that benign overfitting occurs with this high-dimensional case even with the endogeneity.
(b) The convergence occurs even when $Z_{i}$ is not generated by the Gaussian distribution, which implies that our theoretical results would be applicable to the non-Gaussian case.

\begin{figure}[htbp]
    \centering
    \includegraphics[width=0.55\linewidth]{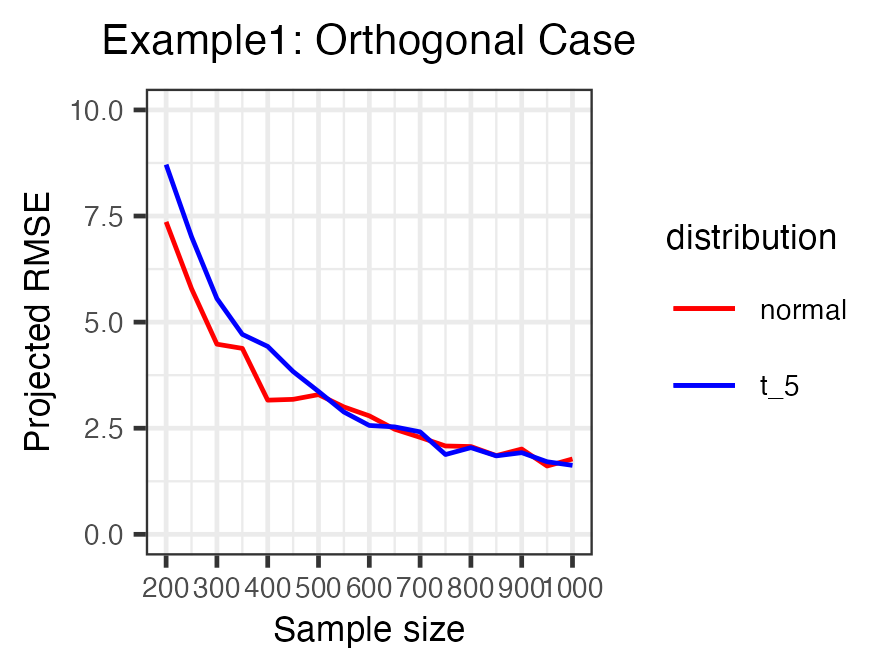}\includegraphics[width=0.55\linewidth]{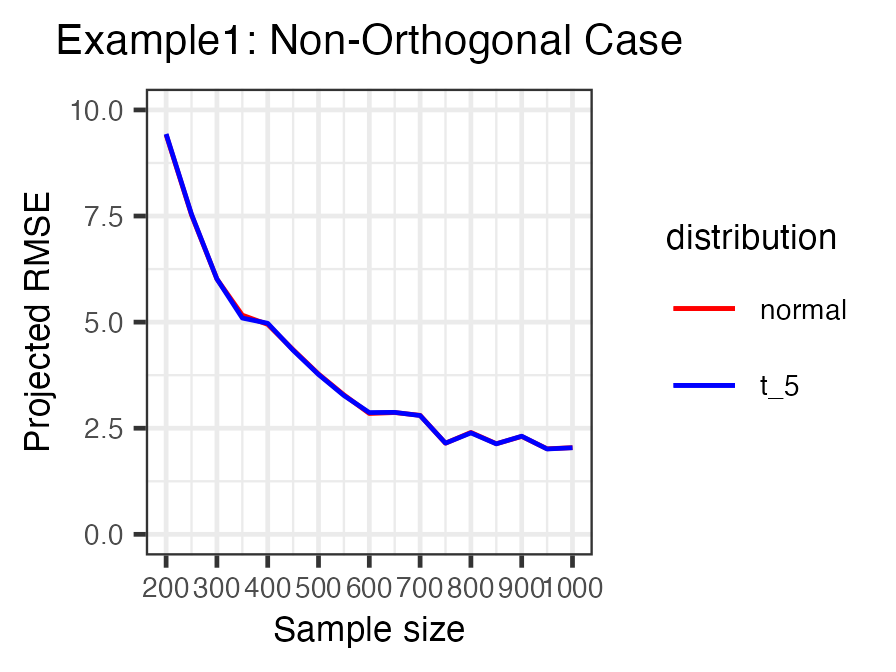}
    \centering
    \includegraphics[width=0.55\linewidth]{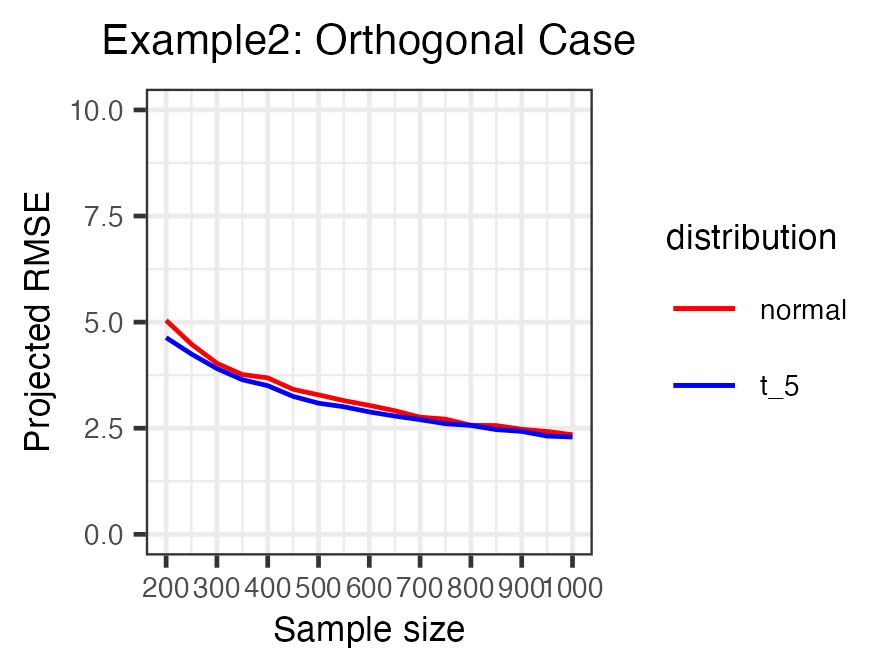}\includegraphics[width=0.55\linewidth]{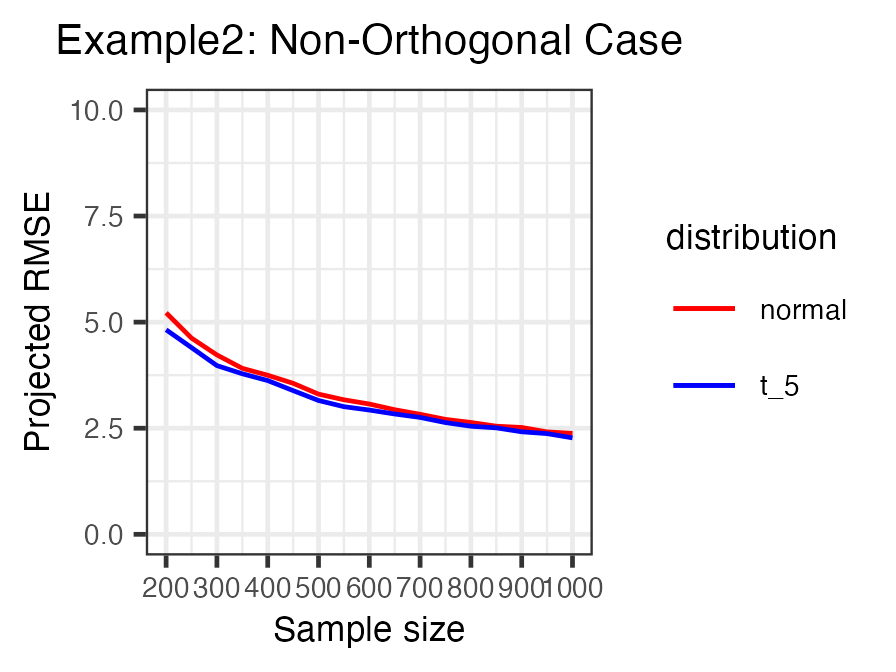}
    \centering
    \includegraphics[width=0.55\linewidth]{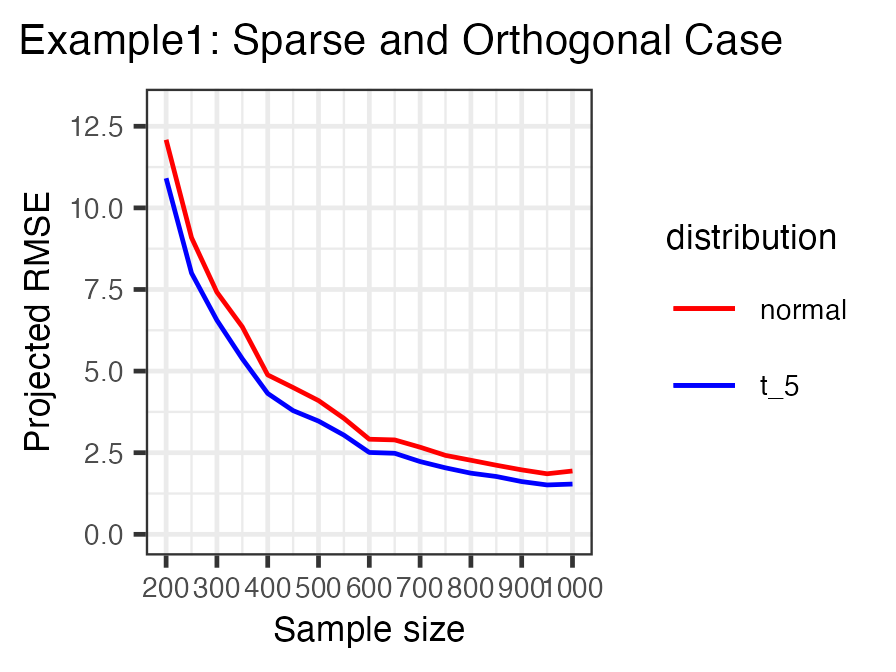}\includegraphics[width=0.55\linewidth]{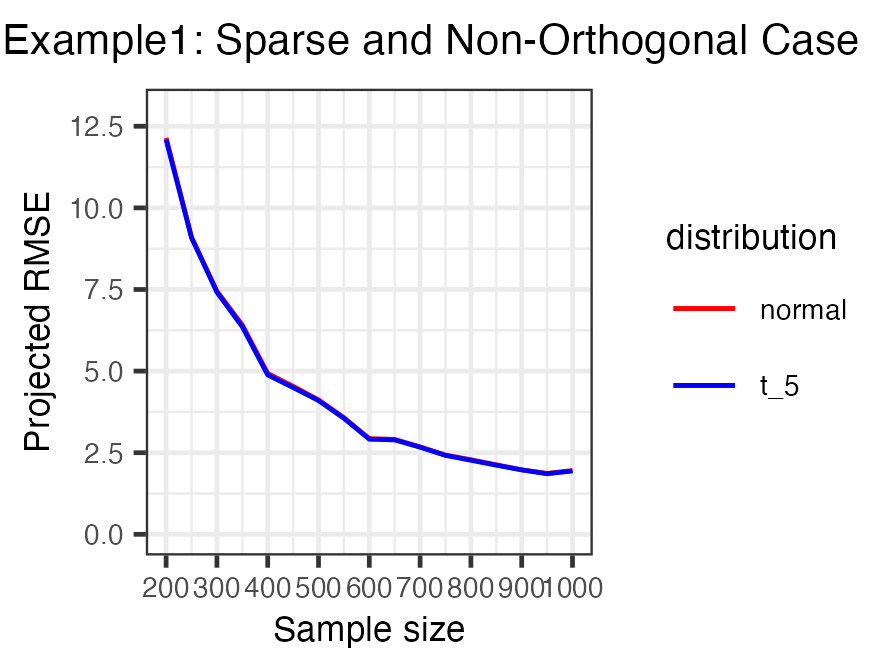}
    \caption{
    Projected RMSE of the ridgeless estimator against the sample size $n$. 
    The red line shows the Gaussian case, and the blue line shows the non-Gaussian case. The dimension $p$ of the parameters is set as $p=5n$ (Example \ref{example1}) or $p = n^{3/2}$ (Example \ref{example2}).
    }
    \label{Experiment_result}
\end{figure}

\subsection{Comparison with Related Method}

\subsubsection{Setups}

We compare the ridgeless estimator to a regularized estimator for high-dimensions, such as the lasso-type method.
Specifically, we consider methods for estimating sparse parameters under high-dimensional covariates and instrumental variables, such as those developed by
\citet{belloni2012sparse,chernozhukov2015post} and many others.

We present our setting.
Similar to Section \ref{sec:experiment1}, we generate $n \in \{100,200,...,1000\}$ observations $(X_1,Y_1,Z_1),...,(X_n,Y_n,Z_n)$ from the regression model \eqref{model:reg} and the data generating process \eqref{dgp_x_xi}.
Here, $k$ denotes a dimension of endogenous variables and we set $k = n/10$.

\begin{enumerate}
    \item[Setup (vii)] \textbf{Non-Sparse Case}: 
    We consider the case where the true parameter $\theta_0$ is not sparse.
    We set $p=5n$ and set the true parameter $\theta_{0} \in \R^p$, which has $20/\sqrt{i}$ as its $i$-th element.
    Further, we set the base matrix $\overline{\Sigma} \in \R^{p \times p}$ that has $\lambda_{i}=300 i^{-1}(\log(i+1)\exp/2)^{-2}$ as its $i$-th largest eigenvalue. The correlation coefficient $\corrcoef \in \R^p$ with its $i$-th element is $(2/i)\mathbbm{1}\{i \in [1,k]\} $.
    With these settings, we define $\Sigma_{x}, \Sigma_{u}$, and $\Xi_{z}$ as in \eqref{def:non-ortho_matrix}.
The matrices $\Sigma_{u}$ and $\Xi_{z}$ with the correlation term $\corrcoef$ satisfy the sufficient conditions in Theorem \ref{Non_ortho_Variant_Theorem12}. 
    
    \item[Setup (viii)] \textbf{Partially Sparse Case}: 
     We consider the case where the true parameter $\theta_0$ is less sparse.
     We set $p=5n$ and define $\theta_{0} \in \R^p$ whose $i$-th element is $(20/\sqrt{i})\mathbbm{1}\{i\leq 0.8n\}$.
     We define $\Sigma_{x}, \Sigma_{u}$, $\Xi_{z}$, and $\corrcoef$ in the same way as the non-sparse case (Setup (v)).

    \item[Setup (ix)] \textbf{Non-Sparse Case (Rotated)}: 
     We set $p=5n$ and define $\Sigma_{x}, \Sigma_{u}$, $\Xi_{z}$, $\theta_{0}$, and $\corrcoef$ in the same way as the non-sparse case (Setup (v)).
    Further, we set the $(k/5+1)$-th to $k$-th variables and the $(k_{n}^{*}+1)$-th to $(k_{n}^{*}+k/5)$-th variables as endogenous.  
\end{enumerate}

For the method to be compared, we utilize the estimator by \citet{chernozhukov2015post} named \textit{LassoIV}.
First, we divide the sample in half, then we use one-half of the sample to estimate the parameters of endogenous variables and use the other to estimate the other parameters. We use the R package $hdm$ \citep{chernozhukov2016high} for implementation. For the estimation of exogenous variables, we subtract the endogenous part from the outcome and define the new outcome $\tilde{Y}_{i}$, that is,
\begin{align*}
    \tilde{Y}_{i}:=Y_{i}-\hat{\beta}^{\top}W_{i},
\end{align*}
where $W_{i}$ is a $k\times 1$ endogenous variable and $\hat{\beta}$ is an estimator by LassoIV. To obtain an estimator for the parameters of exogenous variables, we regress $\tilde{Y}_{i}$ on exogenous variables.

\subsubsection{Results}

Figure \ref{fig:comparison} summarizes the results of each experiment. 
We report means of $30$ repetitions.
When the sample size is small, the projected RMSE by the Lasso method is notably larger than that by the ridgeless estimator. 
As the sample size grows, though the errors get smaller, the error by the ridgeless estimator is still  relatively small.
Specifically, in setup (ix), we change the location of endogenous variables. 
Nevertheless, we can see that the ridgeless estimator provides the smaller projected RMSE.

\begin{figure}[htbp]
    \centering
    \includegraphics[width=\hsize]{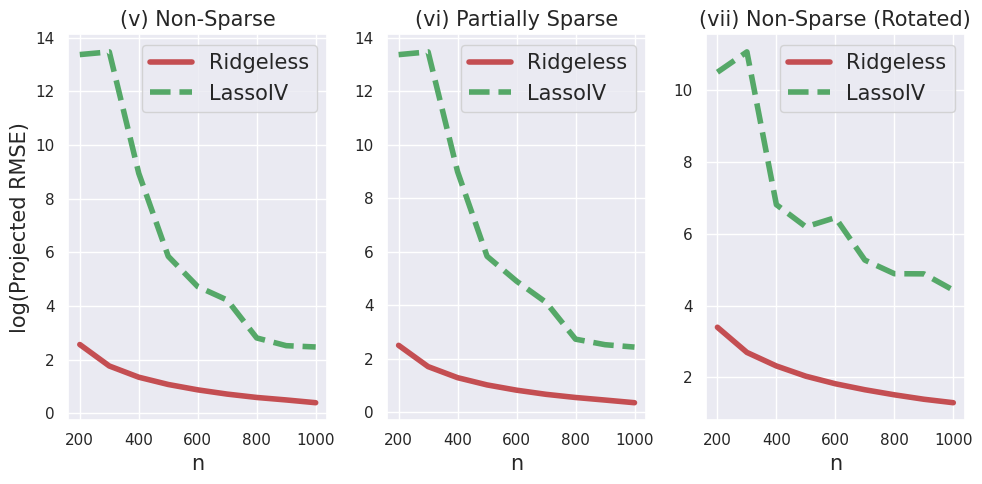}
    \caption{Projected RMSEs of the ridgeless estimator and LassoIV. Each value is a mean of $30$ repetitions.}
    \label{fig:comparison}
\end{figure}

\subsection{Real Data Analysis}

We implement real data analysis in this subsection to exemplify our theoretical result. We used the Current Population Survey (CPS), a monthly survey of U.S. households conducted by the Bureau of the Census of the Bureau of Labor Statistics. Our data consists of the March 2009 survey, including the Asian male individuals who were employed full-time (defined as those who had worked at least 36 hours per week for at least 48 weeks the past year), and excluded those in the military. The sample size is 1,435.

In this analysis, we set the natural log of hourly wage as the outcome variable $Y_{i}$. From the dataset, we use the year of education, the square of the year of education, age, the square of age, and the product of education and age as the covariates. 
Furthermore, to study a high-dimensional setting, we generate the 20,000-dimensional normal variables $X^{*}_{i}$ with the diagonal variance matrix $\Sigma$ whose $\ell$-th diagonal $300 \ell^{-1}(\log(\ell+1)\exp/2)^{-2}$ for $\ell=1,...,20,000$. For each $i\in\{1,\cdots,1435\}$, we have
\begin{equation}\label{applied_model}
Y_{i}=\beta_{1}education_{i}+\beta_{2}education_{i}^{2}+\beta_{3}age_{i}+\beta_{4}age_{i}^{2}+\beta_{5}education_{i}age_{i}+\gamma^{\top} X^{*}_{i}+\xi_{i}.
\end{equation}
As the error term $\xi_{i}$ included the unobserved ability of an individual that will affect both the natural log of hourly wage and the year of education, the year of education
will correlate with the error term $\xi_{i}$, that is, the year of education is endogenous.

Under the setting \eqref{applied_model}, we calculate the sample RMSE. We estimate the interpolator and evaluate the sample RMSE by using 5-fold cross validation. The sample RMSE is 0.6165. As the estimated RMSE obtained from the LASSO estimator with 5-fold cross validation is 0.4463, this result implies the sample RMSE obtained by the interpolator will approximate RMSE even with the presence of  the correlation between the covariates and the noise $\xi_{i}$.

\section{Discussion and Conclusion} \label{sec:discussion}

We studied the estimation error in the over-parameterized linear regression problem when the covariates are endogenous. 
In particular, we examined the situation where data are Gaussian and the covariates have a linear model on an instrumental variable. 
In this setting, we derived sufficient conditions under which the risk of the ridgeless estimator converges to zero.
In other words, we show the ridgeless estimator achieves benign overfitting even in the presence of endogeneity in this setting. 
To show this result, we developed an extended version of CGMT.

An important future challenge for the study of over-parameterization with endogeneity is the development of methods to infer whether our sufficient conditions hold from data. 
This challenge may be addressed, for example, by estimating a decay rate of eigenvalues of covariance matrices, as in the Hill estimator \citep{hill1975simple}. 
The development of such practical methods is an important future task.

One limitation of this study depends on the Gaussianity of data. 
As this is an essential condition for using CGMT, it is not easy to relax. 
However, there has been some research to extending risks with Gaussian data to those of non-Gaussian data, known as universality \citep{han2022universality,montanari2022universality}, so it may be a way to analyze non-Gaussian data.

\appendix

\section{Organization of Appendix} \label{app_sec:organization}

This appendix provides the full proofs of the results in the main body.
The first half of the appendix follows the proof outline described in Section \ref{sec:proof_outline}: (i) a proof of CGMT (Section \ref{app_sec:cgmt}), (ii) a proof of an upper bound for the projected RMSE (Section \ref{app_sec:prmse}), and (iii) a proof of an upper bound for the ridgeless estimator (Section \ref{app_sec:norm_estimator}). 
In Section \ref{app_sec:example_benign}, we provide proofs for the primary statement for benign overfitting. 
In Section \ref{app_sec:non_orthogonal}, we independently present the proof for the non-orthogonal case in Section \ref{sec:non-orthogonal}. 
Finally, supportive results are listed in Section \ref{app_sec:support}.

\section{Proof of CGMT} \label{app_sec:cgmt}

We present a proof of Theorem \ref{VariantCGMT} for CGMT. 
The proof of the standard CGMT is given in \citet{thrampoulidis2015regularized}. 
We extend the standard proof to accommodate partitions of a parameter space.
Remember that $\mathbf{W}:n\times d$ is a matrix with i.i.d. $N(0, 1)$ entries and suppose $G\sim N(0,I_{n})$ and $H\sim N(0,I_{d})$ are independent Gaussian vectors.

\begin{proof}[Proof of Theorem \ref{VariantCGMT}]
The sets $S_{\omega}$ and $S_{u}$ are non-empty, compact, and convex by assumption. As the function $\langle u,\mathbf{W}\omega \rangle +\psi((\omega,\omega'),(u,u'))$ is continuous, finite, and convex-concave on $S_{\omega}\times S_{u}$, it holds from the minimax result in \citet{rockafellar1997convex} (Corollary 37.3.2) that
\begin{equation*}
    \Phi(\mathbf{W})=\max_{(u,u')\in S_{U}}\min_{(\omega,\omega')\in S_{W}}\langle u,\mathbf{W}\omega \rangle +\psi((\omega,\omega'),(u,u')),
\end{equation*}
where we define $\Phi(\mathbf{W})$ as $\min_{(\omega,\omega')\in S_{W}}\max_{(u,u')\in S_{U}}\langle u,\mathbf{W}\omega \rangle +\psi((\omega,\omega'),(u,u'))$.
Consequently, the min-max problem in (\ref{POCGMT}) is replaced with a max-min problem.
This form implies
\begin{equation*}
    -\Phi(\mathbf{W})=\min_{(u,u')\in S_{U}}\max_{(\omega,\omega')\in S_{W}}-\langle u,\mathbf{W}\omega \rangle -\psi((\omega,\omega'),(u,u')).
\end{equation*}
By using the symmetry of $\mathbf{W}$, we obtain that for any $c\in\mathbb{R}$,
\begin{equation*}
    \Pr(-\Phi(\mathbf{W})\leq c)=\Pr \left(\min_{(u,u')\in S_{U}}\max_{(\omega,\omega')\in S_{W}}\{\langle u,\mathbf{W}\omega \rangle -\psi((\omega,\omega'),(u,u'))\}\leq c \right).
\end{equation*}
Then, by a variant of the Gaussian minimax theorem (Theorem 10 of \citet{koehler2022uniform}), we have
\begin{align*}
     &\Pr(-\Phi(\mathbf{W})< c)\\
     &\leq 2 \Pr\left(\min_{(u,u')\in S_{U}}\max_{(\omega,\omega')\in S_{W}}\{\|u\|\langle H, \omega \rangle+\|\omega\|\langle G, u \rangle -\psi((\omega,\omega'),(u,u'))\}\leq c \right) \\
     &=2\Pr \left(\min_{(u,u')\in S_{U}}\max_{(\omega,\omega')\in S_{W}}\{-\|u\|\langle H, \omega \rangle-\|\omega\|\langle G, u \rangle -\psi((\omega,\omega'),(u,u'))\}\leq c \right),
\end{align*}
where the last equation follows because of the symmetry of $H$ and $G$.  Note that we have
\begin{align*}
    &\min_{(u,u')\in S_{U}}\max_{(\omega,\omega')\in S_{W}}\{-\|u\|\langle H, \omega \rangle-\|\omega\|\langle G, u \rangle -\psi((\omega,\omega'),(u,u'))\} \\
  =-&\max_{(u,u')\in S_{U}}\min_{(\omega,\omega')\in S_{W}}\{\|u\|\langle H, \omega \rangle+\|\omega\|\langle G, u \rangle +\psi((\omega,\omega'),(u,u'))\}.
\end{align*}
By  the minimax inequality (\citet{rockafellar1997convex}, Lemma 36.1), we obtain that for all $G,H$,
\begin{align*}
    &\max_{(u,u')\in S_{U}}\min_{(\omega,\omega')\in S_{W}}\{\|\omega\|\langle G, u \rangle+\|u\|\langle H, \omega \rangle+\psi((\omega,\omega'),(u,u'))\} \\
    \leq& \min_{(\omega,\omega')\in S_{W}}\max_{(u,u')\in S_{U}}\{\|\omega\|\langle G, u \rangle+\|u\|\langle H, \omega \rangle+\psi((\omega,\omega'),(u,u'))\}:=\phi(G,H).
\end{align*}
Therefore, we have for any $c\in\mathbb{R}$,
\begin{equation*}
    \Pr(\Phi(\mathbf{W})>-c)=\Pr(-\Phi(\mathbf{W})<c)\leq2\Pr(-\phi(G,H)\leq c)=2\Pr(\phi(G,H)\geq -c).
\end{equation*}

\end{proof}

\section{Upper Bound for Projected Residual Mean Squared Error} \label{app_sec:prmse}

In this section, we provide the upper bound for the projected RMSE. 
Specifically, we prove Corollary \ref{Variant_Corollary2} in the main body, and then give Corollary \ref{Variant_Corollary3}, which generalized a norm.
The objective of this section is to show a general upper bound (Theorem \ref{Variant_Theorem1}). 
To this end, we analyze the projected RMSE by CGMT using Lemmas \ref{Variant_Lemma3} and \ref{Variant_Lemma4}. 
We then analyze the projected RMSE in Lemma \ref{Variant_Lemma5} to show Theorem \ref{Variant_Theorem1},  leading to Corollaries \ref{Variant_Corollary2} and \ref{Variant_Corollary3}.

In the following lemma, we rewrite the projected RMSE \eqref{def:projected_RMSE} in the form of an optimization problem to use CGMT. 
In the statement, we use the empirical squared risk $\hat{L}(\theta)$ in \eqref{def:empirical_risk} and the representation of the data matrix $\mathbf{X}$ in \eqref{eq:X_matrix} and \eqref{def:cov_wwxi}.
As the ridgeless estimator $\hat{\theta}$ satisfies $\hat{L}(\hat{\theta}) = 0$, we are interested in a parameter $\theta$ which satisfies $\hat{L}(\theta) = 0$.

\begin{lemma}\label{Variant_Lemma3}
Let $\mathcal{K}$ denote a compact set in $\mathbb{R}^{p}$. Assume Assumptions \ref{asmp:Gaussianity} and \ref{asmp:orthogonal} hold.
Define the primary optimization problem (PO) as
\begin{equation*}
    \Phi:=\max_{\substack{(\theta_{1},\theta_{2})\in S, \\ \mathbf{W}_{1}\theta_{1}+\mathbf{W}_{2}\theta_{2}=\xi}}\|\theta_{1}\|_{2}^{2}, \tag{\ref{Variant_40}}
\end{equation*}
where we define
    $S:=\{(\theta_{1},\theta_{2}):\exists\theta\in\mathcal{K}\ s.t.\ \theta_{1}=\Xi_z^{1/2}(\theta-\theta_{0})\ and\ \theta_{2}=\Sigma_{u}^{1/2}(\theta-\theta_{0}) )\}$.
Then, the following maximized projected RMSE in \eqref{def:projected_RMSE} is equal in distribution to the PO: 
\begin{align*}
   \max_{\substack{\theta \in \mathcal{K}, \hat{L}(\theta)=0}}E\left[\left(E[\langle {\theta}, X \rangle-\langle\theta_{0}, X \rangle|Z]\right)^{2}\right]&\overset{\mathcal{D}}{=}\Phi.
\end{align*}
\end{lemma}

\begin{proof}[Proof of Lemma \ref{Variant_Lemma3}]
Note that $\hat{L}(\theta)=0$ is equivalent to $\mathbf{Y} = \mathbf{X} \theta$.
By the definitions of $\Xi_{z}$ and $\Sigma_{u}$, we have
\begin{equation*}
     \mathbf{X}\overset{\mathcal{D}}{=}\mathbf{W}_{1}\Xi_z^{1/2}+\mathbf{W}_{2}\Sigma_{u}^{1/2}.
\end{equation*}
Hence, we obtain 
\begin{align*}
    &\max_{\substack{\theta \in \mathcal{K}, \hat{L}(\theta)=0}}E\left[\left(E[\langle {\theta}, X \rangle-\langle\theta_{0}, X \rangle|Z]\right)^{2}\right]\\
    &=\max_{\substack{\theta \in \mathcal{K}, \hat{L}(\theta)=0}}({\theta}-\theta_{0})^\top\Pi_{0}\mathbb{E}[ZZ^\top]\Pi_{0}^\top({\theta}-\theta_{0}) \\
    &=\max_{\theta\in\mathcal{K},\mathbf{X}\theta=\mathbf{Y}}\|\theta-\theta_{0}\|_{\Xi_z}^{2} \\
    &=\max_{\theta\in\mathcal{K},\mathbf{X}(\theta-\theta_{0})=\xi}\|\theta-\theta_{0}\|_{\Xi_z}^{2} \\
    &\overset{\mathcal{D}}{=}\max_{\substack{\theta\in\mathcal{K}-\theta_{0} \\ (\mathbf{W}_{1}\Xi_z^{1/2}+\mathbf{W}_{2}\Sigma_{u}^{1/2})\theta=\xi}}\|\theta\|_{\Xi_z}^{2}.
\end{align*}
By the definition of $S$, we have
\begin{equation*}
   \max_{\substack{\theta\in\mathcal{K}-\theta_{0} \\ (\mathbf{W}_{1}\Xi_z^{1/2}+\mathbf{W}_{2}\Sigma_{u}^{1/2})\theta=\xi}}\|\theta\|_{\Xi_z}^{2}
    =\max_{\substack{(\theta_{1},\theta_{2})\in S \\ \mathbf{W}_{1}\Xi_z^{1/2}\theta_{1}+\mathbf{W}_{2}\Sigma_{u}^{1/2}\theta_{2}=\xi}}\|\theta_{1}\|_{2}^{2}.
\end{equation*}
Then, the stated result holds.
\end{proof}

\begin{lemma}\label{Variant_Lemma4} 
Let $G\sim N(0,I_{n})$, $H\sim N(0,I_{d})$ be Gaussian vectors independent of $\mathbf{W}_{1}, \mathbf{W}_{2},\xi$, and each other. Define the auxiliary optimization problem (AO) as
\begin{equation}\label{Variant_43}
\phi:=\max_{\substack{(\theta_{1},\theta_{2})\in S \\
\|\xi-\mathbf{W}_{2}\theta_{2}-G\|\theta_{1}\|_{2}\|_{2}\leq\langle\theta_{1},H\rangle}} \|\theta_{1}\|_{2}^{2}.
\end{equation}
Then, it holds that
\begin{equation*}
    \Pr(\Phi>t|\mathbf{W}_{2},\xi)\leq 2\Pr(\phi\geq t|\mathbf{W}_{2},\xi).
\end{equation*}
Furthermore, by taking expectations, we obtain
\begin{equation*}
    \Pr(\Phi>t)\leq 2\Pr(\phi\geq t).
\end{equation*}
\end{lemma} 
\begin{proof}[Proof of Lemma \ref{Variant_Lemma4}]

This lemma is quite similar to Lemma 4 in \citet{koehler2022uniform}. The only difference between the two is the objective function of the constrained maximization problems. However, because the objective function (\ref{Variant_43}) does not affect the proof of Lemma 4  in \citet{koehler2022uniform}, the result of Lemma \ref{Variant_Lemma4} also holds.
\end{proof}

We then offer a  bound on the projected RMSE. 
The following lemma is an extension of Lemma 5 in \citet{koehler2022uniform} to the case where the covariates correlate with errors.

As preparation, we define the Gaussian width, which is used in Lemma  \ref{Variant_Lemma5}.
\begin{definition}[Gaussian width  \citep{vershynin2018high}]
The Gaussian width of a set $S\subset \mathbb{R}^{p}$ is
\begin{equation*}
    W(S):=\underset{H \sim N(0,I_{d})}{E}\left[\sup_{s\in S}|\langle s,H \rangle|\right].
\end{equation*}
\end{definition}
\begin{lemma}\label{Variant_Lemma5}
Let $\beta=12\sqrt{\frac{\log(32/\delta)}{n}}+3\sqrt{\frac{\rank(\Sigma_{u})}{n}}$. If $n$ is sufficiently large such that $\beta\leq1$, for every $\delta \in (0,1)$, the following holds with probability at least $1-\delta$:
\begin{align}\label{Variant_48}
    \phi&\leq \frac{1+\beta}{n}\left\{W(\Xi_z^{1/2}\mathcal{K})+\rad(\Xi_z^{1/2}\mathcal{K})\sqrt{2\log(16/\delta)}+\|\theta_{0}\|_{\Xi_z}\sqrt{2\log(16/\delta)}\right\}^{2}-\tilde{\sigma}^{2},
\end{align}
where we define
\begin{align*}
\tilde{\sigma}^{2}&:=\sigma^2 - \|\corrcoef\|_{\Sigma_u^+}^2 =\min_{\theta_{2}\in\Sigma_{u}^{1/2}\mathbb{R}^{p}}\left(\sigma^{2}-\|\rho \|^2+\|\theta_{2}-\rho\|_{2}^{2}\right).
\end{align*}

\end{lemma}

\begin{proof}[Proof of Lemma \ref{Variant_Lemma5}]
Fix $\delta \in (0,1)$ in this proof.
To simplify notations, we define coefficients:
\begin{align*}
\alpha_{1}:=2\sqrt{\frac{\log(32/\delta)}{n}}\quad \mbox{~and~}\quad
\alpha_{2}:=\sqrt{\frac{\rank(\Sigma_{u})+1}{n}}+2\sqrt{\frac{\log(16/\delta)}{n}}. 
\end{align*}

To prepare for the derivation of the upper bound, we consider a list of the following inequalities, each of which holds with probability at least $1-\delta/8$.
\begin{enumerate}
    \item[(i)]  By (\ref{Variant_26}) in Lemma \ref{Variant_Lemma1}, uniformly over all $\theta_{2}\in \Sigma_{u}^{1/2}(\mathcal{K}-\theta_{0})$, 
    it holds that
    \begin{equation}\label{Variant_49}
    |\langle\xi-\mathbf{W}_{2}\theta_{2}, G \rangle|\leq\|\xi-\mathbf{W}_{2}\theta_{2}\|_{2}\|G\|_{2}\alpha_{2}.
\end{equation}
$V$, $s$, and $\delta$ in Lemma \ref{Variant_Lemma1} correspond to $G$, $\xi-\mathbf{W}_{2}\theta_{2}$, and $\delta/8$ in (\ref{Variant_49}), respectively.

\item[(ii)] By Lemma \ref{Variant_Lemma2}, it holds that
\begin{equation}\label{Variant_52}
    -\alpha_{1}\leq\frac{1}{\sqrt{n}}\|G\|_{2}-1\leq\alpha_{1}.
\end{equation}
Moreover, as we obtain the following from \eqref{eq:X_matrix} that 
\begin{equation*}
\begin{pmatrix}
W_{2,i} \\
\xi_{i}
\end{pmatrix}
\sim N\left(
\begin{pmatrix}
\mathbf{0}_{p\times 1} \\
0
\end{pmatrix},
\begin{pmatrix}
 I_{p\times p} & \mathbf{\rho} \\
 \mathbf{\rho}^{T} & \sigma^{2}
\end{pmatrix}
\right),
\end{equation*}
we have $\xi_{i}-W_{2,i}^{T}\theta_{2}\sim N(0, \sigma^{2}-2\rho^{T}\theta_{2}+\|\theta_{2}\|^2)$. As $\{\xi_{i}-W_{2,i}^{T}\theta_{2}\}_{i=1}^{n}$ are i.i.d., we have
\[
\xi-\mathbf{W}_{2}\theta_{2} \overset{\mathcal{D}}{=}\sqrt{\sigma^{2}-2\rho^{T}\theta_{2}+\|\theta_{2}\|^2}G.
\]
Further, by Lemma \ref{Variant_Lemma2}, we have
\begin{align}
    -\alpha_{1}\sqrt{\sigma^{2}-2\rho^{T}\theta_{2}+\|\theta_{2}\|^2}&\leq\frac{1}{\sqrt{n}}\|\xi-\mathbf{W}_{2}\theta_{2}\|_{2}-\sqrt{\sigma^{2}-2\rho^{T}\theta_{2}+\|\theta_{2}\|^2} \notag  \\
    &\leq\alpha_{1}\sqrt{\sigma^{2}-2\rho^{T}\theta_{2}+\|\theta_{2}\|^2}. \label{Variant_53}
\end{align}

\item[(iii)] By the standard Gaussian tail bound $\Pr(|Z|\geq t)\leq2e^{-t^{2}/2}$, 
it holds that
\begin{equation}\label{Variant_54}
    |\langle \Xi_z^{1/2}\theta_{0}, H\rangle| \overset{\mathcal{D}}{=} |Z'| \leq\|\theta_{0}\|_{\Xi_z}\sqrt{2\log(16/\delta)},
\end{equation}
where $Z' \sim N(0,\|\theta_{0}\|^{2}_{\Xi_z})$.

\item[(iv)] By Theorem \ref{Variant_Theorem6}, it holds that
\begin{equation}\label{Variant_55}
    \max_{\theta_{1}\in\Xi_z^{1/2}\mathcal{K}}|\langle \theta_{1},H\rangle|\leq W(\Xi_z^{1/2}\mathcal{K})+\rad(\Xi_z^{1/2}\mathcal{K})\sqrt{2\log(16/\delta)}
\end{equation}
because $\max_{\theta_{1}\in\Xi_z^{1/2}\mathcal{K}}|\langle \theta_{1}, H\rangle|$ is a $\rad(\Xi_{z}^{1/2}\mathcal{K})$-Lipschitz function of $H$, and $W(\Xi_z^{1/2} \mK) = \Ep[\sup_{\theta_{1}\in\Xi_z^{1/2}\mathcal{K}}|\langle \theta_{1}, H\rangle|]$.
\end{enumerate}

We further prepare several inequalities.
By squaring the last constraint in the definition of the auxiliary optimization problem $\phi$, we see that
\begin{align*}
    \langle \theta_{1}, H \rangle^{2}&\geq \|\xi-\mathbf{W}_{2}\theta_{2}-\|\theta_{1}\|_{2}G\|^{2} \\
    &=\|\xi-\mathbf{W}_{2}\theta_{2}\|_{2}^{2}+\|\theta_{1}\|_{2}^{2}\|G\|^{2}_{2}-2\langle \xi-\mathbf{W}_{2}\theta_{2} , \|\theta_{1}\|_{2}G\rangle.
\end{align*}
From (\ref{Variant_49}) and the AM-GM inequality ($a^{2}/2+b^{2}/2\geq ab$), we have
\begin{equation*}
    \langle \theta_{1}, H \rangle^{2}\geq (1-\alpha_{2})[\|\xi-\mathbf{W}_{2}\theta_{2}\|_{2}^{2}+\|\theta_{1}\|_{2}^{2}\|G\|^{2}_{2}].
\end{equation*}
From the rearrangement of the above inequality, we have
\begin{align}
    \|\theta_{1}\|_{2}^{2}&\leq \frac{(1-\alpha_{2})^{-1}\langle \theta_{1}, H \rangle^{2}-\|\xi-\mathbf{W}_{2}\theta_{2}\|_{2}^{2}}{\|G\|_{2}^{2}} \notag \\
    &\leq \frac{(1-\alpha_{2})^{-1}\langle \theta_{1}, H \rangle^{2}-\|\xi-\mathbf{W}_{2}\theta_{2}\|_{2}^{2}}{n(1-\alpha_{1})^{2}} \notag \\
    &\leq \frac{(1-\alpha_{2})^{-1}\langle \theta_{1}, H \rangle^{2}-n(1-\alpha_{1})^{2}(\|\theta_{2}\|^2-2\rho^{T}\theta_{2}+\sigma^{2})}{n(1-\alpha_{1})^{2}}, \label{5-inequaltiy}
\end{align}
where the second inequality holds from (\ref{Variant_52}) and the third inequality holds from  (\ref{Variant_53}).

Now, we are ready to construct the upper bound on $\phi$ from the restriction of the optimization problem \eqref{Variant_43}. 
Plugging (\ref{5-inequaltiy}) into (\ref{Variant_43}), we obtain 
\begin{align*}
    \phi&\leq
    \max_{\theta_{1}\in\Xi_z^{1/2}(\mathcal{K}-\theta_{0})}\frac{(1-\alpha_{2})^{-1}\langle \theta_{1}, H \rangle^{2}}{n(1-\alpha_{1})^{2}}+\max_{\theta_{2}\in\Sigma_{u}^{1/2}(\mathcal{K}-\theta_{0})} -(\|\theta_{2}\|^2-2\rho^{T}\theta_{2}+\sigma^{2}) \\
     &\leq \frac{1}{n(1-\alpha_{2})(1-\alpha_{1})^{2}}\left(\max_{\theta_{1}\in\Xi_z^{1/2}\mathcal{K}}|\langle \theta_{1}, H \rangle|+|\langle\Xi_z^{1/2}\theta_{0},H\rangle|\right)^{2} \\
     & \quad -\min_{\theta_{2}\in\Sigma_{u}^{1/2}(\mathcal{K}-\theta_{0})}\left(\sigma^{2}-\|\rho \|^2+\|\theta_{2}-\rho\|_{2}^{2}\right) \\
     &\leq \frac{1}{n(1-\alpha_{2})(1-\alpha_{1})^{2}}\left(\max_{\theta_{1}\in\Xi_z^{1/2}\mathcal{K}}|\langle \theta_{1}, H \rangle|+|\langle\Xi_z^{1/2}\theta_{0},H\rangle|\right)^{2} \\
     & \quad -\min_{\theta_{2}\in\Sigma_{u}^{1/2}\mathbb{R}^{p}}\left(\sigma^{2}-\|\rho \|^2+\|\theta_{2}-\rho\|_{2}^{2}\right) \\
     &\leq \frac{1}{n(1-\alpha_{2})(1-\alpha)^{2}}\left(W(\Xi_z^{1/2}\mathcal{K})+\rad(\Xi_z^{1/2}\mathcal{K})\sqrt{2\log(16/\delta)}+\|\theta_{0}\|_{\Xi_z}^{2}\sqrt{2\log(16/\delta)}\right)^{2}-\tilde{\sigma}^{2},
\end{align*}
where the fourth inequality holds from  (\ref{Variant_54}) and (\ref{Variant_55}).

We simplify the effect of $\alpha_{1}$ and $\alpha_{2}$ on the upper bound for $\phi$. As $(1-\alpha_{1})^{2}\geq1-2\alpha_{1}$, we have
\begin{equation*}
    \frac{1}{(1-\alpha_{2})(1-\alpha_{1})^{2}}\leq\frac{1}{(1-\alpha_{2})(1-2\alpha_{1})}.
\end{equation*}
If $\alpha_{1}<1/2 and \alpha_{2}<1$,
\begin{align*}
    (1-2\alpha_{1})(1-\alpha_{2})&=1-\alpha_{2}-2\alpha_{1}+2\alpha_{1}\alpha_{2} \\
    &\geq1-\alpha_{2}-2\alpha_{1}.
\end{align*}
Assume $\alpha_{2}+2\alpha_{1}<1/2$. By using the inequality $(1-x)^{-1}\leq1+2x$ for $x\in[0,1/2]$, we can show that
\begin{equation*}
     \frac{1}{(1-\alpha_{2})(1-\alpha_{1})^{2}}\leq\frac{1}{(1-\alpha_{2})(1-2\alpha_{1})}\leq1+2\alpha_{2}+4\alpha_{1}.
\end{equation*}
Therefore, if we choose $\beta$ to satisfy the following inequality:
\begin{equation*}
    2\alpha_{2}+4\alpha_{1}\leq 3\sqrt{\frac{\rank(\Sigma_{u})}{n}}+12\sqrt{\frac{\log(32/\delta)}{n}}:=\beta,
\end{equation*}
the stated result holds.
\end{proof}

Finally, we obtain the generalization bound from Lemma \ref{Variant_Lemma5}.
\noindent
\begin{theorem}[General Bound]\label{Variant_Theorem1}
There exists an absolute constant $C_{1}\leq24$ such that the following is true. Assume Assumptions \ref{asmp:Gaussianity} and \ref{asmp:orthogonal} hold. Let $\mathcal{K}$ denote an arbitrary compact set, and take $\Sigma_x=\Xi_z + \Sigma_{u}$. Fixing $\delta\leq1/4$, let $\beta=C_{1}\left(\sqrt{\rank(\Sigma_{u})/n}+\sqrt{\log(1/\delta)/n}\right)$.
If $n$ is large enough that $\beta\leq1$, then the following holds with probability at least $1-\delta$:
\begin{align*}
    &\max_{\theta\in\mathcal{K}, Y=X\theta}\|\theta-\theta_{0}\|_{\Xi_z}^{2}\\
    &\leq \frac{1+\beta}{n}\left[W(\Xi_z^{1/2}\mathcal{K})+\left(\rad(\Xi_z^{1/2}\mathcal{K})+\|\theta_{0}\|_{\Xi_z}\right)\sqrt{2\log\frac{32}{\delta}}\right]^{2}-\tilde{\sigma}^{2}.
\end{align*}
\end{theorem}

\begin{proof}[Proof of Theorem \ref{Variant_Theorem1}]
For any $t>0$, it holds from Lemmas \ref{Variant_Lemma3} and \ref{Variant_Lemma4} that
\begin{equation*}
    \Pr\left(\max_{\theta\in\mathcal{K}, Y=X\theta}\|\theta-\theta_{0}\|_{\Xi_z}^{2}>t\right)\leq2\Pr(\phi\geq t).
\end{equation*}
Lemma \ref{Variant_Lemma5} implies that the above term is upper bounded by $\delta$ if we choose $t$ using the result (\ref{Variant_48}) with $\delta$ replaced by $\delta/2$. Then, we obtain the stated result.
\end{proof}

By using the definition of the radius of sets and the Gaussian width, we can reduce the generalization bound in Theorem \ref{Variant_Theorem1} to a simpler bound:
\begin{corollary}\label{Variant_Corollary3}
There exists an absolute constant $C_{1}\leq32$ such that the following is true. Assume Assumptions \ref{asmp:Gaussianity} and \ref{asmp:orthogonal} hold. Pick $\Sigma_x=\Xi_z + \Sigma_{u}$, fix $\delta\leq 1/4$, and let $\gamma=C_{1}(\sqrt{\log(1/\delta)/r_{\|\cdot\|}(\Xi_z)}+\sqrt{\log(1/\delta)/n}+\sqrt{\rank(\Sigma_{u})/n})$. If $B\geq\|\theta_{0}\|$ and $n$ is large enough that $\gamma\leq1$, the following holds with probability at least $1-\delta$:
\begin{align*}
    \max_{\|\theta\|\leq B,  \mathbf{Y}=\mathbf{X}\theta}\|\theta-\theta_{0}\|_{\Xi_z}^{2}&\leq (1+\gamma)\frac{\left(B\mathbb{E}\|\Sigma_{2}^{1/2}H\|_{*}\right)^{2}}{n}-\tilde{\sigma}^{2}.
\end{align*}
\end{corollary}
\begin{proof}[Proof of Corollary \ref{Variant_Corollary3}]

Let $\mathcal{K}$ define $\{\theta:\|\theta\|\leq B\}$ in Theorem \ref{Variant_Theorem1}. By the definition of the Gaussian width and the radius of a set, we have
\begin{align*}
     W(\Xi_z^{1/2}\mathcal{K})&=\mathbb{E}\sup_{\|\theta\|\leq B}|\langle\Xi_z^{1/2}\theta, H\rangle|=\mathbb{E}\sup_{\|\theta\|\leq B}|\langle\theta,\Xi_z^{1/2} H\rangle|=B\mathbb{E}\|\Xi_z^{1/2}H\|_{*}, \\
     \rad(\Xi_z^{1/2}\mathcal{K})&=\sup_{\|\theta\|\leq B}\|\Xi_z^{1/2}\theta\|_{2}=B\sup_{\|\theta\|\leq 1}\|\theta\|_{\Xi_{z}}.
\end{align*}
Hence, we obtain
\begin{equation*}
    r_{\|\cdot\|}(\Xi_{z})=\left(\frac{W(\Xi_z^{1/2}\mathcal{K})}{\mathrm{rad}(\Xi_z^{1/2}\mathcal{K})}\right)^{2}.
\end{equation*}
As we have 
$\|\theta_{0}\|_{\Xi_{z}}=\sqrt{(\theta_{0}^{\top}\Xi_{z}\theta_{0}/\|\theta_{0}\|^{2})\cdot\|\theta_{0}\|^{2}}$,
it holds that  $\|\theta_{0}\|_{\Xi_{z}}\leq \|\theta_{0}\|\sup_{\|\theta\|\leq 1}\|\theta\|_{\Xi_{z}}$. By definition, it is clear that $\|\theta_{0}\|_{\Xi_{z}}\leq \mathrm{rad}(\Xi_{z}^{1/2}\mathcal{K})$. Hence,
\begin{align*}
    &W(\Xi_z^{1/2}\mathcal{K})+\left(\rad(\Xi_z^{1/2}\mathcal{K})+\|\theta_{0}\|_{\Xi_z}\right)\sqrt{2\log\frac{32}{\delta}}\\
    &\leq W(\Xi_z^{1/2}\mathcal{K})+2\sqrt{2\log\frac{32}{\delta}}\rad(\Xi_z^{1/2}\mathcal{K}) \\
    &= W(\Xi_z^{1/2}\mathcal{K})+2\sqrt{\frac{2\log(32/\delta)}{r_{\|\cdot\|}(\Xi_{z})}} W(\Xi_z^{1/2}\mathcal{K}) \\
    &= \left(1+2\sqrt{\frac{2\log(32/\delta)}{r_{\|\cdot\|}(\Xi_z)}}\right)B\mathbb{E}\|\Xi_z^{1/2}H\|_{*}
\end{align*}
holds where the last equality holds by the definition of $r_{\|\cdot\|}(\Xi_{z})$. Provided that $\gamma\leq1$ and $\delta\leq1/4$, we obtain
\begin{align*}
&(1+\beta)\left(1+2\sqrt{\frac{2\log(32/\delta)}{r_{\|\cdot\|}(\Xi_z)}}\right)^{2}\\
&\leq\left(1+\beta+4\sqrt{\frac{2\log(32/\delta)}{r_{\|\cdot\|}(\Xi_z)}}\right)\left(1+2\sqrt{\frac{2\log(32/\delta)}{r_{\|\cdot\|}(\Xi_z)}}\right) \\
&\leq 1+\gamma,
\end{align*}
where the inequalities follow from using $(1+x)(1+y) \leq1+x+2y$ for $x\leq1$. 
Plugging into Theorem \ref{Variant_Theorem1} completes the proof.
\end{proof}

When we consider the Euclidean space, we can reduce the main generalization bound to a simpler bound.

\vspace{1\baselineskip}

\noindent
\textbf{Corollary \ref{Variant_Corollary2}} \textit{
There exists an absolute constant $C_{1}\leq32$ such that the following is true. Assume Assumptions \ref{asmp:Gaussianity} and \ref{asmp:orthogonal} hold. Pick $\Sigma_x=\Xi_z + \Sigma_{u}$, fix $\delta\leq 1/4$, and let $\gamma=C_{1}\left(\sqrt{\log(1/\delta)/r(\Xi_z)}+\sqrt{\log(1/\delta)/n}+\sqrt{\rank(\Sigma_{u})/n}\right)$. If $B\geq\|\theta_{0}\|_{2}$ and $n$ is large enough that $\gamma\leq1$, the following holds with probability at least $1-\delta$:
\begin{align*}
    \max_{\|\theta\|_{2}\leq B,  \mathbf{Y}=\mathbf{X}\theta}\|\theta-\theta_{0}\|_{\Xi_z}^{2}&\leq (1+\gamma)\frac{B^{2}\trace(\Xi_z)}{n}-\tilde{\sigma}^{2}.
\end{align*}
}

\begin{proof}[Proof of Corollary \ref{Variant_Corollary2}]

By trivial calculation, we have
\begin{equation*}
    W(\Xi_z^{1/2}\mathcal{K})\leq B\trace(\Xi_z)^{1/2}\ \ \textit{and}\ \ \rad(\Xi_z^{1/2}\mathcal{K})=B\|\Xi_z\|_{\mathrm{op}}^{1/2}.
\end{equation*}
By the definition of $\rad(\Xi_z^{1/2}\mathcal{K})$, we have $\|\theta_{0}\|_{\Xi_z}\leq \rad(\Xi_z^{1/2}\mathcal{K})=B\|\Xi_z\|_{op}^{1/2}$. Hence,
\begin{align*}
    &W(\Xi_z^{1/2}\mathcal{K})+\left(\rad(\Xi_z^{1/2}\mathcal{K})+\|\theta_{0}\|_{\Xi_z}\right)\sqrt{2\log\frac{32}{\delta}} \\
    &\leq W(\Xi_z^{1/2}\mathcal{K})+2\sqrt{2\log\frac{32}{\delta}}\rad(\Xi_z^{1/2}\mathcal{K}) \\
    &\leq B\trace(\Xi_z)^{1/2}+2\sqrt{2\log\frac{32}{\delta}}B\|\Xi_z\|^{1/2}_{\mathrm{op}} \\
    &= \left(1+2\sqrt{\frac{2\log(32/\delta)}{r(\Xi_z)}}\right)B\trace(\Xi_z)^{1/2}
\end{align*}
holds. 
The last equality holds by the definition of the effective rank $r(\Xi_{z})$ in Definition \ref{def:effrank}. 
Under our assumptions that $\gamma\leq1$ and $\delta\leq1/4$, we can show that 
\begin{align*}
&(1+\beta)\left(1+2\sqrt{\frac{2\log(32/\delta)}{r(\Xi_z)}}\right)^{2}\\
&\leq\left(1+\beta+4\sqrt{\frac{2\log(32/\delta)}{r(\Xi_z)}}\right)\left(1+2\sqrt{\frac{2\log(32/\delta)}{r(\Xi_z)}}\right)\\
&\leq 1+\gamma,
\end{align*}
where the inequality follows from using  $(1+x)(1+y) \leq1+x+2y$ for $x\leq1$ and $y \geq 0$. 
Plugging into Theorem \ref{Variant_Theorem1} completes the proof.
\end{proof}

\section{Bounds for the Ridgeless Estimator} \label{app_sec:norm_estimator}

In this section, we provide an upper bound of a norm of the ridgeless estimator with the existence of a correlation between the covariates and the error terms. 
In Lemmas \ref{Variant_Lemma6} and \ref{Variant_Lemma7}, 
we rewrite the norm of the estimator to apply CGMT. 
Lemma \ref{Variant_Lemma8} bounds an element in the rewritten form of the norm.
Then, Theorem \ref{Variant_Theorem4} develops the desired bound on the norm, and Theorem \ref{Variant_Theorem2} offers its Euclidean norm case.

First, we formulate the constrained minimization problem with Gaussian covariates.

\begin{lemma}\label{Variant_Lemma6}
Assume Assumptions \ref{asmp:Gaussianity} and \ref{asmp:orthogonal} hold. Let $\|\cdot\|$ denote an arbitrary norm. Define the primary optimization problem (PO) as
\begin{equation*}
    \Phi:=\min_{\mathbf{W}_{1}\theta_{1}+\mathbf{W}_{2}\theta_{2}=\xi} \|\Sigma_x^{-1/2}(\theta_{1}+\theta_{2})\|,
\end{equation*}
where $\theta_{1}=\Xi_{z}^{1/2}\theta$ and $\theta_{2}=\Sigma_{u}^{1/2}\theta$ for $\theta\in\mathbb{R}^{p}$. 
Then, for any $t$, it holds that
\begin{equation*}
    \Pr\left(\min_{\mathbf{X}\theta=\mathbf{Y}}\|\theta\|>t\right)\leq
    \Pr\left(\|\theta_{0}\|+\Phi>t\right).
\end{equation*}
\end{lemma}

\begin{proof}[Proof of Lemma \ref{Variant_Lemma6}]

We have $\mathbf{X}\overset{D}{=}\mathbf{W}_{1}\Xi_z^{1/2}+\mathbf{W}_{2}\Sigma_{u}^{1/2}$ by equality in distribution. It follows from the triangle inequality and change of variables that
\begin{align*}
    \min_{\mathbf{X}\theta=\mathbf{Y}}\|\theta\|&=\min_{\mathbf{X}\theta=\xi}\|\theta+\theta_{0}\| \leq \|\theta_{0}\|+\min_{(\mathbf{W}_{1}\Xi_z^{1/2}+\mathbf{W}_{2}\Sigma_{u}^{1/2})\theta=\xi}\|\theta\|.
\end{align*}
As $\Sigma_x^{1/2}\theta=\Xi_z^{1/2}\theta+\Sigma_{u}^{1/2}\theta$, we have $\theta=\Sigma_x^{-1/2}(\theta_{1}+\theta_{2})$ where $\theta_{1}=\Xi_z^{1/2}\theta$ and $\theta_{2}=\Sigma_{u}^{1/2}\theta$. Then, the following inequality holds:
\begin{equation*}
    \min_{X\theta=Y}\|\theta\|\leq \|\theta_{0}\|+\min_{\mathbf{W}_{1}\theta_{1}+\mathbf{W}_{2}\theta_{2}=\xi}\|\Sigma_x^{-1/2}(\theta_{1}+\theta_{2})\|. 
\end{equation*}
\end{proof}

As in Lemma \ref{Variant_Lemma4}, we use the result of Theorem \ref{VariantCGMT} to derive the auxiliary optimization problem.
\begin{lemma}\label{Variant_Lemma7}
In the same setting as Lemma \ref{Variant_Lemma6}, let $G\sim N(0,I_{n})$ and $H\sim N(0,I_{d})$ be Gaussian vectors independent of $\xi,\mathbf{W}_{1},\mathbf{W}_{2}$, and each other. Define the auxiliary optimization problem (AO) as
\begin{equation}\label{Variant_58}
    \phi:=\min_{\|\xi-\mathbf{W}_{2}\theta_{2}-\|\theta_{1}\|_{2}G\|_{2}\leq\langle H,\theta_{1} \rangle} \|\Sigma_x^{-1/2}(\theta_{1}+\theta_{2})\|.
\end{equation}
Then, it holds that
\begin{equation*}
    \Pr(\Phi>t|\xi,\mathbf{W}_{2})\leq2\Pr(\phi\geq t|\xi,\mathbf{W}_{2}),
\end{equation*}
and taking the expectations we have
\begin{equation*}
    \Pr(\Phi>t)\leq2\Pr(\phi\geq t).
\end{equation*}
\end{lemma}

\begin{proof}[Proof of Lemma \ref{Variant_Lemma7}]

We reformulate $\Phi$ to apply the extended CGMT (Theorem \ref{VariantCGMT}).
By using Lagrangian multipliers, it holds that
\begin{align*}
    \Phi&=\min_{\theta_{1},\theta_{2}}\max_{\lambda}\|\Sigma_x^{-1/2}(\theta_{1}+\theta_{2})\|+\langle\lambda,\mathbf{W}_{1}\theta_{1}+\mathbf{W}_{2}\theta_{2}-\xi\rangle \\
    &=\min_{\theta_{1},\theta_{2}}\max_{\lambda}\langle\lambda,\mathbf{W}_{1}\theta_{1}\rangle+\|\Sigma_x^{-1/2}(\theta_{1}+\theta_{2})\|-\langle\lambda,\xi-\mathbf{W}_{2}\theta_{2}\rangle.
\end{align*}
As $\mathbf{W}_{1}$ is independent of $\mathbf{W}_{2}$ and $\xi$, the distribution of $\mathbf{W}_{1}$ is unchanged even though we condition on $\mathbf{W}_{2}$ and $\xi$. 
For any $r,t>0$, we define
\begin{equation*}
    \Phi_{r}(t):=\min_{\|\Sigma_x^{-1/2}(\theta_{1}+\theta_{2})\|\leq 2t}\max_{\|\lambda\|_{2}\leq r}\langle\lambda,\mathbf{W}_{1}\theta_{1}\rangle+\|\Sigma_x^{-1/2}(\theta_{1}+\theta_{2})\|-\langle\lambda,\xi-\mathbf{W}_{2}\theta_{2}\rangle. 
\end{equation*}
The corresponding AO is defined as follows:
\begin{align*}
    &\phi_{r}(t)\\
    &:=\min_{\|\Sigma_x^{-1/2}(\theta_{1}+\theta_{2})\|\leq 2t}\max_{\|\lambda\|_{2}\leq r}\|\theta_{1}\|_{2}\langle G,\lambda\rangle +\|\lambda\|_{2}\langle H,\theta_{1}\rangle+\|\Sigma_x^{-1/2}(\theta_{1}+\theta_{2})\|-\langle\lambda,\xi-\mathbf{W}_{2}\theta_{2}\rangle \\
    &=\min_{\|\Sigma_x^{-1/2}(\theta_{1}+\theta_{2})\|\leq 2t}\max_{\|\lambda\|_{2}\leq r}\|\lambda\|_{2}\langle H,\theta_{1}\rangle-\langle\lambda,\xi-\mathbf{W}_{2}\theta_{2}-G\|\theta_{1}\|_{2}\rangle+\|\Sigma_x^{-1/2}(\theta_{1}+\theta_{2})\|  \\
    &=\min_{\|\Sigma_x^{-1/2}(\theta_{1}+\theta_{2})\|\leq 2t}\max_{0\leq\lambda\leq r}\lambda(\langle H,\theta_{1}\rangle+\|\xi-\mathbf{W}_{2}\theta_{2}-G\|\theta_{1}\|_{2}\|_{2})+\|\Sigma_x^{-1/2}(\theta_{1}+\theta_{2})\|.
\end{align*}
As two optimization problems $\Phi_{r}(t)$ and $\phi_{r}(t)$ are defined on compact sets, we can apply Theorem \ref{VariantCGMT} to those two optimization problems. As an intermediate problem between $\Phi$ and $\Phi_{r}(t)$, we introduce 
\begin{align*}
    \Phi(t)&:=\min_{\|\Sigma_x^{-1/2}(\theta_{1}+\theta_{2})\|\leq 2t}\max_{\lambda}\langle\lambda,\mathbf{W}_{1}\theta_{1}\rangle+\|\Sigma_x^{-1/2}(\theta_{1}+\theta_{2})\|-\langle\lambda,\xi-\mathbf{W}_{2}\theta_{2}\rangle \\
    &=\min_{\substack{\mathbf{W}_{1}\theta_{1}+\mathbf{W}_{2}\theta_{2}=\xi \\
\|\Sigma_x^{-1/2}(\theta_{1}+\theta_{2})\|\leq 2t}}\|\Sigma_x^{-1/2}(\theta_{1}+\theta_{2})\|
\end{align*}
and also define the corresponding AO as 
\begin{align*}
    &\phi(t)\\
    &:=\min_{\|\Sigma_x^{-1/2}(\theta_{1}+\theta_{2})\|\leq 2t}\max_{\lambda\geq 0}\lambda(\langle H,\theta_{1}\rangle+\|\xi-\mathbf{W}_{2}\theta_{2}-G\|\theta_{1}\|_{2}\|_{2})+\|\Sigma_x^{-1/2}(\theta_{1}+\theta_{2})\| \\
    &=\min_{\substack{ \|\xi-\mathbf{W}_{2}\theta_{2}-G\|\theta_{1}\|_{2}\|_{2}\leq  \langle H,\theta_{1}\rangle \\
    \|\Sigma_x^{-1/2}(\theta_{1}+\theta_{2})\|\leq 2t 
    }}\|\Sigma_x^{-1/2}(\theta_{1}+\theta_{2})\|_{2}.
\end{align*}
By definition, clearly, $\Phi\leq \Phi(t)$. Therefore, if $\Phi>t$, $\Phi(t)>t$ holds. If $t\geq \Phi$, then there exists $(\theta_{1}^{*},\theta_{2}^{*})$ such that $\|\Sigma_x^{-1/2}(\theta_{1}^{*}+\theta_{2}^{*})\|\leq t$ and $\mathbf{W}_{1}\theta_{1}^{*}+\mathbf{W}_{2}\theta_{2}^{*}=\xi$. As $\|\Sigma_x^{-1/2}(\theta_{1}^{*}+\theta_{2}^{*})\|\leq 2t$, we obtain
\begin{equation*}
    \Phi(t)\leq \|\Sigma_x^{-1/2}(\theta_{1}^{*}+\theta_{2}^{*})\|\leq t.
\end{equation*}
Therefore, it holds that
\begin{equation*}
    \Phi>t \Leftrightarrow \Phi(t)>t.
\end{equation*}
Likewise, $\phi(t)>t$ is equivalent to $\phi>t$.

To establish the result $\Pr(\Phi>t)\leq 2\pr(\phi>t)$, we need to clarify the relationship between $\Phi$ and $\Phi_{r}(t)$, $\phi$ and $\phi_{r}(t)$, respectively, that is,
\begin{equation*}
    Pr(\Phi>t|\xi,\mathbf{W}_{2})\leq\lim_{r\rightarrow\infty} \Pr(\Phi_{r}(t)>t|\xi,\mathbf{W}_{2}),
\end{equation*}
and
\begin{equation*}
    \lim_{r\rightarrow\infty} \Pr(\phi_{r}(t)>t|\xi,\mathbf{W}_{2})\leq \Pr(\phi>t|\xi,\mathbf{W}_{2}).
\end{equation*} 
As $\phi_{r}(t)\leq \phi(t)$ for any $r$, $\Pr(\phi_{r}(t)>t|\xi,\mathbf{W}_{2})\leq \Pr(\phi(t)>t|\xi,\mathbf{W}_{2})$ holds. Then, all we need to show is the following:
\begin{equation*}
    \Phi_{r}(t)\rightarrow \Phi(t) \quad \text{as $r\rightarrow\infty$}.
\end{equation*}
We consider the following two cases: (i) $\Phi(t) = \infty$ and (ii) $\Phi(t) < \infty$.

\textbf{Case (i)}:
    $\Phi(t)=\infty$, that is, the minimization problem defining $\Phi(t)$ is infeasible. In this case, for all $\|\Sigma_x^{-1/2}(\theta_{1}+\theta_{2})\|\leq 2t$, we have
    \begin{equation*}
        \|\mathbf{W}_{1}\theta_{1}+\mathbf{W}_{2}\theta_{2}-\xi\|_{2}>0.
    \end{equation*}
    By closedness, there exists $\eta=\eta(\mathbf{W}_{1},\mathbf{W}_{2}, \xi)$ such that
    \begin{equation*}
        \|\mathbf{W}_{1}\theta_{1}+\mathbf{W}_{2}\theta_{2}-\xi\|_{2}\geq \eta.
    \end{equation*}
    By definition, $\eta$ is independent of $r$. Then, it holds that
    \begin{equation*}
        \Phi_{r}(t)=\min_{\|\Sigma_x^{-1/2}(\theta_{1}+\theta_{2})\|_{2}\leq 2t}\max_{\|\lambda\|_{2}\leq r}\langle\lambda,\mathbf{W}_{1}\theta_{1}+\mathbf{W}_{2}\theta_{2}-\xi\rangle+\|\Sigma_x^{-1/2}(\theta_{1}+\theta_{2})\|\geq r\eta.
    \end{equation*}
    Therefore, $\Phi_{r}(t)\rightarrow\infty$ as $r\rightarrow\infty$.
    
    \textbf{Case (ii)}: $\Phi(t)<\infty$, that is, the minimization problem defining $\Phi(t)$ is feasible. By compactness, $\Phi_{r}(t)$ has solutions for the minimax problem. Let $(\theta_{1}(r), \theta_{2}(r))$ be one of solutions for $\Phi_{r}(t)$. If we take a sequence $\{(\theta_{1}(r), \theta_{2}(r))\}_{r=1}^{\infty}$, there exists a convergent subsequence $\{(\theta_{1}(r_{n}), \theta_{2}(r_{n}))\}_{n=1}^{\infty}$ by sequential compactness. Let $\{(\theta_{1}(\infty), \theta_{2}(\infty))\}$ be a convergent point of this subsequence. 
    For the sake of contradiction, assume that $\mathbf{W}_{1}\theta_{1}(\infty)+\mathbf{W}_{2}\theta_{2}(\infty)\neq \xi$. By continuity, there exists $\eta$ and $\varepsilon$ such that, if $\|(\theta_{1},\theta_{2})-(\theta_{1}(\infty), \theta_{2}(\infty))\|_{2}\leq \varepsilon$, 
    \begin{equation*}
               \|\mathbf{W}_{1}\theta_{1}+\mathbf{W}_{2}\theta_{2}-\xi\|_{2}\geq \eta.
    \end{equation*}
    This implies that for a sufficiently large $n$, it holds that
    \begin{equation*}
         \|\mathbf{W}_{1}\theta_{1}(r_{n})+\mathbf{W}_{2}\theta_{2}(r_{n})-\xi\|_{2}\geq\eta.
    \end{equation*}
    As in the previous section, we have
    \begin{equation*}
        \Phi_{r_{n}}(t)=\max_{\|\lambda\|_{2}\leq r_{n}}\langle\lambda,\mathbf{W}_{1}\theta_{1}(r_{n})+\mathbf{W}_{2}\theta_{2}(r_{n})-\xi\rangle+\|\Sigma_x^{-1/2}(\theta_{1}(r_{n})+\theta_{2}(r_{n}))\|\geq r_{n}\eta.
    \end{equation*}
    Hence, $\lim_{n\rightarrow\infty}\Phi_{r_{n}}(t)=\infty$. However, this is a contradiction because, for any $r$, $\Phi_{r}(t)\leq \Phi(t)<\infty$. Therefore,  $\mathbf{W}_{1}\theta_{1}(\infty)+\mathbf{W}_{2}\theta_{2}(\infty)=\xi$.
    If we set $\lambda=0$, we have
    \begin{equation*}
        \Phi_{r_{n}}(t)\geq \|\Sigma_x^{-1/2}(\theta_{1}(r_{n})+\theta_{2}(r_{n}))\|.
    \end{equation*}
    By continuity, we show that
    \begin{equation*}
        \lim\inf_{n\rightarrow\infty}\Phi_{r_{n}}(t)\geq \lim_{n\rightarrow\infty}\|\Sigma_x^{-1/2}(\theta_{1}(r_{n})+\theta_{2}(r_{n}))\|=\|\Sigma_x^{-1/2}(\theta_{1}(\infty)+\theta_{2}(\infty))\|\geq \Phi(t).
    \end{equation*}
     As, for any $r$, $\Phi_{r}(t)\leq \Phi(t)$, we have
     \begin{equation*}
         \lim\sup_{n\rightarrow\infty}\Phi_{r_{n}}(t)\leq \Phi(t)\leq \lim\inf_{n\rightarrow\infty}\Phi_{r_{n}}(t),
     \end{equation*}
     that is, $\lim_{n\rightarrow\infty}\Phi_{r_{n}}(t)=\Phi(t)$. As $\Phi_{r}(t)$ is an increasing function in terms of $r$, we have $\lim_{r\rightarrow\infty}\Phi_{r}(t)=\Phi(t)$.

Through the application of Theorem \ref{VariantCGMT} and two inequalities, $\Pr(\Phi>t|\xi,\mathbf{W}_{2})\leq\lim_{r\rightarrow\infty} \Pr(\Phi_{r}(t)>t|\xi,\mathbf{W}_{2})\ $ and $\lim_{r\rightarrow\infty} \Pr(\phi_{r}(t)>t|\xi,\mathbf{W}_{2})\leq \Pr(\phi>t|\xi,\mathbf{W}_{2})$, we prove the result $\Pr(\Phi>t)\leq 2\pr(\phi>t)$. 
By the last part of Theorem \ref{VariantCGMT}, we have
\begin{equation*}
    \Pr(\Phi_{r}(t)>t|\xi,\mathbf{W}_{2})\leq 2\Pr(\phi_{r}(t)>t|\xi,\mathbf{W}_{2}).
\end{equation*}
As $\Phi_{r}(t)$ monotonically increases to $\Phi(t)$ almost surely, it follows from the continuity of the probability measure that
\begin{align*}
    \Pr(\Phi>t|\xi,\mathbf{W}_{2})&=\Pr(\Phi(t)>t|\xi,\mathbf{W}_{2}) \\
    &\leq \Pr(\cup_{r}\cap_{r'\geq r}\Phi_{r'}(t)>t|\xi,\mathbf{W}_{2}) \\
    &=\Pr(\varliminf_{r\rightarrow\infty}\Phi_{r}(t)>t|\xi,\mathbf{W}_{2}) \\
    &=\lim_{r\rightarrow\infty} \Pr(\Phi_{r}(t)>t|\xi,\mathbf{W}_{2}).
\end{align*}
As $\phi_{r}(t)\leq \phi(t)$ holds for any $r$, $\Pr(\phi_{r}(t)>t|\xi,\mathbf{W}_{2})\leq \Pr(\phi(t)>t|\xi,\mathbf{W}_{2})$. Therefore, we have
\begin{align*}
    \Pr(\Phi>t|\xi,\mathbf{W}_{2})\leq 2\lim_{r\rightarrow\infty}\Pr(\phi_{r}(t)>t|\xi,\mathbf{W}_{2})\leq 2\Pr(\phi(t)>t|\xi,\mathbf{W}_{2}).
\end{align*}
\end{proof}

Then, we obtain the general upper bound for the auxiliary optimization problem  (AO).
\begin{lemma}\label{Variant_Lemma8}
Denote $P_{z}$ and $P_{u}$ as the orthogonal projection matrix onto the space spanned by $\Xi_z$ and $\Sigma_{u}$, respectively. Let $v_{*}=\arg\min_{v\in\partial\|\Xi_z^{1/2}H\|}\|v\|_{\Xi_z}$. Assume that there exists $\varepsilon_{1},\varepsilon_{2}\geq0$ such that with probability at least $1-\delta/2$,
\begin{equation}\label{Variant_65}
    \|v^{*}\|_{\Xi_z}\leq(1+\varepsilon_{1})\mathbb{E}\|v^{*}\|_{\Xi_z},
\end{equation}
and
\begin{equation}\label{Variant_66}
    \|P_{z}v^{*}\|^{2}\leq 1+\varepsilon_{2}.
\end{equation}
Define $\varepsilon$ as
\begin{equation*}
    \varepsilon:=16\sqrt{\frac{\rank(\Sigma_{u})}{n}}+28\sqrt{\frac{\log(32/\delta)}{n}}+8\sqrt{\frac{\log(8/\delta)}{r_{\|\cdot\|}(\Xi_z)}}+2(1+\varepsilon_{1})^{2}\frac{n}{R_{\|\cdot\|}(\Xi_z)}+2\varepsilon_{2}.
\end{equation*}
If $n$ and the effective ranks are sufficiently large such that $\varepsilon\leq1$, then with probability at least $1-\delta$, it holds that
\begin{equation}\label{Variant_67}
    \phi^{2}\leq\|\Sigma_{u}^{+}\corrcoef\|^{2}+(1+\varepsilon)\tilde{\sigma}^{2}\frac{n}{(\mathbb{E}\|\Sigma_{2}^{1/2}H\|_{*})^{2}},
\end{equation}
where we denote $r_{\|\cdot\|}(\Sigma)$ and $R_{\|\cdot\|}(\Sigma)$ as follows:
\begin{equation*}
    r_{\|\cdot\|}(\Sigma)=\left(\frac{E\|\Sigma^{1/2}H\|_{*}}{\sup_{\|u\|\leq1}\|u\|_{\Sigma}}\right)^{2}\ \text{and} \ R_{\|\cdot\|}(\Sigma)=\left(\frac{E\|\Sigma^{1/2}H\|_{*}}{E\|v^{*}\|_{\Sigma}}\right)^{2}.
\end{equation*}
\end{lemma}

\begin{proof}[Proof of Lemma \ref{Variant_Lemma8}]
Fix $\delta \in (0,1)$ in this proof.
To simplify notations, we define coefficients:
\begin{align*}
\alpha_{1}:=2\sqrt{\frac{\log(32/\delta)}{n}} \mbox{~and~}
\alpha_{2}:=\sqrt{\frac{\rank(\Sigma_{u})+1}{n}}+2\sqrt{\frac{\log(16/\delta)}{n}}. 
\end{align*}
To prepare for the derivation of the upper bound as in the proof of Lemma \ref{Variant_Lemma5}, we consider the following three inequalities:
\begin{enumerate}
 \item[(i)]  By Lemma \ref{Variant_Lemma1}, uniformly over all $\theta_{2}\in \Sigma_{u}^{1/2}(\mathbb{R}^{p}-\theta_{0})$, 
    it holds that
    \begin{equation}
    |\langle\xi-\mathbf{W}_{2}\theta_{2}, G \rangle|\leq\|\xi-\mathbf{W}_{2}\theta_{2}\|_{2}\|G\|_{2}\alpha_{2}. \tag{\ref{Variant_49}}
\end{equation}
\item[(ii)] By Lemma \ref{Variant_Lemma2}, it holds that
\begin{equation}\label{Variant_69}
    -\alpha_{1}\leq\frac{1}{\sqrt{n}}\|G\|_{2}-1\leq\alpha_{1} \tag{\ref{Variant_52}}
\end{equation}
and 
\begin{align}
    -\alpha_{1}\sqrt{\sigma^{2}-2\rho^{T}\theta_{2}+\theta_{2}^{T}\theta_{2}}&\leq\frac{1}{\sqrt{n}}\|\xi-\mathbf{W}_{2}\theta_{2}\|_{2}-\sqrt{\sigma^{2}-2\rho^{T}\theta_{2}+\theta_{2}^{T}\theta_{2}} \notag \\
    &\leq\alpha_{1}\sqrt{\sigma^{2}-2\rho^{T}\theta_{2}+\theta_{2}^{T}\theta_{2}}. \tag{\ref{Variant_53}}
\end{align}
\item[(iii)] By Theorem \ref{Variant_Theorem6}, it holds that
\begin{align}
    \|\Xi_z^{1/2}H\|_{*}&\geq\mathbb{E}\|\Xi_z^{1/2}H\|_{*}-\sup_{\|u\|\leq1}\|u\|_{\Xi_{z}}\sqrt{2\log(8/\delta)} \notag \\
    &=\left(1-\sqrt{\frac{2\log(8/\delta)}{r_{\|\cdot\|}(\Xi_z)}}\right) \mathbb{E}\|\Xi_z^{1/2}H\|_{*} \label{Variant_71},
\end{align}
because $\|\Xi^{1/2}_{z}H\|_{*}$ is a $\sup_{\|u\|\leq 1}\|u\|_{\Xi_{z}}$-Lipschitz continuous function of $H$.
\end{enumerate}

We construct the upper bound from the restriction of the optimization problem \eqref{Variant_58}. From the restriction of the auxiliary problem, we have 
\begin{align*}
    \|\xi-\mathbf{W}_{2}\theta_{2}-\|\theta_{1}\|_{2}G\|_{2}^{2}&=\|\xi-\mathbf{W}_{2}\theta_{2}\|_{2}^{2}-2\|\theta_{1}\|_{2}\langle \xi-\mathbf{W}_{2}\theta_{2},G\rangle+\|\theta_{1}\|_{2}^{2}\|G\|_{2}^{2} \\
    &\leq (1+\alpha_{2})\left(\|\xi-\mathbf{W}_{2}\theta_{2}\|_{2}^{2}+\|\theta_{1}\|_{2}^{2}\|G\|_{2}^{2}\right),
\end{align*}
where the last inequality follows from (\ref{Variant_49}) and the AM-GM inequality. Combining the results of (\ref{Variant_69}) and (\ref{Variant_70}) yields 
\begin{align}\label{8-1}
    \|\xi-\mathbf{W}_{2}\theta_{2}-\|\theta_{1}\|_{2}G\|_{2}^{2}
    &\leq (1+\alpha_{2})(1+\alpha_{1})^{2}n\left(\sigma^{2}-2\rho^{T}\theta_{2}+\theta_{2}^{T}\theta_{2}+\theta_{1}^{T}\theta_{1}\right).
\end{align}

To consider an upper bound of the ridgeless estimator, we need to choose a suitable $\theta$ which satisfies the restriction of the auxiliary problem. We consider the following form of $\theta$:
\begin{equation*}
    \theta=P_{u}(\Sigma_{u}^{+}\corrcoef)+sP_{z}v^{*}.
\end{equation*}
As $\Sigma_{u}^{1/2}\theta=\Sigma_{u}^{1/2}\Sigma_{u}^{+}\corrcoef=(\Sigma_{u}^{1/2})^{+}\corrcoef$ and $\Xi_z^{1/2}\theta=s\Xi_z^{1/2}v^{*}$, from (\ref{8-1}) and the restriction of the auxiliary problem, it suffices to choose $s$ such that
\begin{align*}
    &(1+\alpha_{2})(1+\alpha_{1})^{2}n\left(\sigma^{2}-2\rho^{T}\Sigma_{u}^{1/2}\Sigma_{u}^{+}\corrcoef+\corrcoef^{T}\Sigma_{u}^{+}\Sigma_{u}\Sigma_{u}^{+}\corrcoef+s^{2}\|\Xi_z^{1/2}v^{*}\|_{2}^{2}\right) \\
    &=(1+\alpha_{2})(1+\alpha_{1})^{2}n(\tilde{\sigma}^{2}+s^{2}\|\Xi_z^{1/2}v^{*}\|_{2}^{2}) \\
    &\leq (\langle H,s\Xi_z^{1/2}v^{*} \rangle)^{2} \\
    &=s^{2}\|\Xi_z^{1/2}H\|_{*}^{2}.
\end{align*}
Solving for $s$, we can choose
\begin{equation*}
    s^{2}=\tilde{\sigma}^{2}\left(\frac{\|\Xi_z^{1/2}H\|_{*}^{2}}{(1+\alpha_{2})(1+\alpha_{1})^{2}n}-\|v^{*}\|^{2}_{\Xi_z}\right)^{-1},
\end{equation*}
under the assumption that 
\begin{equation}\label{positive}
    \left(\frac{\|\Xi_z^{1/2}H\|_{*}^{2}}{(1+\alpha_{2})(1+\alpha_{1})^{2}n}-\|v^{*}\|^{2}_{\Xi_z}\right)>0.
\end{equation}

We need to guarantee (\ref{positive}) holds. By (\ref{Variant_65}) and (\ref{Variant_71}), we have
\begin{align*}
    &\frac{\|\Xi_z^{1/2}H\|_{*}^{2}}{(1+\alpha_{2})(1+\alpha_{1})^{2}n}-\|v^{*}\|^{2}_{\Xi_z} \\
    \geq & \frac{(\mathbb{E}\|\Xi_z^{1/2}H\|_{*})^{2}}{(1+\alpha_{2})(1+\alpha_{1})^{2}n}\left(1-\sqrt{\frac{2\log(8/\delta)}{r_{\|\cdot\|}(\Xi_z)}}\right)^{2}-(1+\varepsilon_{1})^{2}(\mathbb{E}\|v^{*}\|_{\Xi_{z}})^{2} \\
    \geq & \frac{(\mathbb{E}\|\Xi_z^{1/2}H\|_{*})^{2}}{n}\left(\frac{1}{(1+\alpha_{2})(1+\alpha_{1})^{2}}\left(1-2\sqrt{\frac{2\log(8/\delta)}{r_{\|\cdot\|}(\Xi_z)}}\right)
    -(1+\varepsilon_{1})^{2}\frac{n}{R_{\|\cdot\|}(\Sigma_{2})}\right),
\end{align*}
where the last inequality follows from the definition of $R_{\|\cdot\|}(\cdot)$.

As in the proof of Lemma \ref{Variant_Lemma5}, we linearize the terms including $\alpha_{1}$ and $\alpha_{2}$ to simplify the upper bound. Provided that $\alpha_{1}<1$, we have
\begin{align*}
    (1+\alpha_{2})(1+\alpha_{1})^{2}&=1+2\alpha_{1}+\alpha_{1}^{2}+\alpha_{2}+2\alpha_{2}\alpha_{1}+\alpha_{2}\alpha_{1}^{2} \\
    &\leq 1+3\alpha_{1}+4\alpha_{2}.
\end{align*}
As $(1-x)^{-1}\geq 1+x$ for any $x$, it holds that
\begin{align*}
\frac{1}{(1+\alpha_{2})(1+\alpha_{1})^{2}}&\geq ( 1+3\alpha_{1}+4\alpha_{2})^{-1} \\
&\geq( 1-(3\alpha_{1}+4\alpha_{2})).
\end{align*}
Hence,  we have
\begin{align*}
    &\frac{1}{(1+\alpha_{2})(1+\alpha_{1})^{2}}\left(1-2\sqrt{\frac{2\log(8/\delta)}{r_{\|\cdot\|}(\Xi_z)}}\right)
    -(1+\varepsilon_{1})^{2}\frac{n}{R_{\|\cdot\|}(\Sigma_{2})} \\
    \geq &( 1-(3\alpha_{1}+4\alpha_{2}))\left(1-2\sqrt{\frac{2\log(8/\delta)}{r_{\|\cdot\|}(\Xi_z)}}\right)
    -(1+\varepsilon_{1})^{2}\frac{n}{R_{\|\cdot\|}(\Sigma_{2})} \\
    \geq &1-(3\alpha_{1}+4\alpha_{2})-2\sqrt{\frac{2\log(8/\delta)}{r_{\|\cdot\|}(\Xi_z)}}
    -(1+\varepsilon_{1})^{2}\frac{n}{R_{\|\cdot\|}(\Sigma_{2})} \\
    \geq &1-\varepsilon',
\end{align*}
where
\begin{equation*}
    \varepsilon'=8\sqrt{\frac{\rank(\Sigma_{u})}{n}}+14\sqrt{\frac{\log(32/\delta)}{n}}+4\sqrt{\frac{\log(8/\delta)}{r_{\|\cdot\|}(\Xi_z)}}+(1+\varepsilon_{1})^{2}\frac{n}{R_{\|\cdot\|}(\Xi_z)}.
\end{equation*}

Finally, we derive the upper bound of the ridgeless estimator. If $\varepsilon'\leq 1/2$, because $(1-x)^{-1}\leq 1+2x$ for $x\in[0,1/2]$, it holds that 
\begin{align*}
    s^{2}&\leq \tilde{\sigma}^{2}\frac{n}{(\mathbb{E}\|\Xi_z^{1/2}H\|_{*})^{2}}\frac{1}{1-\varepsilon'} \leq \tilde{\sigma}^{2}\frac{n}{(\mathbb{E}\|\Xi_z^{1/2}H\|_{*})^{2}}(1+2\varepsilon'). 
\end{align*}
Then, it holds from (\ref{Variant_66}) that 
\begin{equation*}
    \phi^{2}\leq\|\Sigma_{u}^{+}\corrcoef\|^{2}+s^{2}\|P_{z}v^{*}\|^{2} \\
    \leq \|\Sigma_{u}^{+}\corrcoef\|^{2}+s^{2}(1+\varepsilon_{2}).
\end{equation*}
Therefore, we have
\begin{equation*}
    \phi^{2}\leq\|\Sigma_{u}^{+}\corrcoef\|^{2}+(1+\varepsilon)\tilde{\sigma}^{2}\frac{n}{(\mathbb{E}\|\Xi_{z}^{1/2}H\|_{*})^{2}},
\end{equation*}
with $\varepsilon=2\varepsilon'+2\varepsilon_{2}$.
\end{proof}

We can now derive the general norm bound in the case where the covariates correlate with errors.
\begin{theorem}[General norm bound]\label{Variant_Theorem4}
There exists an absolute constant $C_{2}\leq 56$ such that the following
is true. Under Assumptions \ref{asmp:Gaussianity} and \ref{asmp:orthogonal} with covariance split $\Sigma_{x}=\Xi_z + \Sigma_{u}$, let $\|\cdot\|$ be an
arbitrary norm, and fix $\delta\leq1/4$. Denote the $\ell_{2}$ orthogonal projection matrix onto the space spanned by $\Xi_z$, $\Sigma_{u}$ as $P_{z}$, $P_{u}$, respectively. Let $H$ be normally distributed with mean zero and variance $I_{d}$, that is, $H\sim N(0,I_{d})$. Denote $v_{*}$ as $\arg\min_{v\in\partial\|\Xi_z^{1/2}H\|_{*}} \|v\|_{\Xi_z}$. Suppose that there exist $\varepsilon_{1},\varepsilon_{2}\geq0$ such that with probability at least $1-\delta/4$ 
\begin{equation*}
\|v^{*}\|_{\Xi_z}\leq(1+\varepsilon_{1})E\|v^{*}\|_{\Xi_z}
\end{equation*}
and
\begin{equation*}
    \|Pv^{*}\|^{2}\leq1+\varepsilon_{2}.
\end{equation*}
Let $\varepsilon$ denote $C_{2}\left(\sqrt{\frac{\rank(\Sigma_{u})}{n}}+\sqrt{\frac{\log(1/\delta)}{r_{\|\cdot\|}(\Xi_z)}}+\sqrt{\frac{\log(1/\delta)}{n}}+(1+\varepsilon_{1})^{2}\frac{n}{R_{\|\cdot\|}(\Xi_z)}+\varepsilon_{2}\right)$. Then, if $n$ and the effective ranks are large enough that $\varepsilon\leq1$, with probability at least $1-\delta$, it holds that
\begin{equation*}
    \|\hat{\theta}\|\leq\|\theta_{0}\|+\|\Sigma_{u}^{+}\corrcoef\|+(1+\varepsilon)^{1/2}\tilde{\sigma}\frac{\sqrt{n}}{(\mathbb{E}\|\Xi_{z}^{1/2}H\|_{*})}.
\end{equation*}
\end{theorem}

\begin{proof}[Proof of Theorem \ref{Variant_Theorem4}]
For any $t>0$, it holds from Lemmas \ref{Variant_Lemma6} and \ref{Variant_Lemma7} that
\begin{equation*}
\Pr(\|\hat{\theta}\|>t)\leq\Pr(\Phi>t-\|\theta_{0}\|)\leq2\Pr(\phi\geq t-\|\theta_{0}\|).
\end{equation*}
Lemma \ref{Variant_Lemma8} implies that the above term is upper bounded by $\delta$ if we choose $t-\|\theta_{0}\|$ using the result (\ref{Variant_67}) with $\delta$ replaced by $\delta/2$. We obtain the stated result by moving $\|\theta_{0}\|$ to the other side.
\end{proof}

When we consider the Euclidean space, we can reduce the upper bound of the ridgeless estimator to a simpler bound.
\noindent
\begin{theorem}[Euclidean norm bound; special case of Theorem \ref{Variant_Theorem4}]\label{Variant_Theorem2}
Fix any $\delta\leq1/4$. Under Assumptions \ref{asmp:Gaussianity} and \ref{asmp:orthogonal} with covariance splitting $\Sigma_x=\Xi_z + \Sigma_{u}$, there exists some $\varepsilon\lesssim \sqrt{\frac{\rank(\Sigma_{u})}{n}}+\sqrt{\frac{\log(1/\delta)}{r(\Xi_z)}}+\sqrt{\frac{\log(1/\delta)}{n}}+\frac{n\log(1/\delta)}{R(\Xi_z)}$ such that the following is true. If $n$ and the effective ranks are such that $\varepsilon\leq1$ and $R(\Xi_z)\gtrsim \log(1/\delta)^{2}$, then with probability at least $1-\delta$, it holds that
\begin{equation}\label{Variant_6}
\|\hat{\theta}\|_{2}\leq\|\theta_{0}\|_{2}+\|\Sigma_{u}^{+}\corrcoef\|_{2}+(1+\varepsilon)^{1/2}\tilde{\sigma}\sqrt{\frac{n}{\trace(\Xi_z)}},
\end{equation}
where $\tilde{\sigma}^{2}:=\sigma^{2}-\corrcoef^{T}\Sigma_{u}^{+}\corrcoef$.
\end{theorem}

\begin{proof}[Proof of Theorem \ref{Variant_Theorem2}]
Throughout this proof, we simplify the upper bound, especially $
    n/R_{\|\cdot\|_{2}}(\Xi_z)$ and $1/r_{\|\cdot\|_{2}}(\Xi_{z})$, in Theorem \ref{Variant_Theorem4}.
By the definition of the dual norm and $\partial\|\Xi_{x}^{1/2}H\|_{*}$ with Euclidean norm, $v^{*}$ is equal to $\Xi_z^{1/2}H/\|\Xi_z^{1/2}H\|_{2}$. Hence, $\|v^{*}\|_{\Xi_z}$ is $\|\Xi_zH\|_{2}/\|\Xi_z^{1/2}H\|_{2}$. From the result of (\ref{Variant_78}), for some constant $c>0$, we can choose $\varepsilon_{1}$ such that
\begin{equation*}
    (1+\varepsilon_{1})E\|v^{*}\|_{\Xi_z}=c\sqrt{\log(16/\delta)\frac{\trace(\Xi_z^{2})}{\trace(\Xi_z)}}.
\end{equation*}
If we assume effective rank is sufficiently large, (\ref{Variant_72}) provides that $\left(E\|\Xi_z^{1/2}H\|_{2}\right)^{2}\gtrsim \trace(\Xi_z)$. Therefore, we have
\begin{equation*}
    (1+\varepsilon_{1})^{2}\frac{n}{R_{\|\cdot\|_{2}}(\Xi_z)}=n\frac{(1+\varepsilon_{1})^{2}(E\|v^{*}\|_{\Xi_z})^{2}}{\left(E\|\Xi_z^{1/2}H\|_{2}\right)^{2}}\lesssim n\log(16/\delta)\frac{\trace(\Xi_z^{2})}{\trace(\Xi_z)^{2}}=\frac{n\log(16/\delta)}{R(\Xi_z)}.
\end{equation*}
Moreover, because $P_{z}$ is an $l_{2}$ projection matrix, let $\varepsilon_{2}$ be zero. Then, it holds from (\ref{Variant_74}) of Lemma \ref{Variant_Lemma9}  that 
\begin{equation*}
\varepsilon\lesssim  \sqrt{\frac{\rank(\Sigma_{u})}{n}}+\sqrt{\frac{\log(1/\delta)}{n}}+ \sqrt{\frac{\log(1/\delta)}{r(\Xi_z)}}+\frac{n\log(1/\delta)}{R(\Xi_z)}. 
\end{equation*}
By using the inequality $(1-x)^{-1}\leq1+2x$ for $x\in[0,1/2]$ and (\ref{Variant_72}) of Lemma \ref{Variant_Lemma9}, we finally obtain
\begin{align*}
    (1+\varepsilon)^{1/2}\tilde{\sigma}\frac{\sqrt{n}}{E\|\Xi_z^{1/2}H\|_{2}}&\leq (1+\varepsilon)^{1/2}\left(1-\frac{1}{r(\Xi_z)}\right)^{-1/2}\tilde{\sigma}\frac{\sqrt{n}}{\trace(\Xi_z)}  \\
    &\leq (1+\varepsilon)^{1/2}\left(1+\frac{2}{r(\Xi_z)}\right)^{1/2}\tilde{\sigma}\frac{\sqrt{n}}{\trace(\Xi_z)} \\
    &\leq \left(1+2\varepsilon+\frac{2}{r(\Xi_z)}\right)^{1/2}\tilde{\sigma}\frac{\sqrt{n}}{\trace(\Xi_z)},
\end{align*}
with $\varepsilon$ replaced by
\begin{equation*}
    \varepsilon'=2\varepsilon+\frac{2}{r(\Xi_z)}\lesssim \sqrt{\frac{\rank(\Sigma_{u})}{n}}+ \sqrt{\frac{\log(1/\delta)}{n}}+ \sqrt{\frac{\log(1/\delta)}{r(\Xi_z)}}+\frac{n\log(1/\delta)}{R(\Xi_z)}.
\end{equation*}
\end{proof}

\section{Benign Overfitting} \label{app_sec:example_benign}
In this section, we state the primary result on the conditions of benign overfitting by combining the results from the two previous sections. 
First, we derive the result with an arbitrary norm.

\begin{theorem}[Benign Overfitting]\label{Variant_Theorem5}
Fix any $\delta\leq1/2$. Under Assumptions \ref{asmp:Gaussianity} and \ref{asmp:orthogonal} with covariance splitting $\Sigma_x=\Xi_z + \Sigma_{u}$, let $\gamma$ and $\varepsilon$ be as defined in Corollary \ref{Variant_Corollary3} and Theorem \ref{Variant_Theorem4}. Suppose that $n$ and the effective ranks are such that $R(\Xi_z)\gtrsim \log(1/\delta)^{2}$ and $\gamma,\varepsilon\leq1$. Then, with probability at least $1-\delta$, it holds that 
\begin{align*}
    &\|\hat{\theta}-\theta_{0}\|_{\Xi_z}^{2}\leq(1+\gamma)(1+\varepsilon)\left(\tilde{\sigma}+(\|\Sigma_{u}^{+}\corrcoef\|+\|\theta_{0}\|)\frac{\mathbb{E}\|\Xi_{z}^{1/2}H\|_{*}}{\sqrt{n}}\right)^{2}-\tilde{\sigma}^{2}.
\end{align*}
\end{theorem}

\begin{proof}[Proof of Theorem \ref{Variant_Theorem5}]
From the result of Theorem \ref{Variant_Theorem4}, if we adopt
\begin{equation*}
    B=\|\theta_{0}\|+\|\Sigma_{u}^{+}\corrcoef\|+(1+\varepsilon)^{1/2}\tilde{\sigma}\frac{\sqrt{n}}{(\mathbb{E}\|\Xi_{z}^{1/2}H\|_{*})},
\end{equation*}
then $\{\theta:\|\theta\|\leq B\}\cap\{\theta: \mathbf{X}\theta=\mathbf{Y}\}$ is not empty with high probability. Clearly, $B>\|\theta_{0}\|$. This intersection necessarily includes the ridgeless estimator $\hat{\theta}$. Therefore, it holds from Corollary \ref{Variant_Corollary3} that
\begin{align*}
     \|\hat{\theta}-\theta_{0}\|_{\Xi_z}^{2} 
     \leq&  \max_{\|\theta\|\leq B, \mathbf{Y}=\mathbf{X}\theta}\|\theta-\theta_{0}\|_{\Xi_z}^{2} \\
     \leq& (1+\gamma)\frac{B^{2}(\mathbb{E}\|\Xi_{z}^{1/2}H\|_{*})^{2}}{n}-\tilde{\sigma}^{2} \\
     =&(1+\gamma)\left((\|\theta_{0}\|+\|\Sigma_{u}^{+}\corrcoef\|)\frac{\mathbb{E}\|\Xi_{z}^{1/2}H\|_{*}}{\sqrt{n}}+(1+\varepsilon)^{1/2}\tilde{\sigma}\right)^{2}-\tilde{\sigma}^{2} \\
     \leq& (1+\gamma)(1+\varepsilon)\left(\tilde{\sigma}+(\|\theta_{0}\|+\|\Sigma_{u}^{+}\corrcoef\|)\frac{\mathbb{E}\|\Xi_{z}^{1/2}H\|_{*}}{\sqrt{n}}\right)^{2}-\tilde{\sigma}^{2}.
\end{align*}
Then, we obtain the statement.
\end{proof}

\noindent
\textbf{Theorem \ref{Variant_Theorem11}} (Sufficient conditions) 
\textit{
Under Assumptions \ref{asmp:Gaussianity} and \ref{asmp:orthogonal}, let $\hat{\theta}$ be the ridgeless estimator. Let $\|\cdot\|$ denote an arbitrary norm. Suppose that as $n$ goes to $\infty$, the covariance splitting $\Sigma_x=\Xi_z + \Sigma_{u}$ satisfies the following conditions:}
\begin{enumerate}
    \item[(i)] (Small large-variance dimension.)
    \begin{equation*}
        \lim_{n\rightarrow\infty}\frac{\rank(\Sigma_{u})}{n}=0.
    \end{equation*}
    \item[(ii)] (Large effective dimension.)
    \begin{equation*}
        \lim_{n\rightarrow\infty}\frac{1}{r_{\|\cdot\|}(\Xi_z)}=0 \quad \text{and} \quad \lim_{n\rightarrow\infty}\frac{n}{R_{\|\cdot\|}(\Xi_z)}=0.
    \end{equation*}
    \item[(iii)] (No aliasing condition.)
    \begin{equation*}
        \lim_{n\rightarrow\infty}\frac{\|\theta_{0}\|\mathbb{E}\|\Xi_{z}^{1/2}H\|_{*}}{\sqrt{n}}=0.
    \end{equation*}
    \item[(iv)] (Contracting $\ell_{2}$ projection condition.) For any $\eta>0$,
    \begin{equation*}
\lim_{n\rightarrow\infty}\Pr(\|P_{u}v^{*}\|^{2}>1+\eta)=0.
    \end{equation*}
    \item[(v)] (Condition for the minimal interpolation of instrumental variable)
    \begin{equation*}
        \lim_{n\rightarrow\infty}\frac{\|\Sigma_{u}^{+}\corrcoef\|\mathbb{E}\|\Xi_{z}^{1/2}H\|_{*}}{\sqrt{n}}=0.
    \end{equation*}
\end{enumerate}
Then, $ \|\hat{\theta}-\theta_{0}\|_{\Xi_z}^{2}$ converges to $0$ in probability.

\begin{proof}[Proof of Theorem \ref{Variant_Theorem11}]
We take advantage of the upper bound on the projected RMSE derived in Theorem \ref{Variant_Theorem5}.
To begin with, we reorganize the upper bound in Theorem \ref{Variant_Theorem5} to elucidate the terms that should be sufficiently small for the projected RMSE to converge.
Fix any $\eta > 0$.
By trivial calculation, we have
\begin{align}
    &(1+\gamma)(1+\varepsilon)\left(\tilde{\sigma}+(\|\theta_{0}\|+\|\Sigma_{u}^{+}\corrcoef\|)\frac{\mathbb{E}\|\Xi_{z}^{1/2}H\|_{*}}{\sqrt{n}}\right)^{2}-\tilde{\sigma}^{2}  \notag \\
    &=(1+\gamma)(1+\varepsilon)\left(\left((\|\theta_{0}\|+\|\Sigma_{u}^{+}\corrcoef\|)\frac{\mathbb{E}\|\Xi_{z}^{1/2}H\|_{*}}{\sqrt{n}}\right)^{2}+2\tilde{\sigma}(\|\theta_{0}\|+\|\Sigma_{u}^{+}\corrcoef\|)\frac{\mathbb{E}\|\Xi_{z}^{1/2}H\|_{*}}{\sqrt{n}}\right)  \notag \\
    &\quad  +(\gamma+\varepsilon+\gamma\varepsilon)\tilde{\sigma}^{2} \notag \\
    &\leq(1+\gamma)(1+\varepsilon)\left(\left((\|\theta_{0}\|+\|\Sigma_{u}^{+}\corrcoef\|)\frac{\mathbb{E}\|\Xi_{z}^{1/2}H\|_{*}}{\sqrt{n}}\right)^{2}+2\sigma(\|\theta_{0}\|+\|\Sigma_{u}^{+}\corrcoef\|)\frac{\mathbb{E}\|\Xi_{z}^{1/2}H\|_{*}}{\sqrt{n}}\right) \notag \\
    &\quad +(\gamma+\varepsilon+\gamma\varepsilon)\sigma^{2} \notag \\
    &\leq \eta. \label{Variant_83}
\end{align}
The second to last inequality follows $\tilde{\sigma}^{2}=\sigma^2 - \|\corrcoef\|_{\Sigma_u^+}^2\leq\sigma^{2}$, and the last inequality holds by selecting sufficiently small $\gamma, \varepsilon$, and  $(\|\theta_{0}\|+\|\Sigma_{u}^{+}\corrcoef\|)(\mathbb{E}\|\Xi_{z}^{1/2}H\|_{*}/\sqrt{n})$. 

 Fix any $\delta>0$. Conditions (i) and (ii) in Theorem \ref{Variant_Theorem11} make $\gamma$ sufficiently small for large enough $n$. $(\|\theta_{0}\|+\|\Sigma_{u}^{+}\corrcoef\|)(\mathbb{E}\|\Xi_{z}^{1/2}H\|_{*}/\sqrt{n})$ goes to zero from conditions (iii) and (v). From condition (iv) of Theorem \ref{Variant_Theorem11}, $\varepsilon_{2}$ in $\varepsilon$ can also be arbitrarily small.
 
Finally, we need to specify the conditions when $\varepsilon$ can be sufficiently small. By the definition of $R_{\|\cdot\|(\Xi_{z})}$, we have
\begin{equation*}
    \sqrt{\frac{n}{R_{\|\cdot\|}(\Xi_{z})}}=\mathbb{E}\left[\frac{\|v^{*}\|_{\Xi_{z}}}{\mathbb{E}\|\Xi_{z}^{1/2}H\|_{*}/\sqrt{n}}\right].
\end{equation*}
It holds from the Markov inequality that for any $\eta'>0$,
\begin{align}\label{Markov}
    \Pr\left(\frac{\|v^{*}\|_{\Xi_{z}}}{\mathbb{E}\|\Xi_{z}^{1/2}H\|_{*}/\sqrt{n}}>\sqrt{\eta'}\right)&\leq \frac{1}{\sqrt{\eta'}}\mathbb{E}\left[\frac{\|v^{*}\|_{\Xi_{z}}}{\mathbb{E}\|\Xi_{z}^{1/2}H\|_{*}/\sqrt{n}}\right] \notag \\
&=\frac{1}{\sqrt{\eta'}}\sqrt{\frac{n}{R_{\|\cdot\|}(\Xi_{z})}}.
\end{align}
As $n/R_{\|\cdot\|}(\Xi_{z})$ converges to zero in its limit, the left-hand side of (\ref{Markov}) can be arbitrarily small. Hence, we can pick up $\varepsilon_{1}$ such that 
\begin{align*}
    (1+\varepsilon_{1})\mathbb{E}\|v^{*}\|_{\Xi_{z}}=\sqrt{\eta'}\frac{\mathbb{E}\|\Xi_{z}^{1/2}H\|_{*}}{\sqrt{n}},
\end{align*}
which implies that
\begin{equation*}
     (1+\varepsilon_{1})^{2}\frac{n}{R_{\|\cdot\|}(\Xi_{z})}=\frac{n}{\left(\mathbb{E}\|\Xi_{z}^{1/2}H\|_{*}\right)^{2}}((1+\varepsilon_{1})\mathbb{E}\|v^{*}\|_{\Xi_{z}})^{2}=\eta'.
\end{equation*}
We have shown that $\gamma$, $\varepsilon$, and $(\|\theta_{0}\|+\|\Sigma_{u}^{+}\corrcoef\|)(\mathbb{E}\|\Xi_{z}^{1/2}H\|_{*}/\sqrt{n})$ are so small that (\ref{Variant_83}) holds for sufficiently large $n$. Therefore, we obtain
\begin{equation*}
    \Pr(\|\hat{\theta}-\theta_{0}\|_{\Xi_z}^{2}>\eta)\leq \delta
\end{equation*}
for any fixed $\eta$. As $\eta$ and $\delta$ are arbitrary, we have for any $\eta$, 
\begin{equation*}
    \lim_{n\rightarrow\infty}\Pr(\|\hat{\theta}-\theta_{0}\|_{\Xi_z}^{2}>\eta)=0.
\end{equation*}
Then, we obtain the statement.
\end{proof}

Second, we establish sufficient conditions of benign overfitting with the Euclidean norm.

\vspace{1\baselineskip}

\noindent
\textbf{Theorem \ref{Variant_Theorem3}} (Benign Overfitting) \textit{
Fix any $\delta\leq1/2$. Under Assumptions \ref{asmp:Gaussianity} and \ref{asmp:orthogonal} with covariance splitting $\Sigma_x=\Xi_z + \Sigma_{u}$, let $\gamma$ and $\varepsilon$ be as defined in Corollary \ref{Variant_Corollary2} and Theorem \ref{Variant_Theorem2}. Suppose that $n$ and the effective ranks are such that $R(\Xi_z)\gtrsim \log(1/\delta)^{2}$ and $\gamma,\varepsilon\leq1$. Then, with probability at least $1-\delta$,
\begin{align*}
    &\|\hat{\theta}-\theta_{0}\|_{\Xi_z}^{2}\leq(1+\gamma)(1+\varepsilon)\left(\tilde{\sigma}+(\|\Sigma_{u}^{+}\corrcoef\|_{2}+\|\theta_{0}\|_{2})\sqrt{\frac{\trace(\Xi_z)}{n}}\right)^{2}-\tilde{\sigma}^{2}.
\end{align*}
}

\begin{proof}[Proof of Theorem \ref{Variant_Theorem3}]
From the result of Theorem \ref{Variant_Theorem2}, if we adopt
\begin{equation*}
    B=\|\theta_{0}\|_{2}+\|\Sigma_{u}^{+}\corrcoef\|_{2}+(1+\varepsilon)^{1/2}\tilde{\sigma}\sqrt{\frac{n}{\trace(\Xi_z)}},
\end{equation*}
then $\{\theta:\|\theta\|_{2}\leq B\}\cap\{\theta: \mathbf{X}\theta=\mathbf{Y}\}$ is not empty with high probability. Clearly, $B>\|\theta_{0}\|_{2}$. This intersection necessarily includes the ridgeless estimator $\hat{\theta}$. Therefore, it holds from Corollary \ref{Variant_Corollary2} that
\begin{align*}
     \|\hat{\theta}-\theta_{0}\|_{\Xi_z}^{2} \leq&  \max_{\|\theta\|_{2}\leq B, \mathbf{Y}=\mathbf{X}\theta}\|\theta-\theta_{0}\|_{\Xi_z}^{2} \\
     \leq& (1+\gamma)\frac{B^{2}\trace(\Xi_z)}{n}-\tilde{\sigma}^{2} \\
     =&(1+\gamma)\left((\|\theta_{0}\|_{2}+\|\Sigma_{u}^{+}\corrcoef\|_{2})\sqrt{\frac{\trace(\Xi_z)}{n}}+(1+\varepsilon)^{1/2}\tilde{\sigma}\right)^{2}-\tilde{\sigma}^{2} \\
     \leq& (1+\gamma)(1+\varepsilon)\left(\tilde{\sigma}+(\|\theta_{0}\|_{2}+\|\Sigma_{u}^{+}\corrcoef\|_{2})\sqrt{\frac{\trace(\Xi_z)}{n}}\right)^{2}-\tilde{\sigma}^{2}.
\end{align*}
\end{proof}

\noindent
\textbf{Theorem \ref{Variant_Theorem12}} (Sufficient conditions) \textit{
Under Assumptions \ref{asmp:Gaussianity} and \ref{asmp:orthogonal}, let $\hat{\theta}$ be the ridgeless estimator. Suppose that as $n$ goes to $\infty$, the covariance splitting $\Sigma_x=\Xi_z + \Sigma_{u}$ satisfies  the following conditions:
\begin{enumerate}
    \item[(i)] (Small large-variance dimension.)
    \begin{equation*}
        \lim_{n\rightarrow\infty}\frac{\rank(\Sigma_{u})}{n}=0.
    \end{equation*}
    \item[(ii)] (Large effective dimension.)
    \begin{equation*}
        \lim_{n\rightarrow\infty}\frac{n}{R(\Xi_z)}=0.
    \end{equation*}
    \item [(iii)] (No aliasing condition.)
    \begin{equation*}
        \lim_{n\rightarrow\infty}\|\theta_{0}\|_{2}\sqrt{\frac{\trace(\Xi_z)}{n}}=0.
    \end{equation*}
    \item [(iv)] (Condition for the minimal interpolation of instrumental variable)
    \begin{equation*}
        \lim_{n\rightarrow\infty}\|\Sigma_{u}^{+}\corrcoef\|_{2}\sqrt{\frac{\trace(\Xi_z)}{n}}=0.
    \end{equation*}
\end{enumerate}
Then, $ \|\hat{\theta}-\theta_{0}\|_{\Xi_z}^{2}$ converges to $0$ in probability.
}

\begin{proof}[Proof of Theorem \ref{Variant_Theorem12}]

As in the proof of Theorem \ref{Variant_Theorem11}, we rearrange the upper bound derived in Theorem \ref{Variant_Theorem3} to clarify which terms should be sufficiently small for the projected RMSE to converge. 
By trivial calculation, we have
\begin{align}
    &(1+\gamma)(1+\varepsilon)\left(\tilde{\sigma}+(\|\theta_{0}\|_{2}+\|\Sigma_{u}^{+}\corrcoef\|_{2})\sqrt{\frac{\trace(\Xi_z)}{n}}\right)^{2}-\tilde{\sigma}^{2}  \notag \\
    &=(1+\gamma)(1+\varepsilon)\left(\left((\|\theta_{0}\|_{2}+\|\Sigma_{u}^{+}\corrcoef\|_{2})\sqrt{\frac{\trace(\Xi_z)}{n}}\right)^{2}\right.    +\left.2\tilde{\sigma}(\|\theta_{0}\|_{2}+\|\Sigma_{u}^{+}\corrcoef\|_{2})\sqrt{\frac{\trace(\Xi_z)}{n}}\right)  \notag \\
    & \quad +(\gamma+\varepsilon+\gamma\varepsilon)\tilde{\sigma}^{2} \notag \\
    &\leq(1+\gamma)(1+\varepsilon)\left(\left((\|\theta_{0}\|_{2}+\|\Sigma_{u}^{+}\corrcoef\|_{2})\sqrt{\frac{\trace(\Xi_z)}{n}}\right)^{2}\right.    +\left.2\sigma(\|\theta_{0}\|_{2}+\|\Sigma_{u}^{+}\corrcoef\|_{2})\sqrt{\frac{\trace(\Xi_z)}{n}}\right) \notag  \\
    & \quad +(\gamma+\varepsilon+\gamma\varepsilon)\sigma^{2}. \label{ineq:squared2}
\end{align}
The inequality follows  $\tilde{\sigma}^{2}\leq \sigma^{2}$ as  in the proof of Theorem \ref{Variant_Theorem11}.

 Fix any $\eta>0$ and $\delta>0$. From Lemma 5 of \citet{bartlett2020benign}, it holds that $R(\Xi_z)\leq r(\Xi_z)^{2}$. If $R(\Xi_z)=\orderomega(n)$ holds as the second condition in Theorem \ref{Variant_Theorem12}, we have $r(\Xi_z)=\orderomega(\sqrt{n})=\orderomega(1)$, which implies the convergence of $1/r(\Xi_z)$ to zero. Hence, conditions (i) and (ii) in Theorem \ref{Variant_Theorem3} make $\gamma$ and $\varepsilon$ sufficiently small for large enough $n$. $(\|\theta_{0}\|_{2}+\|\Sigma_{u}^{+}\corrcoef\|_{2})\sqrt{\trace(\Xi_z)/n}$ goes to zero from conditions (iii) and (iv). Hence, for sufficiently large $n$, we obtain that \eqref{ineq:squared2} is no more than $\eta$.
Therefore, we obtain
\begin{equation*}
    \Pr(\|\hat{\theta}-\theta_{0}\|_{\Xi_z}^{2}>\eta)\leq \delta,
\end{equation*}
for any fixed $\eta$. As $\eta$ and $\delta$ are arbitrary, we have for any $\eta$, 
\begin{equation*}
    \lim_{n\rightarrow\infty}\Pr(\|\hat{\theta}-\theta_{0}\|_{\Xi_z}^{2}>\eta)=0.
\end{equation*}
\end{proof}

\section{Non-Orthogonal Case} \label{app_sec:non_orthogonal}

We present the proof of the non-orthogonal case independently in this section because the case requires additional complicated analysis and is not a simple extension of the orthogonal case.

We present the results in Section \ref{sec:non-orthogonal} for the case where $\Xi_z$ and $\Sigma_{u}$ are non-orthogonal. 
In this section, we use $\Sigma_1 \in \R^{p \times p}$ and $\Sigma_2 \in \R^{p \times p}$ as notations for (potentially non-orthogonal) matrices as the statements in this section can be regarded as generic results for general matrices. 
In the setting for regression with endogeneity, these notations correspond to $\Xi_z$ and $\Sigma_u$, respectively.

First, we introduce auxiliary lemmas for this section.
\begin{lemma}[Corollary 2 in \citet{koehler2022uniform}]\label{Exogenous_Corollary2}
There exists an absolute constant $C_{1}\leq 66$ such that the following is true. Under Assumption \ref{asmp:Gaussianity} with covariance $\Sigma_x=\Sigma_{1}+\Sigma_{2}$, fix $\delta\leq 1/4$ and  let $\gamma=C_{1}(\sqrt{\log(1/\delta)/r(\Sigma_{2})}+\sqrt{\log(1/\delta)/n}+\sqrt{\rank(\Sigma_{1})/n})$. If $B\geq \|\theta_{0}\|_{2}$ and $n$ is large enough that $\gamma\leq 1$, the following holds with probability at least $1-\delta$:
\begin{equation}\label{Variant_5}
    \sup_{\|\theta\|_{2}\leq B, \hat{L}(\theta)=0}L(\theta)\leq (1+\gamma)\frac{B^{2}\trace(\Sigma_{2})}{n},
\end{equation}
where $\hat{L}(\theta)=\|\mathbf{Y}-\mathbf{X}\theta\|^{2}/n$ and $L(\theta):=E[(Y_1-X_1^\top\theta)]^{2}$. 
\end{lemma}

\begin{lemma}[Corollary 4 in \citet{koehler2022uniform}]\label{Variant_Corollary4}
Suppose $X_{1},\cdots,X_{n}\sim N(0,\Sigma)$ are independent with $\Sigma:p\times p$ a positive semidefinite matrix, $t>0$, and $n\geq 4(d+t^{2})$. Let $\hat{\Sigma}=\sum_{i}X_{i}X_{i}^\top/n$ be the empirical covariance matrix. Then, with probability at least $1-\delta$,
\begin{equation*}
    (1-\varepsilon)\Sigma\preceq\hat{\Sigma}\preceq(1+\varepsilon)\Sigma,
\end{equation*}
with $\varepsilon=3\sqrt{d/n}+3\sqrt{2\log(2/\delta)/n}$.
\end{lemma}

\begin{lemma}\label{Variant_10_Orthogonal}
Take any covariance matrix $\Sigma_{1},\Sigma_{2}$. If $\Sigma^{1/2}_{1}\Sigma^{1/2}_{2}\neq0$, it holds that with probability at least $1-\delta$,
\begin{equation}\label{Variant_76_Orthogonal}
    1-\frac{\|\Sigma^{1/2}_{1}\Sigma^{1/2}_{2}H\|^{2}_{2}}{\trace(\Sigma_{1}\Sigma_{2})}\lesssim \frac{\log(4/\delta)}{\sqrt{R(\Sigma_{1}\Sigma_{2})}}.
\end{equation}
Moreover, it holds that with probability at least $1-\delta$,
\begin{equation}\label{Variant_77_Orthogonal}
    \|\Sigma^{1/2}_{1}\Sigma^{1/2}_{2}H\|^{2}_{2}\lesssim\log(4/\delta)\trace(\Sigma_{1}\Sigma_{2}).
\end{equation}
Therefore, if $R(\Sigma_{2})\gtrsim log(4/\delta)^{2}$ holds, we have
\begin{equation}\label{Variant_78_Orthogonal}
    \left(\frac{\|\Sigma^{1/2}_{1}\Sigma^{1/2}_{2}H\|_{2}}{\|\Sigma^{1/2}_{2}H\|_{2}}\right)^{2}\lesssim \log(4/\delta)\frac{\trace(\Sigma_{1}\Sigma_{2})}{\trace(\Sigma_{2})}.
\end{equation}
\end{lemma}
\begin{proof}[Proof of Lemma \ref{Variant_10_Orthogonal}] As $\Sigma_{2}^{1/2}\Sigma_{1}\Sigma_{2}^{1/2}$ is a real symmetric matrix, there exists an orthogonal matrix $Q$ such that $Q\Sigma_{2}^{1/2}\Sigma_{1}\Sigma_{2}^{1/2}Q^\top$ is a diagonal matrix. Further, $QH$ has the normal standard distribution by the definition of $H$. Therefore, without loss of generality, $\Sigma_{2}^{1/2}\Sigma_{1}\Sigma_{2}^{1/2}$ can be considered as a diagonal matrix that consists of eigenvalues of $\Sigma_{2}^{1/2}\Sigma_{1}\Sigma_{2}^{1/2}$, $\lambda_{1},\cdots, \lambda_{p}$. By the sub-exponential Bernstein inequality (\citet{vershynin2018high}, Theorem 2.8.2), we have with probability at least $1-\delta/2$
\begin{align*}
    \left|\frac{\|\Sigma^{1/2}_{1}\Sigma^{1/2}_{2}H\|^{2}_{2}}{\trace(\Sigma_{1}\Sigma_{2})}-1\right|&=\left|\sum_{i=1}^{p}\frac{\lambda_{i}}{\sum_{k}\lambda_{k}}(H^{2}_{i}-1)\right| \\
    &\lesssim \sqrt{\frac{\log(4/\delta)}{R(\Sigma_{1}\Sigma_{2})}}\lor
    \frac{\log(4/\delta)}{r(\Sigma_{1}\Sigma_{2})}\leq \frac{\log(4/\delta)}{\sqrt{R(\Sigma_{1}\Sigma_{2})}},
\end{align*}
where the last inequality follows from the fact $R(\Sigma_{1}\Sigma_{2})\leq (r(\Sigma_{1}\Sigma_{2}))^{2}$. By definition, clearly $R(\Sigma_{1}\Sigma_{2})\geq 1$. Therefore, we have
\begin{equation*}
    \|\Sigma^{1/2}_{1}\Sigma^{1/2}_{2}H\|^{2}_{2}\lesssim\log(4/\delta)\trace(\Sigma_{1}\Sigma_{2}).
\end{equation*}
Provided $R(\Sigma_{2})$ is sufficiently large, we obtain
$\|\Sigma_{2}^{1/2}H\|^{2}_{2}\geq \frac{1}{2}\trace(\Sigma_{2})$.
Therefore, it holds that
\begin{equation*}
    \left(\frac{\|\Sigma^{1/2}_{1}\Sigma^{1/2}_{2}H\|_{2}}{\|\Sigma^{1/2}_{2}H\|_{2}}\right)^{2}\lesssim \log(4/\delta)\frac{\trace(\Sigma_{1}\Sigma_{2})}{\trace(\Sigma_{2})}.
\end{equation*}
\end{proof}

\subsection{When $X_{i}$ and $\xi_{i}$ are independent}

\begin{lemma}\label{Exogenous_Variant_Lemma8}
Denote $P$ as the projection matrix onto the space spanned by $\Sigma_{2}$. Let $v_{*}$ denote $\arg\min_{v\in\partial\|\Sigma_2^{1/2}H\|_{*}}\|v\|_{\Sigma_2}$. Assume that there exist $\varepsilon_{1},\varepsilon_{2}\ \text{and}\ \varepsilon_{3}\geq0$ such that with probability at least $1-\delta/4$,
\begin{equation}\label{Exogenous_Variant_65}
    \|v^{*}\|_{\Sigma_2}\leq(1+\varepsilon_{1})\mathbb{E}\|v^{*}\|_{\Sigma_2},
\end{equation}
\begin{equation}\label{Exogenous_Variant_66}
    \|Pv^{*}\|\leq 1+\varepsilon_{2},
\end{equation}
and
\begin{equation}\label{Exogenous_third}
     \|\Sigma_{1}^{1/2}Pv^{*}\|_{2}\leq (1+\varepsilon_{3})\Ep\|\Sigma_{1}^{1/2}Pv^{*}\|_{2}.
 \end{equation}
Define $\varepsilon$ as 
\begin{align*}
    \varepsilon:=&84\sqrt{\frac{\rank(\Sigma_{1})}{n}}+156\sqrt{\frac{\log(32/\delta)}{n}}+8\sqrt{\frac{\log (8/\delta)}{r_{\|\cdot\|}(\Sigma_{2})}}+2(1+\varepsilon_{1})^{2}\frac{n}{R_{\|\cdot\|}(\Sigma_{2})} \\
    &+2(1+\varepsilon_{3})^{2}\frac{n}{(\mathbb{E}\|\Sigma_{2}^{1/2}H\|_{*})^{2}}(\Ep\|\Sigma_{1}^{1/2}Pv^{*}\|_{2})^{2},
\end{align*}
where $r_{\|\cdot\|}(\Sigma)$ and $R_{\|\cdot\|}(\Sigma)$ are the effective ranks with general norms as provided in Definition \ref{def:effrank_gennorm}.
If $n$ and the effective ranks are sufficiently large such that $\varepsilon\leq1$, then with probability at least $1-\delta$, it holds that
\begin{equation}\label{Exogenous_Variant_67}
    \phi^{2}\leq(1+\varepsilon)\sigma^{2}\frac{n}{(\Ep\|\Sigma_{2}^{1/2}H\|_{*})^{2}}
\end{equation}
\end{lemma}

\begin{proof}[Proof of Lemma \ref{Exogenous_Variant_Lemma8}]
Denote $\alpha_{1},\alpha_{2}$, and $\alpha_{3}$ as follows:
\begin{align*}
\alpha_{1}&:=2\sqrt{\frac{\log(32/\delta)}{n}}, \\
\alpha_{2}&:=3\sqrt{\frac{\rank(\Sigma_{1})}{n}}+3\sqrt{\frac{2\log(16/\delta)}{n}}, \\
\alpha_{3}&:=\sqrt{\frac{\rank(\Sigma_{1})+1}{n}}+2\sqrt{\frac{\log(16/\delta)}{n}}. 
\end{align*}
To prepare for the derivation of the upper bound, we consider a list of the following inequalities, and each of these holds with probability at least $1-\delta/8$.
\begin{enumerate}
    \item[(i)]  By (\ref{Variant_26}) in Lemma \ref{Variant_Lemma1}, uniformly over all $\theta_{2}\in \Sigma_{1}^{1/2}(\mathbb{R}^{p})$, 
    it holds that
    \begin{equation}\label{Exogenous_Variant_49}
    |\langle\xi-\mathbf{W}_{2}\theta_{2}, G \rangle|\leq\|\xi-\mathbf{W}_{2}\theta_{2}\|_{2}\|G\|_{2}\alpha_{3}
\end{equation}
and 
\begin{equation}\label{Exogenous_Variant_50}
    |\langle\xi, \mathbf{W}_{2}\theta_{2} \rangle|\leq\|\xi\|_{2}\|\mathbf{W}_{2}\theta_{2}\|_{2}\alpha_{3}.
\end{equation}
For (\ref{Exogenous_Variant_49}), $V$ and $s$ in Lemma \ref{Variant_Lemma1} correspond to $G$ and $\xi-\mathbf{W}_{2}\theta_{2}$, respectively. For (\ref{Exogenous_Variant_50}), $V$ and $s$ in Lemma \ref{Variant_Lemma1} correspond to $\xi$ and $\mathbf{W}_{2}\theta_{2}$, respectively.
$\delta$ is replaced by $\delta/8$.

\item[(ii)] By Lemma \ref{Variant_Corollary4}, uniformly over all $\theta_{2}\in \Sigma_{1}^{1/2}(\mathbb{R}^{p})$, it holds that
\begin{equation}\label{Exogenous_Variant_51}
    (1-\alpha_{2})\|\theta_{2}\|^{2}_{2}\leq \frac{\|\mathbf{W}_{2}\theta_{2}\|^{2}_{2}}{n}\leq  (1+\alpha_{2})\|\theta_{2}\|^{2}_{2}.
\end{equation}
$\Sigma$ and $d$ in Lemma \ref{Variant_Corollary4} correspond to $\Sigma_{1}$ and  $\rank(\Sigma_{1})$ in (\ref{Exogenous_Variant_51}), respectively.

\item[(iii)] By Lemma \ref{Variant_Lemma2}, it holds that
\begin{equation}\label{Exogenous_Variant_52}
    -\alpha_{1}\leq\frac{1}{\sqrt{n}}\|G\|_{2}-1\leq\alpha_{1}
\end{equation}
and 
\begin{equation}\label{Exogenous_Variant_53}
    -\alpha_{1}\sigma\leq\frac{1}{\sqrt{n}}\|\xi\|_{2}-\sigma\leq\alpha_{1}\sigma.
\end{equation}

\item[(iv)] By Theorem \ref{Variant_Theorem6}, it holds that
\begin{align}
    \|\Sigma_2^{1/2}H\|_{*}&\geq\mathbb{E}\|\Sigma_2^{1/2}H\|_{*}-\sup_{\|u\|\leq1}\|u\|_{\Sigma_2}\sqrt{2\log(8/\delta)} \notag \\
    &=\left(1-\sqrt{\frac{2\log(8/\delta)}{r_{\|\cdot\|}(\Sigma_2)}}\right) \mathbb{E}\|\Sigma^{1/2}_{2}H\|_{*} \label{Exogenous_Variant_71}
\end{align}
because $\|\Sigma^{1/2}_{2}H\|_{*}$ is a $\sup_{\|u\|\leq1}\|u\|_{\Sigma_{2}}$-Lipschitz continuous function of $H$.

\end{enumerate}
We construct the upper bound from the restriction of the optimization problem \eqref{Variant_58}. It holds from (\ref{Exogenous_Variant_49}), (\ref{Exogenous_Variant_50}), and the AM-GM inequality that
\begin{align*}
    \|\xi-\mathbf{W}_{2}\theta_{2}-\|\theta_{1}\|_{2}G\|_{2}^{2}&\leq (1+\alpha_{3})( \|\xi-\mathbf{W}_{2}\theta_{2}\|^{2}_{2}+\|\theta_{1}\|_{2}^{2}\|G\|_{2}^{2}) \\
    &\leq (1+\alpha_{3})((1+\alpha_{3}) (\|\xi\|^{2}+\|\mathbf{W}_{2}\theta_{2}\|^{2}_{2})+\|\theta_{1}\|_{2}^{2}\|G\|_{2}^{2}).
\end{align*}
From the results of (\ref{Exogenous_Variant_51}) (\ref{Exogenous_Variant_52}), and (\ref{Exogenous_Variant_53}), we obtain
$\|\mathbf{W}_{2}\theta_{2}\|^{2}_{2}\leq n(1+\alpha_{2})\|\theta_{2}\|^{2}_{2}$, $\|\xi\|^{2}_{2}\leq(1+\alpha_{1})^{2}n\sigma^{2}$, and $
    \|G\|^{2}\leq (1+\alpha_{1})^{2}n$.
Therefore, we have
\begin{equation}\label{Exogenous_inequality}
     \|\xi-\mathbf{W}_{2}\theta_{2}-\|\theta_{1}\|_{2}G\|_{2}^{2}\leq n(1+\alpha_{1})^{2}(1+\alpha_{2})(1+\alpha_{3})^{2}(\sigma^{2}+\|\theta_{2}\|^{2}_{2}+\|\theta_{1}\|^{2}_{2}).
\end{equation}

To consider an upper bound of the ridgeless estimator, we need to choose a suitable $\theta$ which satisfies the restriction of the auxiliary problem. We define $\theta:=s(Pv^{*})$. Then, we have
    $\theta_{1}=s\Sigma_{2}^{1/2}v^{*}$ and 
    $\theta_{2}=s\Sigma_{1}^{1/2}Pv^{*}$.
If we consider the value of $s$ that satisfies the inequality:
\begin{equation*}
   n(1+\alpha_{1})^{2}(1+\alpha_{2})(1+\alpha_{3})^{2}(\sigma^{2}+\|s\Sigma_{1}^{1/2}Pv^{*}\|^{2}_{2}+\|s\Sigma_{2}^{1/2}v^{*}\|^{2}_{2})\leq s^{2}\langle H,\Sigma_{2}^{1/2} v^{*} \rangle^{2},
\end{equation*}
it holds from (\ref{Exogenous_inequality}) that 
\begin{equation*}
    \|\xi-\mathbf{W}_{2}\theta_{2}-\|\theta_{1}\|_{2}G\|^{2}_{2}\leq(\langle H,\theta_{1} \rangle)^{2}.
\end{equation*}
As $\langle H,\theta_{1} \rangle\geq 0$ by the definition of $\theta_{1}$, $\theta=s(Pv^{*})$ satisfies the restriction of the auxiliary problem in Lemma \ref{Variant_Lemma7}. Solving for $s$, we can select
\begin{equation*}
    s^{2}=\sigma^{2}\Biggl( ~\underbrace{\frac{\langle H,\Sigma^{1/2}_{2}v^{*} \rangle^{2}}{n(1+\alpha_{1})^{2}(1+\alpha_{2})(1+\alpha_{3})^{2}}-\|\Sigma^{1/2}_{1}Pv^{*}\|_{2}^{2}-\|\Sigma^{1/2}_{2}v^{*}\|_{2}^{2}}_{=: \Upsilon}~\Biggr)^{-1},
\end{equation*}
under the condition that $\Upsilon$ is positive.
We derive a lower bound of $\Upsilon$ as in the proof of Lemma \ref{Variant_Lemma8}: 
 \begin{align*}
     \Upsilon &= \frac{\|\Sigma_{2}^{1/2}H\|^{2}_{*}}{(1+\alpha_{1})^{2}(1+\alpha_{2})(1+\alpha_{3})^{2}n}-\|v^{*}\|^{2}_{\Sigma_{2}}-\|\Sigma^{1/2}_{1}Pv^{*}\|_{2}^{2} \\
     &\geq \frac{(\mathbb{E}\|\Sigma_{2}^{1/2}H\|_{*})^{2}}{n}\left(\frac{1}{(1+\alpha_{1})^{2}(1+\alpha_{2})(1+\alpha_{3})^{2}}\left(1-2\sqrt{\frac{2\log (8/\delta)}{r_{\|\cdot\|}(\Sigma_{2})}}\right)\right. \\
     &\qquad \left.-(1+\varepsilon_{1})^{2}\frac{n}{R_{\|\cdot\|}(\Sigma_{2})}-\frac{n}{(\mathbb{E}\|\Sigma_{2}^{1/2}H\|_{*})^{2}}\|\Sigma^{1/2}_{1}Pv^{*}\|_{2}^{2}\right) \\
     &\geq \frac{(\mathbb{E}\|\Sigma_{2}^{1/2}H\|_{*})^{2}}{n}\left(\frac{1}{(1+\alpha_{1})^{2}(1+\alpha_{2})(1+\alpha_{3})^{2}}\left(1-2\sqrt{\frac{2\log (8/\delta)}{r_{\|\cdot\|}(\Sigma_{2})}}\right)\right. \\
     &\qquad \left.-(1+\varepsilon_{1})^{2}\frac{n}{R_{\|\cdot\|}(\Sigma_{2})} -(1+\varepsilon_{3})^{2}\frac{n}{(\mathbb{E}\|\Sigma_{2}^{1/2}H\|_{*})^{2}}(\Ep\|\Sigma^{1/2}_{1}Pv^{*}\|_{2})^{2}\right).
 \end{align*}
 The equality holds by the definition of $R_{\|\cdot\|}(\Sigma_{2})$, the first inequality holds from (\ref{Exogenous_Variant_65}) and (\ref{Exogenous_Variant_71}), and the second inequality follows $\|\Sigma_{1}^{1/2}Pv^{*}\|_{2}\leq (1+\varepsilon_{3})\Ep\|\Sigma_{1}^{1/2}Pv^{*}\|_{2}$ by Assumption (\ref{Exogenous_third}).

We linearize the terms including $\alpha_{1}$, $\alpha_{2}$, and $\alpha_{3}$ to simplify the upper bound. If $\alpha_{1}<1$, $\alpha_{2}<1$, and $\alpha_{3}<1$, it holds that
 \begin{align*}
(1+\alpha_{1})^{2}(1+\alpha_{2})(1+\alpha_{3})^{2}&=(1+2\alpha_{1}+\alpha_{1}^{2})(1+\alpha_{2})(1+2\alpha_{3}+\alpha_{3}^{2}) \\
&\leq (1+3\alpha_{1})(1+\alpha_{2})(1+3\alpha_{3}) \\
  &= (1+3\alpha_{1}+\alpha_{2}+3\alpha_{1}\alpha_{2})(1+3\alpha_{3}) \\
   &\leq (1+3\alpha_{1}+4\alpha_{2})(1+3\alpha_{3}) \\
    &\leq (1+12\alpha_{1}+4\alpha_{2}+15\alpha_{3}).
 \end{align*}
 As $(1+x)^{-1}\geq (1-x)$ holds for any $x$, we have
 \begin{align*}
     &\left(\frac{1}{(1+\alpha_{1})^{2}(1+\alpha_{2})(1+\alpha_{3})^{2}}\left(1-2\sqrt{\frac{2\log (8/\delta)}{r_{\|\cdot\|}(\Sigma_{2})}}\right)\right. \\
     &\qquad \left.-(1+\varepsilon_{1})^{2}\frac{n}{R_{\|\cdot\|}(\Sigma_{2})}-(1+\varepsilon_{3})^{2}\frac{n}{(\mathbb{E}\|\Sigma_{2}^{1/2}H\|_{*})^{2}}(\Ep\|\Sigma^{1/2}_{1}Pv^{*}\|_{2})^{2}\right) \\
     &\geq 1-(12\alpha_{1}+4\alpha_{2}+15\alpha_{3})-2\sqrt{\frac{2\log (8/\delta)}{r_{\|\cdot\|}(\Sigma_{2})}} \\
     & \qquad -(1+\varepsilon_{1})^{2}\frac{n}{R_{\|\cdot\|}(\Sigma_{2})}- (1+\varepsilon_{3})^{2}\frac{n}{(\mathbb{E}\|\Sigma_{2}^{1/2}H\|_{*})^{2}}(\Ep\|\Sigma_{1}^{1/2}Pv^{*}\|_{2})^{2}\\
     &\geq 1-\varepsilon'
\end{align*}
where
\begin{align*}
    \varepsilon'&=42\sqrt{\frac{\rank(\Sigma_{1})}{n}}+78\sqrt{\frac{\log(32/\delta)}{n}}+4\sqrt{\frac{\log (8/\delta)}{r_{\|\cdot\|}(\Sigma_{2})}} \\
    & \quad +(1+\varepsilon_{1})^{2}\frac{n}{R_{\|\cdot\|}(\Sigma_{2})}+(1+\varepsilon_{3})^{2}\frac{n}{(\mathbb{E}\|\Sigma_{2}^{1/2}H\|_{*})^{2}}(\Ep\|\Sigma_{1}^{1/2}Pv^{*}\|_{2})^{2}
\end{align*}
If $\varepsilon'\leq 1/2$, because $(1-x)^{-1}\leq 1+2x$ for $x\in[0,1/2]$, it holds that 
\begin{align*}
s^{2}&=\sigma^{2}\Upsilon^{-1} \leq (1-\varepsilon')^{-1}\sigma^{2}\frac{n}{(\mathbb{E}\|\Sigma_{2}^{1/2}H\|_{*})^{2}} \leq (1+2\varepsilon')\sigma^{2}\frac{n}{(\mathbb{E}\|\Sigma_{2}^{1/2}H\|_{*})^{2}}.
\end{align*}
Therefore, we have
\begin{equation*}
    \phi^{2}\leq s^{2}\|Pv^{*}\|^{2}\leq(1+\varepsilon_{2})(1+2\varepsilon')\sigma^{2}\frac{n}{(\mathbb{E}\|\Sigma_{2}^{1/2}H\|_{*})^{2}}\leq (1+\varepsilon)\sigma^{2}\frac{n}{(\mathbb{E}\|\Sigma_{2}^{1/2}H\|_{*})^{2}}
\end{equation*}
with $\varepsilon=2\varepsilon'+2\varepsilon_{2}$.
\end{proof}

\begin{theorem}[General norm bound]\label{Exogenous_Variant_Theorem4}
There exists an absolute constant $C_{2}\leq 312$ such that the following
is true. Under Assumption \ref{asmp:Gaussianity} with covariance split $\Sigma_{x}=\Sigma_{1}+\Sigma_{2}$, let $\|\cdot\|$ be an
arbitrary norm, and fix $\delta\leq1/4$. Denote the $\ell_{2}$ orthogonal projection matrix onto the space spanned by $\Sigma_{2}$ as $P$. Let $H$ be normally distributed with mean zero and variance $I_{d}$, that is, $H\sim N(0,I_{d})$. Denote $v_{*}$ as $\arg\min_{v\in\partial\|\Sigma_2^{1/2}H\|_{*}} \|v\|_{\Sigma_2}$. Assume that there exists $\varepsilon_{1},\varepsilon_{2}\ \text{and}\ \varepsilon_{3}\geq0$ such that with probability at least $1-\delta/8$,
\begin{equation*}
    \|v^{*}\|_{\Sigma_2}\leq(1+\varepsilon_{1})\mathbb{E}\|v^{*}\|_{\Sigma_2},
\end{equation*}
\begin{equation*}
    \|Pv^{*}\|\leq 1+\varepsilon_{2},
\end{equation*}
and
\begin{equation*}
     \|\Sigma_{1}^{1/2}Pv^{*}\|_{2}\leq (1+\varepsilon_{3})\Ep\|\Sigma_{1}^{1/2}Pv^{*}\|_{2}.
 \end{equation*}
Define $\varepsilon$ as 
\begin{align*}
    \varepsilon := C_{2}\left(\sqrt{\frac{\rank(\Sigma_{1})}{n}}+\sqrt{\frac{\log(1/\delta)}{r_{\|\cdot\|}(\Sigma_{2})}}+\sqrt{\frac{\log(1/\delta)}{n}}+(1+\varepsilon_{1})^{2}\frac{n}{R_{\|\cdot\|}(\Sigma_{2})}\right. \\
    \left.+(1+\varepsilon_{3})^{2}\frac{n}{(\mathbb{E}\|\Sigma_{2}^{1/2}H\|_{*})^{2}}(\Ep\|\Sigma_{1}^{1/2}Pv^{*}\|_{2})^{2}+\varepsilon_{2}\right).
\end{align*}
If $n$ and the effective ranks are sufficiently large such that $\varepsilon\leq1$, then with probability at least $1-\delta$, it holds that
\begin{equation*}
    \|\hat{\theta}\|\leq\|\theta_{0}\|+(1+\varepsilon)^{1/2}\sigma\frac{\sqrt{n}}{\Ep\|\Sigma_{2}^{1/2}H\|_{*}}.
\end{equation*}
\end{theorem}

\begin{proof}[Proof of Theorem \ref{Exogenous_Variant_Theorem4}]
For any $t>0$, it holds from Lemmas \ref{Variant_Lemma6} and \ref{Variant_Lemma7} that
\begin{equation*}
\Pr(\|\hat{\theta}\|>t)\leq\Pr(\Phi>t-\|\theta_{0}\|)\leq2\Pr(\phi\geq t-\|\theta_{0}\|).
\end{equation*}
Lemma \ref{Exogenous_Variant_Lemma8} implies that the above term is upper bounded by $\delta$ if we choose $t-\|\theta_{0}\|$ using the result (\ref{Exogenous_Variant_67}) with $\delta$ replaced by $\delta/2$. We obtain the stated result by moving $\|\theta_{0}\|$ to the other side.
\end{proof}

When we consider the Euclidean space, we can reduce the upper bound of the ridgeless estimator to a simpler bound.
\begin{theorem}[Euclidean norm bound; special case of Theorem \ref{Exogenous_Variant_Theorem4}]\label{Exogenous_Variant_Theorem2}
Fix any $\delta\leq1/4$. Under Assumption \ref{asmp:Gaussianity} with covariance $\Sigma_x=\Sigma_{1}+\Sigma_{2}$, there exists some $\varepsilon\lesssim \sqrt{\rank(\Sigma_{1})/n}+\sqrt{\log(1/\delta)/r(\Sigma_{2})}+\sqrt{\log(1/\delta)/n}+n\log(1/\delta)/R(\Sigma_{2})\left(1+\trace(\Sigma_{1}\Sigma_{2})/\trace(\Sigma_{2}^{2})\right)$ such that the following is true. If $n$ and the effective ranks are sufficiently large such that $\varepsilon\leq1$ and $R(\Sigma_{2})\gtrsim \log(1/\delta)^{2}$, then with probability at least $1-\delta$, it holds that
\begin{equation}\label{Exogenous_Variant_6}
\|\hat{\theta}\|_{2}\leq\|\theta_{0}\|_{2}+(1+\varepsilon)^{1/2}\sigma\sqrt{\frac{n}{\trace(\Sigma_{2})}}.
\end{equation}
\end{theorem}

\begin{proof}[Proof of Theorem \ref{Exogenous_Variant_Theorem2}]
Throughout this proof, we simplify the upper bound, especially $
    n/R_{\|\cdot\|_{2}}(\Xi_z)$, $1/r_{\|\cdot\|_{2}}(\Xi_{z})$, and $(n\Ep\|\Sigma_{1}^{1/2}Pv^{*}\|_{2}^{2})/(\mathbb{E}\|\Sigma_{2}^{1/2}H\|_{*})^{2}$, in Theorem \ref{Exogenous_Variant_Theorem4}. By the definition of the dual norm and $\partial\|\Xi_{x}^{1/2}H\|_{*}$ with Euclidean norm, $v^{*}$ is equal to $\Sigma_2^{1/2}H/\|\Sigma_2^{1/2}H\|_{2}$. Hence, $\|v^{*}\|_{\Xi_z}$ is $\|\Sigma_2H\|_{2}/\|\Sigma_2^{1/2}H\|_{2}$. From the result of (\ref{Variant_78}), for some constant $c_{1}>0$, we can choose $\varepsilon_{1}$ such that
\begin{equation*}
    (1+\varepsilon_{1})E\|v^{*}\|_{\Sigma_2}=c_{1}\sqrt{\log(64/\delta)\frac{\trace(\Sigma_2^{2})}{\trace(\Sigma_2)}}.
\end{equation*}
If we assume effective rank is sufficiently large, (\ref{Variant_72}) provides that $\left(E\|\Sigma_2^{1/2}H\|_{2}\right)^{2}\gtrsim \trace(\Sigma_2)$. Therefore, we have
\begin{align*}
    (1+\varepsilon_{1})^{2}\frac{n}{R_{\|\cdot\|_{2}}(\Sigma_2)}&=n\frac{(1+\varepsilon_{1})^{2}(E\|v^{*}\|_{\Sigma_2})^{2}}{\left(E\|\Sigma_2^{1/2}H\|_{2}\right)^{2}}\\
    &\lesssim n\log(64/\delta)\frac{\trace(\Sigma_2^{2})}{\trace(\Sigma_2)^{2}}\\
    &=\frac{n\log(64/\delta)}{R(\Sigma_2)}.
\end{align*}
It also holds from (\ref{Variant_78_Orthogonal}) that for some constant $c_{2}>0$, there exists $\varepsilon_{3}$ such that 
\begin{equation*}
    (1+\varepsilon_{3})\Ep\|\Sigma_{1}^{1/2}Pv^{*}\|_{2}=c_{2}\sqrt{\log(64/\delta)\frac{\trace(\Sigma_{1}\Sigma_{2})}{\trace(\Sigma_{2})}}.
\end{equation*}
By (\ref{Variant_72}), for sufficiently large effective rank, it holds that $(E\|\Sigma_2^{1/2}H\|_{2})^{2}\gtrsim \trace(\Sigma_2)$. Therefore, we have
\begin{align*}
    (1+\varepsilon_{3})^{2}\frac{n}{(\mathbb{E}\|\Sigma_{2}^{1/2}H\|_{2})^{2}}(\Ep\|\Sigma_{1}^{1/2}Pv^{*}\|_{2})^{2}&\lesssim  n\log(64/\delta)\frac{\trace(\Sigma_1\Sigma_2)}{\trace(\Sigma_2)^{2}} \\
    &=\frac{n\log(64/\delta)}{R(\Sigma_2)}\frac{\trace(\Sigma_1\Sigma_2)}{\trace(\Sigma_2^{2})}.
\end{align*}

Finally, we obtain the upper bound of $\varepsilon$. As $P$ is an $l_{2}$ projection matrix, let $\varepsilon_{2}$ be zero. Then, it holds from (\ref{Variant_74}) of Lemma \ref{Variant_Lemma9}  that 
\begin{equation*}
\varepsilon\lesssim  \sqrt{\frac{\rank(\Sigma_{1})}{n}}+\sqrt{\frac{\log(1/\delta)}{n}}+ \sqrt{\frac{\log(1/\delta)}{r(\Sigma_{2})}}+\frac{n\log(1/\delta)}{R(\Sigma_{2})}\left(1+\frac{\trace(\Sigma_{1}\Sigma_{2})}{\trace(\Sigma_{2}^{2})}\right). 
\end{equation*}
By using the inequality $(1-x)^{-1}\leq1+2x$ for $x\in[0,1/2]$ and (\ref{Variant_72}) of Lemma \ref{Variant_Lemma9}, we finally obtain
\begin{align*}
    (1+\varepsilon)^{1/2}\sigma\frac{\sqrt{n}}{E\|\Sigma_{2}^{1/2}H\|_{2}}&\leq (1+\varepsilon)^{1/2}\left(1-\frac{1}{r(\Sigma_{2})}\right)^{-1/2}\sigma\sqrt{\frac{n}{\trace(\Sigma_2)}}  \\
    &\leq (1+\varepsilon)^{1/2}\left(1+\frac{2}{r(\Sigma_{2})}\right)^{1/2}\sigma\sqrt{\frac{n}{\trace(\Sigma_2)}} \\
    &\leq \left(1+2\varepsilon+\frac{2}{r(\Sigma_{2})}\right)^{1/2}\sigma\sqrt{\frac{n}{\trace(\Sigma_2)}},
\end{align*}
with $\varepsilon$ replaced by
\begin{align*}
    \varepsilon'&:=2\varepsilon+\frac{2}{r(\Sigma_{2})}\\
    &\lesssim \sqrt{\frac{\rank(\Sigma_{1})}{n}}+ \sqrt{\frac{\log(1/\delta)}{n}}+ \sqrt{\frac{\log(1/\delta)}{r(\Sigma_{2})}}+\frac{n\log(1/\delta)}{R(\Sigma_{2})}\left(1+\frac{\trace(\Sigma_{1}\Sigma_{2})}{\trace(\Sigma_{2}^{2})}\right).
\end{align*}
\end{proof}

\begin{theorem}[Benign Overfitting (Non-orthogonal)]\label{Exogenous_Variant_Theorem3}
Fix any $\delta\leq1/2$. Under Assumption \ref{asmp:Gaussianity} with covariance $\Sigma_x=\Sigma_{1}+\Sigma_{2}$, let $\gamma$ and $\varepsilon$ be as defined in Lemma \ref{Exogenous_Corollary2} and Theorem \ref{Exogenous_Variant_Theorem2}, respectively. Suppose also that $n$ and the effective ranks are such that $R(\Sigma_{2})\gtrsim \log(1/\delta)^{2}$ and $\gamma,\varepsilon\leq1$, then, with probability at least $1-\delta$, it holds that 
\begin{align*}
    L(\hat{\theta})\leq (1+\gamma)(1+\varepsilon)\left(\sigma+\|\theta_{0}\|_{2}\sqrt{\frac{\trace(\Sigma_{2})}{n}}\right)^{2},
\end{align*}
where we denote $L(\theta)$ as $\Ep(y-\langle\theta,x\rangle)^{2}$.
\end{theorem}

\begin{proof}[Proof of Theorem \ref{Exogenous_Variant_Theorem3}]
From the result of Theorem \ref{Exogenous_Variant_Theorem2}, if we adopt
\begin{equation*}
    B=\|\theta_{0}\|_{2}+(1+\varepsilon)^{1/2}\sigma\sqrt{\frac{n}{\trace(\Sigma_{2})}},
\end{equation*}
then $\{\theta:\|\theta\|_{2}\leq B\}\cap\{\theta: \mathbf{X}\theta=\mathbf{Y}\}$ is not empty with high probability. 
This intersection necessarily contains the ridgeless estimator $\hat{\theta}$. Clearly, $B>\|\theta_{0}\|_{2}$. Therefore, it holds from Lemma \ref{Exogenous_Corollary2} that
\begin{align*}
     L(\hat{\theta})&\leq \sup_{\|\theta\|_{2}\leq B, \hat{L}(\theta)=0}L(\theta) \\
     &\leq (1+\gamma)\left(\|\theta_{0}\|_{2}+(1+\varepsilon)^{1/2}\sigma\sqrt{\frac{n}{\trace(\Sigma_{2})}}\right)^{2}\frac{\trace(\Sigma_{2})}{n} \\
     &\leq (1+\gamma)(1+\varepsilon)\left(\sigma+\|\theta_{0}\|_{2}\sqrt{\frac{\trace(\Sigma_{2})}{n}}\right)^{2},
\end{align*}
where we denote $\hat{L}(\theta)$ as $\|\mathbf{Y}-\mathbf{X}\theta\|_{2}^{2}/n$.
\end{proof}

\noindent
\textbf{Theorem \ref{Exogenous_Variant_Theorem12}} (Sufficient conditions: Non-Orthogonal Case when $X_{i}$ and $\xi_{i}$ are independent) \textit{
Under Assumption \ref{asmp:Gaussianity}, let $\hat{\theta}$ be the ridgeless estimator. Suppose also that as $n$ goes to $\infty$, there exists a sequence of covariance $\Sigma_x=\Sigma_{1}+\Sigma_{2}$ such that the following conditions hold:
\begin{enumerate}
    \item[(i)] (Small large-variance dimension.)
    \begin{equation*}
        \lim_{n\rightarrow\infty}\frac{\rank(\Sigma_{1})}{n}=0.
    \end{equation*}
    \item[(ii)] (Large effective dimension.)
    \begin{equation*}
        \lim_{n\rightarrow\infty}\frac{n}{R(\Sigma_{2})}=0.
    \end{equation*}
    \item[(iii)] (No aliasing condition.)
    \begin{equation*}
        \lim_{n\rightarrow\infty}\|\theta_{0}\|_{2}\sqrt{\frac{\trace(\Sigma_{2})}{n}}=0.
    \end{equation*}
    \item[(iv)] (The cost of non-orthogonality)
    \begin{equation*}
        \lim_{n\rightarrow\infty}\frac{n}{R(\Sigma_{2})}\left(\frac{\trace(\Sigma_{1}\Sigma_{2})}{\trace(\Sigma_{2}^{2})}\right)=0.
    \end{equation*}
\end{enumerate}
Then, $ L(\hat{\theta})$ converges to $\sigma^{2}$ in probability.
}
\begin{proof}[Proof of Theorem \ref{Exogenous_Variant_Theorem12}]
Fix any $\eta>0$ and $\delta>0$. From Lemma 5 of \citet{bartlett2020benign}, it holds that $R(\Sigma_2)\leq r(\Sigma_2)^{2}$. If $R(\Sigma_2)=\orderomega(n)$ holds as the second condition in Theorem \ref{Exogenous_Variant_Theorem12}, we have $r(\Sigma_2)=\orderomega(\sqrt{n})=\orderomega(1)$, which implies  the convergence of $1/r(\Xi_z)$ to zero.  Hence, conditions (i) and (ii) in Theorem \ref{Exogenous_Variant_Theorem12} make $\gamma$ sufficiently small for large enough $n$. Clearly, $\|\theta_{0}\|_{2}\sqrt{\trace(\Sigma_{2})/n}$ goes to zero from condition (iii). By the definition of $\varepsilon$,  conditions (i) (ii), and (iv) in Theorem \ref{Exogenous_Variant_Theorem12} imply that $\varepsilon$ can be arbitrarily small. Therefore, for sufficiently large $n$, we obtain
 \begin{equation}\label{Exogenous_Variant_87}
   (1+\gamma)(1+\varepsilon)\left(\sigma+\|\theta_{0}\|_{2}\sqrt{\frac{\trace(\Sigma_{2})}{n}}\right)^{2}-\sigma^{2}\leq\eta.
\end{equation}

We have shown that $\gamma$, $\varepsilon$, and $\|\theta_{0}\|_{2}\sqrt{\trace(\Sigma_{2})/n}$ are so small that  equation (\ref{Exogenous_Variant_87}) holds for sufficiently large $n$. Therefore, we obtain
\begin{equation*}
    \Pr(|L(\hat{\theta})-\sigma^{2}|>\eta)\leq\delta
\end{equation*}
for any fixed $\eta$. As $\eta$ and $\delta$ are arbitrary, we have for any $\eta$, 
\begin{equation*}
    \lim_{n\rightarrow\infty}\Pr(|L(\hat{\theta})-\sigma^{2}|>\eta)=0.
\end{equation*}
\end{proof}

\subsection{When $X_{i}$ and $\xi_{i}$ are dependent}

Throughout this subsection, we assume $\tilde{\sigma}^{2}=\sigma^2 -\|\corrcoef\|^{2}_{\Sigma_{u}^{+}}>0$ holds.

\begin{lemma}\label{Non_ortho_Variant_Lemma8}
Denote $P_{z}$, $P_{u}$ as the projection matrix onto the space spanned by $\Xi_z$ and $\Sigma_{u}$, respectively. Let $v_{*}=\arg\min_{v\in\partial\|\Xi_z^{1/2}H\|}\|v\|_{\Xi_z}$. Assume that there exists $\varepsilon_{1}$, $\varepsilon_{2}$, and $\varepsilon_{3}\geq0$ such that with probability at least $1-\delta/4$,
\begin{equation}\label{Non_ortho_Variant_65}
    \|v^{*}\|_{\Xi_z}\leq(1+\varepsilon_{1})\mathbb{E}\|v^{*}\|_{\Xi_z},
\end{equation}
\begin{equation}\label{Non_ortho_Variant_66}
    \|P_{z}v^{*}\|\leq 1+\varepsilon_{2},
\end{equation}
and
\begin{equation}\label{Non_ortho_third}
     \|\Sigma_{u}^{1/2}P_{z}v^{*}\|_{2}\leq (1+\varepsilon_{3})\Ep\|\Sigma_{u}^{1/2}P_{z}v^{*}\|_{2}.
 \end{equation}
Denote $\varepsilon$ as 
\begin{align*}
    \varepsilon:=& 12\sqrt{\frac{\rank(\Sigma_{u})}{n}}+24\sqrt{\frac{\log(32/\delta)}{n}}+8\sqrt{\frac{\log(8/\delta)}{r_{\|\cdot\|}(\Xi_z)}}+2(1+\varepsilon_{1})^{2}\frac{n}{R_{\|\cdot\|}(\Xi_z)} \\
    &+64\frac{\|\Sigma_{u}^{+}\corrcoef\|_{2}}{\tilde{\sigma}}\frac{\mathbb{E}\|\Xi_z^{1/2}H\|_{2}}{\sqrt{n}}\left(1+\sqrt{\frac{2\log(8/\delta)}{r_{\|\cdot\|_{2}}(\Xi_z)}}\right)  \\
    &+2(1+\varepsilon_{3})^{2}\frac{n}{(\mathbb{E}\|\Xi_{z}^{1/2}H\|_{*})^{2}}(\Ep\|\Sigma_{u}^{1/2}P_{z}v^{*}\|_{2})^{2}+2\varepsilon_{2},
\end{align*}
where we denote $r_{\|\cdot\|}(\Sigma)$ and $R_{\|\cdot\|}(\Sigma)$ as follows:
\begin{align*}
    r_{\|\cdot\|}(\Sigma)=\left(\frac{E\|\Sigma^{1/2}H\|_{*}}{\sup_{\|u\|\leq1}\|u\|_{\Sigma}}\right)^{2}\ \text{and} \ R_{\|\cdot\|}(\Sigma)=\left(\frac{E\|\Sigma^{1/2}H\|_{*}}{E\|v^{*}\|_{\Sigma}}\right)^{2}.
\end{align*}
If $n$ and the effective ranks are sufficiently large such that $\varepsilon\leq1$, then with probability at least $1-\delta$, the AO defined in \eqref{Variant_58} is upper bounded as
\begin{equation}\label{Non_ortho_Variant_67}
    \phi^{2}\leq\left(\|\Sigma_{u}^{+}\corrcoef\|_{2}+(1+\varepsilon)(2\eta_{1}+\tilde{\sigma}+\eta_{2})\sqrt{\frac{n}{(\mathbb{E}\|\Xi_z^{1/2}H\|_{*})^{2}}}\right)^{2},
\end{equation}
where we denote $\eta_{1}$ and $\eta_{2}$ as follows:
\begin{align*}
    \eta_{1}&:=\sqrt{(1+\varepsilon_{1})^{2}\frac{n}{R_{\|\cdot\|}(\Xi_{z})}}\|\Xi_{z}^{1/2}\Sigma_{u}^{+}\corrcoef\|_{2},  \\
    \eta_{2}&:=\sqrt{\frac{(\mathbb{E}\|\Xi_z^{1/2}H\|_{2})^{2}}{n}\left(1+\sqrt{\frac{2\log(8/\delta)}{r_{\|\cdot\|_{2}}(\Xi_z)}}\right)^{2}\|\Sigma_{u}^{+}\corrcoef\|^{2}_{2}+\|\Xi_{z}^{1/2} \Sigma_{u}^{+}\corrcoef\|^{2}_{2}}.
\end{align*}
\end{lemma}

\begin{proof}[Proof of Lemma \ref{Non_ortho_Variant_Lemma8}]
This proof has four steps (i) preparation, (ii) introducing a coefficient $s$, (iii) deriving a bound on the coefficient $s$, and (iv) developing a bound on $\phi^2$ in \eqref{Variant_58}.

\textit{Step (i): Preparation}.
Denote $\alpha_{1}$ and $\alpha_{2}$ as follows:
\begin{align*}
\alpha_{1}&:=2\sqrt{\frac{\log(32/\delta)}{n}}, \\
\alpha_{2}&:=\sqrt{\frac{\rank(\Sigma_{u})+1}{n}}+2\sqrt{\frac{\log(16/\delta)}{n}}. 
\end{align*}
To prepare for the derivation of the upper bound as in the proof of Lemma \ref{Variant_Lemma5}, we consider the following three inequalities: 
\begin{enumerate}
 \item [(i)] By Lemma \ref{Variant_Lemma1}, uniformly over all $\theta_{2}\in \Sigma_{u}^{1/2}(\mathbb{R}^{p})$, 
    it holds that
    \begin{equation}
    |\langle\xi-\mathbf{W}_{2}\theta_{2}, G \rangle|\leq\|\xi-\mathbf{W}_{2}\theta_{2}\|_{2}\|G\|_{2}\alpha_{2}. \tag{\ref{Variant_49}}
\end{equation}
\item [(ii)] By Lemma \ref{Variant_Lemma2}, it holds that
\begin{equation} 
    -\alpha_{1}\leq\frac{1}{\sqrt{n}}\|G\|_{2}-1\leq\alpha_{1} \tag{\ref{Variant_52}}
\end{equation}
and 
\begin{align}
    -\alpha_{1}\sqrt{\sigma^{2}-2\rho^{T}\theta_{2}+\theta_{2}^{T}\theta_{2}}&\leq\frac{1}{\sqrt{n}}\|\xi-\mathbf{W}_{2}\theta_{2}\|_{2}-\sqrt{\sigma^{2}-2\rho^{T}\theta_{2}+\theta_{2}^{T}\theta_{2}} \notag \\
    &\leq\alpha_{1}\sqrt{\sigma^{2}-2\rho^{T}\theta_{2}+\theta_{2}^{T}\theta_{2}}. \label{Variant_70}
\end{align}

\item [(iii)] By Theorem \ref{Variant_Theorem6}, it holds that
\begin{align}
    \|\Xi_z^{1/2}H\|_{*}&\geq\mathbb{E}\|\Xi_z^{1/2}H\|_{*}-\sup_{\|u\|\leq1}\|u\|_{\Xi_{z}}\sqrt{2\log(8/\delta)} \notag \\
    &=\left(1-\sqrt{\frac{2\log(8/\delta)}{r_{\|\cdot\|}(\Xi_z)}}\right) \mathbb{E}\|\Xi_z^{1/2}H\|_{*} \tag{\ref{Variant_71}}
\end{align}
because $\|\Xi^{1/2}_{z}H\|_{*}$ is a $\sup_{\|u\|\leq1}\|u\|_{\Xi_{z}}$-Lipschitz continuous function of $H$. By Theorem \ref{Variant_Theorem6}, it also holds that
\begin{align}
    \|\Xi_z^{1/2}H\|_{*}&\leq\mathbb{E}\|\Xi_z^{1/2}H\|_{*}+\sup_{\|u\|\leq1}\|u\|_{\Xi_{z}}\sqrt{2\log(8/\delta)} \notag \\
    &=\left(1+\sqrt{\frac{2\log(8/\delta)}{r_{\|\cdot\|}(\Xi_z)}}\right) \mathbb{E}\|\Xi_z^{1/2}H\|_{*}. \label{Variant_71_2}
\end{align}
\end{enumerate}

\textit{Step (ii): Introducing the coefficient $s$}.
To derive an upper bound of the ridgeless estimator, we need to choose a suitable $\theta$ which satisfies the restriction of the auxiliary problem $\phi$ in Lemma \ref{Variant_Lemma7}.
We consider the following form of $\theta$:
\begin{align}
    \theta := P_{u}\Sigma_{u}^{+}\corrcoef+sP_{z}v^{*}. \label{decomp:theta}
\end{align}
Here, the coefficient $s$ describes a volume of $\theta$ along with the space spanned by $\Xi_z$.
By the setting, we have
$\theta_{1}=\Xi_{z}^{1/2}\Sigma_{u}^{+}\corrcoef+s\Xi_{z}^{1/2}v^{*}$ and     $\theta_{2}=\Sigma_{u}^{1/2}\Sigma_{u}^{+}\corrcoef+s\Sigma_{u}^{1/2}P_{z}v^{*}$.
Hence, we need to choose $s$ that attains the restriction of the auxiliary problem $\phi$, that is,
\begin{equation*}
    \|\xi-\mathbf{W}_{2}\theta_{2}-\|\theta_{1}\|_{2}G\|_{2}\leq \langle H, \theta_{1}\rangle. \tag{\ref{Variant_58}}
\end{equation*}
By the definition of $\theta_{1}$, we have the following result:
\begin{align}\label{odd_inequality}
&(\|\xi-\mathbf{W}_{2}\theta_{2}-\|\theta_{1}\|_{2}G\|_{2}-\langle H, \Xi_{z}^{1/2}\Sigma_{u}^{+}\corrcoef\rangle )^{2}\leq s^{2}\|\Xi_{z}^{1/2}H\|_{*}^{2} \\
\Rightarrow  &\|\xi-\mathbf{W}_{2}\theta_{2}-\|\theta_{1}\|_{2}G\|_{2}\leq \langle H, \theta_{1}\rangle. \notag
\end{align}
Therefore, it is sufficient to consider $s$ satisfying the inequality \eqref{odd_inequality}.

For the derivation of the inequality \eqref{odd_inequality}, we need to consider the upper bound of 
$(\|\xi-\mathbf{W}_{2}\theta_{2}-\|\theta_{1}\|_{2}G\|_{2}-\langle H, \Xi_{z}^{1/2}\Sigma_{u}^{+}\corrcoef\rangle )^{2}$. It holds from (\ref{Variant_49}) and the AM-GM inequality that
\begin{align*}
    \|\xi-\mathbf{W}_{2}\theta_{2}-\|\theta_{1}\|_{2}G\|_{2}^{2}&=\|\xi-\mathbf{W}_{2}\theta_{2}\|_{2}^{2}-2\|\theta_{1}\|_{2}\langle \xi-\mathbf{W}_{2}\theta_{2},G\rangle+\|\theta_{1}\|_{2}^{2}\|G\|_{2}^{2} \\
    &\geq \|\xi-\mathbf{W}_{2}\theta_{2}\|_{2}^{2}-2\alpha_{2}\|\theta_{1}\|_{2}\| \xi-\mathbf{W}_{2}\theta_{2}\|_{2}\|G\|_{2}+\|\theta_{1}\|_{2}^{2}\|G\|_{2}^{2} \\
    &\geq (1-\alpha_{2})\left(\|\xi-\mathbf{W}_{2}\theta_{2}\|_{2}^{2}+\|\theta_{1}\|_{2}^{2}\|G\|_{2}^{2}\right).
\end{align*}
Combining the results of (\ref{Variant_69}) and (\ref{Variant_70}) yields 
\begin{align*}
    &(1-\alpha_{2})\left(\|\xi-\mathbf{W}_{2}\theta_{2}\|_{2}^{2}+\|\theta_{1}\|_{2}^{2}\|G\|_{2}^{2}\right)\\
    &\geq (1-\alpha_{2})(1-\alpha_{1})^{2}n\left(\sigma^{2}-2\rho^{T}\theta_{2}+\theta_{2}^{T}\theta_{2}+\theta_{1}^{T}\theta_{1}\right) .
\end{align*}
Then, we have
\begin{align*}
    \left(\sigma^{2}-2\rho^{T}\theta_{2}+\theta_{2}^{T}\theta_{2}+\theta_{1}^{T}\theta_{1}\right)&=\left(\sigma^{2}-2\rho^{T}\Sigma_{u}^{1/2}\theta+\theta^{\top}\Sigma_{u}\theta+\theta^{T}\Xi_{z}\theta\right) \\
    &\geq \left(\sigma^{2}-2\rho^{T}\Sigma_{u}^{1/2}\theta+\theta^{\top}\Sigma_{u}\theta\right) \\
    &=\sigma^{2}-\rho^{\top}\rho+\|\Sigma_{u}^{1/2}\theta-\rho\|_{2}^{2} \\
    &\geq \sigma^{2}-\rho^{\top}\rho+\min_{\theta\in\mathbb{R}^{p}}\|\Sigma_{u}^{1/2}\theta-\rho\|_{2}^{2} \\
    &=\tilde{\sigma}^{2}>0.
\end{align*}
Therefore, we have
\begin{equation*}
    0<\frac{1}{ \|\xi-\mathbf{W}_{2}\theta_{2}-\|\theta_{1}\|_{2}G\|_{2}}\leq \frac{1}{\sqrt{(1-\alpha_{2})(1-\alpha_{1})^{2}n\tilde{\sigma}^{2}}}.
\end{equation*}
By trivial calculation, we obtain
\begin{align*}
    &(\|\xi-\mathbf{W}_{2}\theta_{2}-\|\theta_{1}\|_{2}G\|_{2}-\langle H, \Xi_{z}^{1/2}\Sigma_{u}^{+}\corrcoef\rangle )^{2} \\
    &\leq \|\xi-\mathbf{W}_{2}\theta_{2}-\|\theta_{1}\|_{2}G\|_{2}^{2}+2\|\xi-\mathbf{W}_{2}\theta_{2}-\|\theta_{1}\|_{2}G\|_{2}|\langle H, \Xi_{z}^{1/2}\Sigma_{u}^{+}\corrcoef\rangle|+(\langle H, \Xi_{z}^{1/2}\Sigma_{u}^{+}\corrcoef\rangle)^{2} \\
    &=\|\xi-\mathbf{W}_{2}\theta_{2}-\|\theta_{1}\|_{2}G\|_{2}^{2}\left(1+2\frac{|\langle H, \Xi_{z}^{1/2}\Sigma_{u}^{+}\corrcoef\rangle|}{\|\xi-\mathbf{W}_{2}\theta_{2}-\|\theta_{1}\|_{2}G\|_{2}}\right)+(\langle H, \Xi_{z}^{1/2}\Sigma_{u}^{+}\corrcoef\rangle)^{2} \\
    &\leq\|\xi-\mathbf{W}_{2}\theta_{2}-\|\theta_{1}\|_{2}G\|_{2}^{2}\left(1+2\frac{|\langle H, \Xi_{z}^{1/2}\Sigma_{u}^{+}\corrcoef\rangle|}{\sqrt{(1-\alpha_{2})(1-\alpha_{1})^{2}n\tilde{\sigma}^{2}}}\right)+(\langle H, \Xi_{z}^{1/2}\Sigma_{u}^{+}\corrcoef\rangle)^{2} \\
    &=\|\xi-\mathbf{W}_{2}\theta_{2}-\|\theta_{1}\|_{2}G\|_{2}^{2}(1+\gamma)+(\langle H, \Xi_{z}^{1/2}\Sigma_{u}^{+}\corrcoef\rangle)^{2}.
\end{align*}
Combining the results of (\ref{Variant_69}) and (\ref{Variant_70}) yields 
\begin{align*}
    \|\xi-\mathbf{W}_{2}\theta_{2}-\|\theta_{1}\|_{2}G\|_{2}^{2}
    &\leq (1+\alpha_{2})(1+\alpha_{1})^{2}n\left(\sigma^{2}-2\rho^{T}\theta_{2}+\theta_{2}^{T}\theta_{2}+\theta_{1}^{T}\theta_{1}\right) .
\end{align*}
Then, it holds that
\begin{align*}
    &(\|\xi-\mathbf{W}_{2}\theta_{2}-\|\theta_{1}\|_{2}G\|_{2}-\langle H, \Xi_{z}^{1/2}\Sigma_{u}^{+}\corrcoef\rangle )^{2} \\
    \leq &(1+\alpha_{2})(1+\alpha_{1})^{2}(1+\gamma)n\left(\sigma^{2}-2\rho^{T}\theta_{2}+\theta_{2}^{T}\theta_{2}+\theta_{1}^{T}\theta_{1}\right)+(\langle H, \Xi_{z}^{1/2}\Sigma_{u}^{+}\corrcoef\rangle)^{2} 
\end{align*}
Therefore, we should choose $s$ that satisfies the subsequent equality:
\begin{align}\label{s_equality}
    &s^{2}\|\Xi_{z}^{1/2}H\|^{2}_{*}\\
    &=(1+\alpha_{2})(1+\alpha_{1})^{2}(1+\gamma)n\left(\sigma^{2}-2\rho^{T}\theta_{2}+\theta_{2}^{T}\theta_{2}+\theta_{1}^{T}\theta_{1}\right)+(\langle H, \Xi_{z}^{1/2}\Sigma_{u}^{+}\corrcoef\rangle)^{2} . \notag 
\end{align}

We clarify $s$ that satisfies \eqref{s_equality}. By trivial calculation, we obtain the following results:
\begin{align*}
    &  s^{2}\|\Xi_{z}^{1/2}H\|^{2}_{*}\\
    &=(1+\alpha_{2})(1+\alpha_{1})^{2}(1+\gamma)n\left(\sigma^{2}-2\rho^{T}\theta_{2}+\theta_{2}^{T}\theta_{2}+\theta_{1}^{T}\theta_{1}\right)+(\langle H, \Xi_{z}^{1/2}\Sigma_{u}^{+}\corrcoef\rangle)^{2} \\
    \Leftrightarrow&\beta_{1}\left(s+\frac{\beta_{2}}{\beta_{1}}\right)^{2}-\frac{\beta_{2}^{2}}{\beta_{1}}-(\tilde{\sigma}^{2}-\beta_{3})=0,
\end{align*}
where we define
\begin{align*}
    \beta_{1}&:=\left(\frac{\|\Xi_{z}^{1/2}H\|_{*}^{2}}{(1+\alpha_{2})(1+\alpha_{1})^{2}(1+\gamma)n}-\|v^{*}\|_{\Xi_{z}}^{2}-(v^{*})^\top P_{z}\Sigma_{u} P_{z}v^{*}\right), \\
    \beta_{2}&:=\left(\rho^{\top}\Sigma_{u}^{1/2}P_{z}v^{*}-\corrcoef^\top\Sigma_{u}^{+}\Sigma_{u}P_{z}v^{*}-v^{*}\Xi_{z}\Sigma_{u}^{+}\corrcoef\right), \\
    \beta_{3}&:=-\frac{(\langle H, \Xi_{z}^{1/2}\Sigma_{u}^{+}\corrcoef)\rangle)^{2}}{(1+\alpha_{2})(1+\alpha_{1})^{2}(1+\gamma)n}-(\corrcoef^{\top}\Sigma_{u}^{+}\Xi_{z} \Sigma_{u}^{+}\corrcoef), \\
    \gamma&:=2\frac{|\langle H, \Xi_{z}^{1/2}\Sigma_{u}^{+}\corrcoef\rangle|}{\sqrt{(1-\alpha_{2})(1-\alpha_{1})^{2}n\tilde{\sigma}^{2}}}.
\end{align*}
Therefore, we choose $s$ such that 
\begin{equation*}
    s=\left(-\frac{\beta_{2}}{\beta_{1}}+\sqrt{\frac{\beta_{2}^{2}}{\beta_{1}^{2}}+\frac{\tilde{\sigma}^{2}-\beta_{3}}{\beta_{1}}}\right),
\end{equation*}
under the assumption that 
\begin{equation}\label{positive_non}
     \beta_{1}>0 \quad\text{and}\quad \tilde{\sigma}^{2}-\beta_{3}\geq 0.
\end{equation}

We need to guarantee (\ref{positive_non}) holds. First, we prove 
    $\tilde{\sigma}^{2}-\beta_{3}\geq0$.
By definition, we have $\tilde{\sigma}^{2}>0$ and $-\beta_{3}\geq0$.
Second, we show $\beta_{1}>0$. 
By (\ref{Non_ortho_Variant_65}), (\ref{Non_ortho_third}), and (\ref{Variant_71}), we have
\begin{align*}
    \beta_1 \geq & \frac{(\mathbb{E}\|\Xi_z^{1/2}H\|_{*})^{2}}{(1+\alpha_{2})(1+\alpha_{1})^{2}(1+\gamma)n}\left(1-\sqrt{\frac{2\log(8/\delta)}{r_{\|\cdot\|}(\Xi_z)}}\right)^{2} \\
    & \quad -(1+\varepsilon_{1})^{2}(\mathbb{E}\|v^{*}\|_{\Xi_{z}})^{2}-(1+\varepsilon_{3})^{2}(\Ep\|\Sigma_{u}^{1/2}P_{z}v^{*}\|_{2})^{2} \\
    \geq & \frac{(\mathbb{E}\|\Xi_z^{1/2}H\|_{*})^{2}}{n}\left(\frac{1}{(1+\alpha_{2})(1+\alpha_{1})^{2}(1+\gamma)}\left(1-2\sqrt{\frac{2\log(8/\delta)}{r_{\|\cdot\|}(\Xi_z)}}\right)
    \right. \\
   &\left. \quad - (1+\varepsilon_{1})^{2}\frac{n}{R_{\|\cdot\|}(\Sigma_{2})} - (1+\varepsilon_{3})^{2}\frac{n}{(\mathbb{E}\|\Xi_{z}^{1/2}H\|_{*})^{2}}(\Ep\|\Sigma_{u}^{1/2}P_{z}v^{*}\|_{2})^{2}\right).
\end{align*}

We linearize the terms including $\alpha_{1}$, $\alpha_{2}$, and $\gamma$ to simplify the upper bound. Provided that $\alpha_{1},\alpha_{2}<1/2$, we have
\begin{align*}
    (1+\alpha_{2})(1+\alpha_{1})^{2}(1+\gamma)&=(1+2\alpha_{1}+\alpha_{1}^{2}+\alpha_{2}+2\alpha_{2}\alpha_{1}+\alpha_{2}\alpha_{1}^{2})(1+\gamma) \\
    &\leq (1+3\alpha_{1}+3\alpha_{2})(1+\gamma) \\
    &\leq 1+3\alpha_{1}+3\alpha_{2}+4\gamma.
\end{align*}
As $(1-x)^{-1}\geq 1+x$ for any $x$, it holds that
\begin{align*}
\frac{1}{(1+\alpha_{2})(1+\alpha_{1})^{2}(1+\gamma)}&\geq ( 1+3\alpha_{1}+3\alpha_{2}+4\gamma)^{-1} \\
&\geq( 1-(3\alpha_{1}+3\alpha_{2}+4\gamma)).
\end{align*}
Moreover, by the Cauchy-Schwarz inequality, we have
\begin{align*}
    \gamma\leq \frac{2\|\Xi_{z}^{1/2}H\|_{2}\|\|\Sigma_{u}^{+}\corrcoef\|_{2}}{\sqrt{(1-\alpha_{2})(1-\alpha_{1})^{2}n\tilde{\sigma}^{2}}}.
\end{align*}
As $(1-x)^{-1}\leq 1+2x$ for $x\in [0,1/2]$, it holds that
\begin{equation*}
    \frac{1}{\sqrt{(1-\alpha_{2})(1-\alpha_{1})^{2}}}\leq \frac{1}{(1-\alpha_{2})(1-\alpha_{1})}\leq 4.
\end{equation*}
By (\ref{Variant_71_2}), we have
\begin{equation*}
    \frac{\|\Xi_{z}^{1/2}H\|_{2}}{\sqrt{n}}\leq  \frac{\mathbb{E}\|\Xi_z^{1/2}H\|_{2}}{\sqrt{n}}\left(1+\sqrt{\frac{2\log(8/\delta)}{r_{\|\cdot\|_{2}}(\Xi_z)}}\right).
\end{equation*}
Therefore, it holds that
\begin{equation*}
    \gamma \leq \frac{8\|\Sigma_{u}^{+}\corrcoef\|_{2}}{\tilde{\sigma}}\frac{\mathbb{E}\|\Xi_z^{1/2}H\|_{2}}{\sqrt{n}}\left(1+\sqrt{\frac{2\log(8/\delta)}{r_{\|\cdot\|_{2}}(\Xi_z)}}\right).
\end{equation*}
Hence,  we have
\begin{align*}
    &\frac{1}{(1+\alpha_{2})(1+\alpha_{1})^{2}(1+\gamma)}\left(1-2\sqrt{\frac{2\log(8/\delta)}{r_{\|\cdot\|}(\Xi_z)}}\right) \\
    &\quad -(1+\varepsilon_{1})^{2}\frac{n}{R_{\|\cdot\|}(\Sigma_{2})}-(1+\varepsilon_{3})^{2}\frac{n}{(\mathbb{E}\|\Xi_{z}^{1/2}H\|_{*})^{2}}(\Ep\|\Sigma_{u}^{1/2}P_{z}v^{*}\|_{2})^{2} \\
    &\geq (1-(3\alpha_{1}+3\alpha_{2}+4\gamma))\left(1-2\sqrt{\frac{2\log(8/\delta)}{r_{\|\cdot\|}(\Xi_z)}}\right)
    -(1+\varepsilon_{1})^{2}\frac{n}{R_{\|\cdot\|}(\Sigma_{2})} \\
    & \quad -(1+\varepsilon_{3})^{2}\frac{n}{(\mathbb{E}\|\Xi_{z}^{1/2}H\|_{*})^{2}}(\Ep\|\Sigma_{u}^{1/2}P_{z}v^{*}\|_{2})^{2} \\
    &\geq 1-(3\alpha_{1}+3\alpha_{2}+4\gamma)-2\sqrt{\frac{2\log(8/\delta)}{r_{\|\cdot\|}(\Xi_z)}}
    -(1+\varepsilon_{1})^{2}\frac{n}{R_{\|\cdot\|}(\Sigma_{2})} \\
    & \quad -(1+\varepsilon_{3})^{2}\frac{n}{(\mathbb{E}\|\Xi_{z}^{1/2}H\|_{*})^{2}}(\Ep\|\Sigma_{u}^{1/2}P_{z}v^{*}\|_{2})^{2} \\
    &\geq 1-\varepsilon'
\end{align*}
where we define
\begin{align*}
    \varepsilon'=&6\sqrt{\frac{\rank(\Sigma_{u})}{n}}+12\sqrt{\frac{\log(32/\delta)}{n}}+32\frac{\|\Sigma_{u}^{+}\corrcoef\|_{2}}{\tilde{\sigma}}\frac{\mathbb{E}\|\Xi_z^{1/2}H\|_{2}}{\sqrt{n}}\left(1+\sqrt{\frac{2\log(8/\delta)}{r_{\|\cdot\|_{2}}(\Xi_z)}}\right) \\
    +&4\sqrt{\frac{\log(8/\delta)}{r_{\|\cdot\|}(\Xi_z)}}+(1+\varepsilon_{1})^{2}\frac{n}{R_{\|\cdot\|}(\Xi_z)}+(1+\varepsilon_{3})^{2}\frac{n}{(\mathbb{E}\|\Xi_{z}^{1/2}H\|_{*})^{2}}(\Ep\|\Sigma_{u}^{1/2}P_{z}v^{*}\|_{2})^{2}.
\end{align*}
As $\varepsilon'$ is assumed to be less than $1/2$ and $(\mathbb{E}\|\Xi_z^{1/2}H\|_{*})^{2}/n$ is positive, we have
\begin{align}\label{beta1_lowerbound}
    0&<(1-\varepsilon')\frac{(\mathbb{E}\|\Xi_z^{1/2}H\|_{*})^{2}}{n}
    \leq  \beta_1.
\end{align}

\textit{Step (iii): Bound on the coefficient $s$}.
As $s$ is too complicated, we need to obtain a simplified upper bound of $s$. By trivial calculation, we have
\begin{equation*}
    s\leq \frac{2|\beta_{2}|}{\beta_{1}}+\sqrt{\frac{\tilde{\sigma}^{2}}{\beta_{1}}}+\sqrt{\frac{|\beta_{3}|}{\beta_{1}}}.
\end{equation*}
Then, we consider an upper bound of $1/\beta_{1}$. It holds from (\ref{beta1_lowerbound}) that
\begin{align}\label{beta1}
    \frac{1}{\beta_{1}}&\leq (1-\varepsilon')^{-1}\frac{n}{(\mathbb{E}\|\Xi_z^{1/2}H\|_{*})^{2}} \leq (1+2\varepsilon')\frac{n}{(\mathbb{E}\|\Xi_z^{1/2}H\|_{*})^{2}},
\end{align}
where the last inequality holds because $(1-x)^{-1}\leq 1+2x$ for $x\in [0,1/2]$.

Next, we show  an upper bound of $|\beta_{2}|$. As $\corrcoef:=\Sigma_{u}^{1/2}\rho$ and $\corrcoef^\top\Sigma_{u}^{+}\Sigma_{u}=\corrcoef^\top(\Sigma_{u}^{1/2})^{+}\Sigma_{u}^{1/2}$, we have
\begin{align*}
    \beta_{2}=&\left(\rho^{\top}\Sigma_{u}^{1/2}P_{z}v^{*}-\corrcoef^\top\Sigma_{u}^{+}\Sigma_{u}P_{z}v^{*}-v^{*}\Xi_{z}\Sigma_{u}^{+}\corrcoef\right) \\
    =&\left(\corrcoef^\top(\Sigma_{u}^{1/2})^{+}\Sigma_{u}^{1/2} P_{z}v^{*}-\corrcoef^\top \Sigma_{u}^{+}\Sigma_{u}P_{z}v^{*}-v^{*}\Xi_{z}\Sigma_{u}^{+}\corrcoef\right) \\
=&-v^{*}\Xi_{z}\Sigma_{u}^{+}\corrcoef.
\end{align*}
By the Cauchy-Schwarz inequality, it holds that
$    |\beta_{2}|\leq \|v^{*}\|_{\Xi_{z}}\|\Xi_{z}^{1/2}\Sigma_{u}^{+}\corrcoef\|_{2}$.
From assumption (\ref{Non_ortho_Variant_65}) and the definition of $R_{\|\cdot\|}(\Xi_{z})$, we obtain
\begin{equation}\label{beta2}
    |\beta_{2}|\leq (1+\varepsilon_{1})\sqrt{\frac{(\mathbb{E}\|\Xi_z^{1/2}H\|_{*})^{2}}{n}}\sqrt{\frac{n}{R_{\|\cdot\|}(\Xi_{z})}}\|\Xi_{z}^{1/2}\Sigma_{u}^{+}\corrcoef\|_{2}.
\end{equation}
Combining the result (\ref{beta2}) with (\ref{beta1}), we have
\begin{equation}\label{first_term}
    \frac{2|\beta_{2}|}{\beta_{1}}\leq 2(1+2\varepsilon')(1+\varepsilon_{1})\sqrt{\frac{n}{(\mathbb{E}\|\Xi_z^{1/2}H\|_{*})^{2}}}\sqrt{\frac{n}{R_{\|\cdot\|}(\Xi_{z})}}\|\Xi_{z}^{1/2}\Sigma_{u}^{+}\corrcoef\|_{2}.
\end{equation}

Finally, we derive an upper bound of $\beta_{3}$ and $s$. By using the triangular inequality on $\beta_{3}$, we have
\begin{equation}\label{beta3_1}
    |\beta_{3}|\leq \frac{(\langle H, \Xi_{z}^{1/2}\Sigma_{u}^{+}\corrcoef\rangle)^{2}}{(1+\alpha_{2})(1+\alpha_{1})^{2}n}+\|\Xi_{z}^{1/2}\Sigma_{u}^{+}\corrcoef\|^{2}_{2}.
\end{equation}
By the Cauchy-Schwarz inequality, it holds that
\begin{equation}\label{beta3_2}
    \frac{(\langle H, \Xi_{z}^{1/2}\Sigma_{u}^{+}\corrcoef\rangle)^{2}}{(1+\alpha_{2})(1+\alpha_{1})^{2}n}\leq \frac{\|\Xi_{z}^{1/2}H\|_{2}^{2}}{(1+\alpha_{2})(1+\alpha_{1})^{2}n}\|\Sigma_{u}^{+}\corrcoef\|_{2}^{2}.
\end{equation}
By (\ref{Variant_71_2}), we have
\begin{equation}\label{beta3_3}
    \frac{\|\Xi_{z}^{1/2}H\|_{2}^{2}}{(1+\alpha_{2})(1+\alpha_{1})^{2}n}\leq  \frac{(\mathbb{E}\|\Xi_z^{1/2}H\|_{2})^{2}}{n}\frac{1}{(1+\alpha_{2})(1+\alpha_{1})^{2}}\left(1+\sqrt{\frac{2\log(8/\delta)}{r_{\|\cdot\|_{2}}(\Xi_z)}}\right)^{2}.
\end{equation}
Hence, it holds from (\ref{beta1}), (\ref{beta3_1}), (\ref{beta3_2}), and (\ref{beta3_3}) that
\begin{align}\label{third_term}
    \sqrt{\frac{|\beta_{3}|}{\beta_{1}}}&\leq  \sqrt{(1+2\varepsilon')}\sqrt{\frac{n}{(\mathbb{E}\|\Xi_z^{1/2}H\|_{*})^{2}}} \notag \\ 
    &\times
    \sqrt{\frac{(\mathbb{E}\|\Xi_z^{1/2}H\|_{2})^{2}}{n}\left(1+\sqrt{\frac{2\log(8/\delta)}{r_{\|\cdot\|_{2}}(\Xi_z)}}\right)^{2}\|\Sigma_{u}^{+}\corrcoef\|^{2}_{2}+\|\Xi_{z}^{1/2} \Sigma_{u}^{+}\corrcoef\|^{2}_{2}}.
\end{align}
From (\ref{first_term}) and (\ref{third_term}), we obtain
\begin{equation}\label{upperbound_of_s}
    s\leq \sqrt{\frac{n}{(\mathbb{E}\|\Xi_z^{1/2}H\|_{*})^{2}}}A 
\end{equation}
where $A$ is defined and bounded as follows:
\begin{align}
    A&:=2(1+2\varepsilon')(1+\varepsilon_{1})\sqrt{\frac{n}{R_{\|\cdot\|}(\Xi_{z})}}\|\Xi_{z}^{1/2}\Sigma_{u}^{+}\corrcoef\|_{2} +\sqrt{(1+2\varepsilon')\tilde{\sigma}^{2}} \notag \\
    &\quad +\sqrt{(1+2\varepsilon')}
    \sqrt{\frac{(\mathbb{E}\|\Xi_z^{1/2}H\|_{2})^{2}}{n}\left(1+\sqrt{\frac{2\log(8/\delta)}{r_{\|\cdot\|_{2}}(\Xi_z)}}\right)^{2}\|\Sigma_{u}^{+}\corrcoef\|^{2}_{2}+\|\Xi_{z}^{1/2} \Sigma_{u}^{+}\corrcoef\|^{2}_{2}} \notag \\
    &\leq 2(1+2\varepsilon')\eta_1 + \sqrt{(1 + 2 \varepsilon')}\tilde{\sigma} +  (1+2\varepsilon')\eta_2, \label{def:bound_A}
\end{align}
where 
\begin{align*}
     \eta_{1}&:=\sqrt{(1+\varepsilon_{1})^{2}\frac{n}{R_{\|\cdot\|}(\Xi_{z})}}\|\Xi_{z}^{1/2}\Sigma_{u}^{+}\corrcoef\|_{2},
\end{align*}
and 
\begin{align*}
    \eta_{2}&:=\sqrt{\frac{(\mathbb{E}\|\Xi_z^{1/2}H\|_{2})^{2}}{n}\left(1+\sqrt{\frac{2\log(8/\delta)}{r_{\|\cdot\|_{2}}(\Xi_z)}}\right)^{2}\|\Sigma_{u}^{+}\corrcoef\|^{2}_{2}+\|\Xi_{z}^{1/2} \Sigma_{u}^{+}\corrcoef\|^{2}_{2}}.
\end{align*}

\textit{Step (iv): Bound $\phi^2$.}
We simplify the upper bound \eqref{upperbound_of_s} and derive the upper bound of $\phi^{2}$. By trivial calculation, we utilize the definition of $\theta$ as \eqref{decomp:theta} and obtain
\begin{align*}
    \phi^2 &\leq \|\theta\|^{2} \\
    &\leq \|P_{u}\Sigma_{u}^{+}\corrcoef+sP_{z}v^{*}\|^{2} \\
        &\leq (\|P_{u}\Sigma_{u}^{+}\corrcoef\|+s\|P_{z}v^{*}\|)^{2} \\
            &\leq (\|\Sigma_{u}^{+}\corrcoef\|+s(1+\varepsilon_{2}))^{2} \\
        &\leq \left(\|\Sigma_{u}^{+}\corrcoef\|+\sqrt{\frac{n}{(\mathbb{E}\|\Xi_z^{1/2}H\|_{*})^{2}}}A (1+\varepsilon_{2})\right)^{2}.
\end{align*}
The second to last inequality follows the assumption in \eqref{Non_ortho_Variant_66}, and the last inequality follows the upper bound on $s$ as \eqref{upperbound_of_s}.
We apply \eqref{def:bound_A} and set $\varepsilon:=2(\varepsilon'+\varepsilon_{2})$, and then 
it holds that
\begin{equation*}
    \phi^{2}\leq\left(\|\Sigma_{u}^{+}\corrcoef\|+(1+\varepsilon)(2\eta_{1}+\tilde{\sigma}+\eta_{2})\sqrt{\frac{n}{(\mathbb{E}\|\Xi_z^{1/2}H\|_{*})^{2}}}\right)^{2}.
\end{equation*}
\end{proof}

\begin{theorem}[General norm bound]\label{Non_ortho_Variant_Theorem4}
There exists an absolute constant $C_{2}\leq 160$ such that the following
is true. Under Assumption \ref{asmp:Gaussianity} with $\Sigma_{x}=\Xi_z+\Sigma_{u}$, let $\|\cdot\|$ be an
arbitrary norm, and fix $\delta\leq1/4$. Denote the $\ell_{2}$ orthogonal projection matrix onto the space spanned by $\Xi_z$ and $\Sigma_{u}$ as $P_{z}$ and $P_{u}$, respectively. Let $H$ be normally distributed with mean zero and variance $I_{d}$, that is, $H\sim N(0,I_{d})$. Denote $v_{*}$ as $\arg\min_{v\in\partial\|\Xi_z^{1/2}H\|_{*}} \|v\|_{\Xi_z}$. Suppose that there exist $\varepsilon_{1}\ \text{and}\ \varepsilon_{2}\geq0$ such that with probability at least $1-\delta/8$,
\begin{equation*}
    \|v^{*}\|_{\Xi_z}\leq(1+\varepsilon_{1})\mathbb{E}\|v^{*}\|_{\Xi_z},
\end{equation*}
\begin{equation*}
    \|P_{z}v^{*}\|\leq 1+\varepsilon_{2},
\end{equation*}
and
\begin{equation*}
     \|\Sigma_{u}^{1/2}P_{z}v^{*}\|_{2}\leq (1+\varepsilon_{3})\Ep\|\Sigma_{u}^{1/2}P_{z}v^{*}\|_{2}.
 \end{equation*}
Denote $\varepsilon$ as follows:
\begin{align*}
    \varepsilon:=&C_{2}\left(\sqrt{\frac{\rank(\Sigma_{u})}{n}}+\sqrt{\frac{\log(1/\delta)}{r_{\|\cdot\|_{2}}(\Xi_z)}}+\sqrt{\frac{\log(1/\delta)}{n}}+(1+\varepsilon_{1})^{2}\frac{n}{R_{\|\cdot\|_{2}}(\Xi_z)}\right. \\
     &+\frac{\|\Sigma_{u}^{+}\corrcoef\|_{2}}{\tilde{\sigma}}\frac{\mathbb{E}\|\Xi_z^{1/2}H\|_{2}}{\sqrt{n}}\left(1+\sqrt{\frac{\log(1/\delta)}{r_{\|\cdot\|_{2}}(\Xi_z)}}\right) \\
     & \left.+(1+\varepsilon_{3})^{2}\frac{n}{(\mathbb{E}\|\Xi_{z}^{1/2}H\|_{*})^{2}}(\Ep\|\Sigma_{u}^{1/2}P_{z}v^{*}\|_{2})^{2}+\varepsilon_{2}\right).
\end{align*}
If $n$ and the effective ranks are large enough that $\varepsilon\leq1$, with probability at least $1-\delta$, it holds that
\begin{equation*}
    \|\hat{\theta}\|\leq\|\theta_{0}\|+\|\Sigma_{u}^{+}\corrcoef\|+(1+\varepsilon)(2\eta_{1}+\tilde{\sigma}+\eta_{2})\sqrt{\frac{n}{(\mathbb{E}\|\Xi_z^{1/2}H\|_{*})^{2}}}.
\end{equation*}
\end{theorem}

\begin{proof}[Proof of Theorem \ref{Non_ortho_Variant_Theorem4}]
For any $t>0$, it holds from Lemmas \ref{Variant_Lemma6} and \ref{Variant_Lemma7} that
\begin{equation*}
\Pr(\|\hat{\theta}\|>t)\leq\Pr(\Phi>t-\|\theta_{0}\|)\leq2\Pr(\phi\geq t-\|\theta_{0}\|).
\end{equation*}
Lemma \ref{Non_ortho_Variant_Lemma8} implies that the above term is upper bounded by $\delta$ if we choose $t-\|\theta_{0}\|$ using the result (\ref{Non_ortho_Variant_67}) with $\delta$ replaced by $\delta/2$. We obtain the stated result by moving $\|\theta_{0}\|$ to the other side.
\end{proof}

\noindent
\textbf{Theorem \ref{Non_ortho_Variant_Theorem2}} (Euclidean norm bound; special case of Theorem \ref{Non_ortho_Variant_Theorem4}) \textit{
Fix any $\delta\leq1/4$.  
Under the model assumptions with covariance $\Sigma_x=\Xi_z+\Sigma_{u}$, there exists some $\varepsilon\lesssim \sqrt{\rank(\Sigma_{u})/n}+\sqrt{\log(1/\delta)/n}+(1+(\|\Sigma_{u}^{+}\corrcoef\|_{2}/\tilde{\sigma})\sqrt{\trace(\Xi_{z})/n}) \sqrt{\log(1/\delta)/r(\Xi_z)}+(n\log(1/\delta))/(R(\Xi_z))(1+\trace(\Sigma_{u}\Xi_{z})/\trace(\Xi_{z}^{2}))+(\|\Sigma_{u}^{+}\corrcoef\|_{2}/\tilde{\sigma})\sqrt{\trace(\Xi_{z})/n}$ such that the following is true. If $n$ and the effective ranks are such that $\varepsilon\leq1$ and $R(\Xi_z)\gtrsim \log(1/\delta)^{2}$, then with probability at least $1-\delta$, it holds that
\begin{equation*}
\|\hat{\theta}\|_{2}\leq\|\theta_{0}\|_{2}+\|\Sigma_{u}^{+}\corrcoef\|_{2}+(1+\varepsilon)^{1/2}(2\eta_{1}+\tilde{\sigma}+\eta_{2})\sqrt{\frac{n}{\trace({\Xi_{z}})}}.
\end{equation*}
}

\begin{proof}[Proof of Theorem \ref{Non_ortho_Variant_Theorem2}]
Throughout this proof, we simplify the upper bound, especially $
    n/R_{\|\cdot\|_{2}}(\Xi_z)$, $1/r_{\|\cdot\|_{2}}(\Xi_{z})$ and $(n\Ep\|\Sigma_{u}^{1/2}P_{z}v^{*}\|_{2}^{2})/(\mathbb{E}\|\Xi_{z}^{1/2}H\|_{*})^{2}$, in Theorem \ref{Non_ortho_Variant_Theorem4}. By the definition of the dual norm and $\partial\|\Xi_{x}^{1/2}H\|_{*}$ with Euclidean norm, $v^{*}$ is equal to $\Xi_z^{1/2}H/\|\Xi_z^{1/2}H\|_{2}$. Hence, $\|v^{*}\|_{\Xi_z}$ is $\|\Xi_zH\|_{2}/\|\Xi_z^{1/2}H\|_{2}$. From the result of (\ref{Variant_78}), for some constant $c>0$, we can choose $\varepsilon_{1}$ such that
\begin{equation*}
    (1+\varepsilon_{1})E\|v^{*}\|_{\Xi_z}=c\sqrt{\log(64/\delta)\frac{\trace(\Xi_z^{2})}{\trace(\Xi_z)}}.
\end{equation*}
If we assume effective rank is sufficiently large, (\ref{Variant_72}) provides that $\left(E\|\Xi_z^{1/2}H\|_{2}\right)^{2}\gtrsim \trace(\Xi_z)$. Therefore, we have
\begin{equation*}
    (1+\varepsilon_{1})^{2}\frac{n}{R_{\|\cdot\|_{2}}(\Xi_z)}=n\frac{(1+\varepsilon_{1})^{2}(E\|v^{*}\|_{\Xi_z})^{2}}{\left(E\|\Xi_z^{1/2}H\|_{2}\right)^{2}}\lesssim n\log(16/\delta)\frac{\trace(\Xi_z^{2})}{\trace(\Xi_z)^{2}}=\frac{n\log(16/\delta)}{R(\Xi_z)}.
\end{equation*}
Moreover, it holds from (\ref{Variant_78_Orthogonal}) that for some constant $c_{2}>0$, there exists $\varepsilon_{3}$ such that 
\begin{equation*}
    (1+\varepsilon_{3})\Ep\|\Sigma_{u}^{1/2}Pv^{*}\|_{2}=c_{2}\sqrt{\log(64/\delta)\frac{\trace(\Sigma_{u}\Xi_{z})}{\trace(\Xi_{z})}}.
\end{equation*}
By (\ref{Variant_72}), for sufficiently large effective rank, it holds that $(E\|\Xi_z^{1/2}H\|_{2})^{2}\gtrsim \trace(\Xi_z)$. Therefore, we have
\begin{align*}
    (1+\varepsilon_{3})^{2}\frac{n}{(\mathbb{E}\|\Xi_{z}^{1/2}H\|_{2})^{2}}(\Ep\|\Sigma_{u}^{1/2}P_{z}v^{*}\|_{2})^{2}&\lesssim  n\log(64/\delta)\frac{\trace(\Sigma_u\Xi_{z})}{\trace(\Xi_z)^{2}}\\
    &=\frac{n\log(64/\delta)}{R(\Xi_z)}\frac{\trace(\Sigma_u\Xi_z)}{\trace(\Xi_z^{2})}.
\end{align*}
By trivial calculation, for any covariance matrix $\Sigma$, we have
\begin{align*}
    \trace(\Sigma)=\Ep\|\Sigma^{1/2}H\|^{2}_{2}&=(\Ep\|\Sigma^{1/2}H\|_{2})^{2}+\mathrm{Var}\|\Sigma^{1/2}H\|_{2} \\
    &\geq (\Ep\|\Sigma^{1/2}H\|_{2})^{2}.
\end{align*}
Therefore, we have
\begin{align*}
\frac{\|\Sigma_{u}^{+}\corrcoef\|_{2}}{\tilde{\sigma}}\frac{\mathbb{E}\|\Xi_z^{1/2}H\|_{2}}{\sqrt{n}}\left(1+\sqrt{\frac{2\log(8/\delta)}{r_{\|\cdot\|_{2}}(\Xi_z)}}\right)&\leq \frac{\|\Sigma_{u}^{+}\corrcoef\|_{2}}{\tilde{\sigma}}\sqrt{\frac{\trace(\Xi_{z})}{n}}\left(1+\sqrt{\frac{2\log(8/\delta)}{r_{\|\cdot\|_{2}}(\Xi_z)}}\right).
\end{align*}

Finally, we obtain the upper bound of $\varepsilon$. As $P_{z}$ is an $l_{2}$ projection matrix, let $\varepsilon_{2}$ be zero. Then, it holds from (\ref{Variant_74}) of Lemma \ref{Variant_Lemma9}  that 
\begin{align*}
\varepsilon\lesssim  &\sqrt{\frac{\rank(\Sigma_{u})}{n}}+\sqrt{\frac{\log(1/\delta)}{n}}+\left(1+\frac{\|\Sigma_{u}^{+}\corrcoef\|_{2}}{\tilde{\sigma}}\sqrt{\frac{\trace(\Xi_{z})}{n}}\right) \sqrt{\frac{\log(1/\delta)}{r(\Xi_z)}} \\
&+\frac{n\log(1/\delta)}{R(\Xi_z)}\left(1+\frac{\trace(\Sigma_{u}\Xi_{z})}{\trace(\Xi_{z}^{2})}\right)+\frac{\|\Sigma_{u}^{+}\corrcoef\|_{2}}{\tilde{\sigma}}\sqrt{\frac{\trace(\Xi_{z})}{n}}. 
\end{align*}
By using the inequality $(1-x)^{-1}\leq1+2x$ for $x\in[0,1/2]$ and (\ref{Variant_72}) of Lemma \ref{Variant_Lemma9}, we finally obtain
\begin{align*}
    (1+\varepsilon)\frac{\sqrt{n}}{E\|\Xi_z^{1/2}H\|_{2}}&\leq (1+\varepsilon)\left(1-\frac{1}{r(\Xi_z)}\right)^{-1/2}\sqrt{\frac{n}{\trace(\Xi_z)}}  \\
    &\leq (1+\varepsilon)\left(1+\frac{2}{r(\Xi_z)}\right)^{1/2}\sqrt{\frac{n}{\trace(\Xi_z)}}.
\end{align*}
As $\varepsilon\leq1$, we have
\begin{align*}
    (1+\varepsilon)\left(1+\frac{2}{r(\Xi_z)}\right)^{1/2}&=\left((1+2\varepsilon+\varepsilon^{2})\left(1+\frac{2}{r(\Xi_z)}\right)\right)^{1/2} \\
    &\leq \left((1+3\varepsilon)\left(1+\frac{2}{r(\Xi_z)}\right)\right)^{1/2} \\
    &\leq \left(1+3\varepsilon+\frac{8}{r(\Xi_z)}\right)^{1/2}.
\end{align*}
Therefore, it holds that
\begin{equation*}
    (1+\varepsilon)\frac{\sqrt{n}}{E\|\Xi_z^{1/2}H\|_{2}}\leq \left(1+3\varepsilon+\frac{8}{r(\Xi_z)}\right)^{1/2}\sqrt{\frac{n}{\trace(\Xi_z)}},
\end{equation*}
and we can replace $\varepsilon$ with
\begin{align*}
    \varepsilon'=&3\varepsilon+\frac{8}{r(\Xi_z)} \\
    &\lesssim \sqrt{\frac{\rank(\Sigma_{u})}{n}}+\sqrt{\frac{\log(1/\delta)}{n}}+\left(1+\frac{\|\Sigma_{u}^{+}\corrcoef\|_{2}}{\tilde{\sigma}}\sqrt{\frac{\trace(\Xi_{z})}{n}}\right) \sqrt{\frac{\log(1/\delta)}{r(\Xi_z)}} \\
+&\frac{n\log(1/\delta)}{R(\Xi_z)}\left(1+\frac{\trace(\Sigma_{u}\Xi_{z})}{\trace(\Xi_{z}^{2})}\right)+\frac{\|\Sigma_{u}^{+}\corrcoef\|_{2}}{\tilde{\sigma}}\sqrt{\frac{\trace(\Xi_{z})}{n}}. 
\end{align*}
\end{proof}

\begin{theorem}[Benign Overfitting (Non-orthogonal)]\label{Non_ortho_Variant_Theorem3}
Fix any $\delta\leq1/2$. Under the model assumptions with $\Sigma_x=\Xi_z+\Sigma_{u}$, let $\gamma$ and $\varepsilon$ be as defined in Corollary \ref{Variant_Corollary2} and Theorem \ref{Non_ortho_Variant_Theorem2}. Suppose that $n$ and the effective ranks are such that $R(\Xi_z)\gtrsim \log(1/\delta)^{2}$ and $\gamma,\varepsilon\leq1$. Then, with probability at least $1-\delta$,
\begin{align*}
    \|\hat{\theta}-\theta_{0}\|_{\Xi_z}^{2}&\leq(1+\gamma)(1+\varepsilon)\left((\|\theta_{0}\|_{2}+\|\Sigma_{u}^{+}\corrcoef\|_{2})\sqrt{\frac{\trace(\Xi_z)}{n}}+(2\eta_{1}+\tilde{\sigma}+\eta_{2})\right)^{2}-\tilde{\sigma}^{2}.
\end{align*}
\end{theorem}

\begin{proof}[Proof of Theorem \ref{Non_ortho_Variant_Theorem3}]
From the result of Theorem \ref{Non_ortho_Variant_Theorem2}, if we adopt
\begin{equation*}
    B=\|\theta_{0}\|_{2}+\|\Sigma_{u}^{+}\corrcoef\|_{2}+(1+\varepsilon)^{1/2}(2\eta_{1}+\tilde{\sigma}+\eta_{2})\sqrt{\frac{n}{\trace(\Xi_{z})}},
\end{equation*}
then $\{\theta:\|\theta\|_{2}\leq B\}\cap\{\theta: \mathbf{X}\theta=\mathbf{Y}\}$ is not empty with high probability. This intersection necessarily contains the ridgeless estimator $\hat{\theta}$. Clearly, $B>\|\theta_{0}\|_{2}$ holds. Therefore, it holds from Corollary \ref{Variant_Corollary2} that
\begin{align*}
     &\|\hat{\theta}-\theta_{0}\|_{\Xi_z}^{2}\\ 
     &\leq  \max_{\|\theta\|_{2}\leq B, \mathbf{Y}=\mathbf{X} \theta}\|\theta-\theta_{0}\|_{\Xi_z}^{2} \\
     &\leq(1+\gamma)\frac{B^{2}\trace(\Xi_z)}{n}-\tilde{\sigma}^{2} \\
     &\leq(1+\gamma)(1+\varepsilon)\left((\|\theta_{0}\|_{2}+\|\Sigma_{u}^{+}\corrcoef\|_{2})\sqrt{\frac{\trace(\Xi_z)}{n}}+(2\eta_{1}+\tilde{\sigma}+\eta_{2})\right)^{2}-\tilde{\sigma}^{2}.
\end{align*}

\end{proof}

\begin{theorem}\label{Non_ortho_Pre}
Under Assumption \ref{asmp:Gaussianity}, let $\hat{\theta}$ be the ridgeless estimator. Suppose that as $n$ goes to $\infty$, the covariance splitting $\Sigma_x=\Xi_z + \Sigma_{u}$ satisfies  the following conditions:
\begin{enumerate}
    \item[(i)] (Small large-variance dimension.)
    \begin{equation*}
        \lim_{n\rightarrow\infty}\frac{\rank(\Sigma_{u})}{n}=0.
    \end{equation*}
    \item [(ii)] (Large effective dimension.)
    \begin{equation*}
        \lim_{n\rightarrow\infty}\frac{n}{R(\Xi_z)}=0.
    \end{equation*}
    \item [(iii)] (No aliasing condition.)
    \begin{equation*}
        \lim_{n\rightarrow\infty}\|\theta_{0}\|_{2}\sqrt{\frac{\trace(\Xi_z)}{n}}=0.
    \end{equation*}
    \item [(iv)] (Condition for the minimal interpolation of instrumental variable in the non-orthogonal case)
    \begin{equation*}
        \lim_{n\rightarrow\infty}\frac{\|\Sigma_{u}^{+}\corrcoef\|_{2}}{\tilde{\sigma}}\sqrt{\frac{\trace(\Xi_{z})}{n}}=0.
    \end{equation*}
     \item [(v)] (Non-orthogonality)
       \begin{enumerate}
        \item \begin{equation*}
           \lim_{n\rightarrow\infty}\eta_{1}=0.
        \end{equation*}
        \item \begin{equation*}
           \lim_{n\rightarrow\infty}\eta_{2}=0.
        \end{equation*}
        \item
        \begin{equation*}
            \lim_{n\rightarrow\infty}\frac{n}{R(\Xi_z)}\frac{\trace(\Sigma_{u}\Xi_{z})}{\trace(\Xi_{z}^{2})}=0.
        \end{equation*}
    \end{enumerate}
\end{enumerate}
Then, $ \|\hat{\theta}-\theta_{0}\|_{\Xi_z}^{2}$ converges to $0$ in probability.
\end{theorem}

\begin{proof}[Proof of Theorem \ref{Non_ortho_Pre}]
Fix any $\kappa>0$ and $\delta>0$. From Lemma 5 of \citet{bartlett2020benign}, it holds that $R(\Xi_z)\leq r(\Xi_z)^{2}$. If $R(\Xi_z)=\orderomega(n)$ holds as the second condition in Theorem \ref{Non_ortho_Pre}, we have $r(\Xi_z)=\orderomega(\sqrt{n})=\orderomega(1)$, which implies the convergence of  $1/r(\Xi_z)$ to zero.  Hence, conditions (i) and (ii) in Theorem \ref{Non_ortho_Pre} make $\gamma$ sufficiently small for large enough $n$. Clearly, $(\|\theta_{0}\|_{2}+\|\Sigma_{u}^{+}\corrcoef\|_{2})\sqrt{\trace(\Xi_z)/n}$ goes to zero from conditions (iii) and (iv). By the definition of $\varepsilon$,  conditions (i), (ii), (iv), and (v)(c) in Theorem \ref{Non_ortho_Pre} imply that $\varepsilon$ can be arbitrarily small. Combined with conditions (v)(a) and (v)(b), for sufficiently large $n$, we obtain
\begin{align}
   &(1+\gamma)(1+\varepsilon)\left((\|\theta_{0}\|_{2}+\|\Sigma_{u}^{+}\corrcoef\|_{2})\sqrt{\frac{\trace(\Xi_z)}{n}}+(2\eta_{1}+\tilde{\sigma}+2\eta_{2})\right)^{2}-\tilde{\sigma}^{2} \leq\kappa. \label{Non_Ortho_Variant_87}
\end{align}
We have shown that $\gamma$, $\varepsilon$, $(\|\theta_{0}\|_{2}+\|\Sigma_{u}^{+}\corrcoef\|_{2})\sqrt{\trace(\Xi_z)/n}$ , $\eta_{1}$, and $\eta_{2}$ are so small that  equation (\ref{Non_Ortho_Variant_87}) holds for sufficiently large $n$. Therefore, we obtain
\begin{equation*}
   \Pr(\|\hat{\theta}-\theta_{0}\|_{\Xi_z}^{2}>\kappa)\leq\delta,
\end{equation*}
for any fixed $\kappa$. As $\kappa$ and $\delta$ are arbitrary, we have for any $\kappa$, 
\begin{equation*}
    \lim_{n\rightarrow\infty}\Pr(\|\hat{\theta}-\theta_{0}\|_{\Xi_z}^{2}>\kappa)=0.
\end{equation*}
\end{proof}

\begin{lemma}\label{sufficient_3}
Suppose $ \lim_{n\rightarrow\infty}\corrcoef^\top\Sigma_{u}^{+}\Xi_{z}\Sigma_{u}^{+}\corrcoef=0$ holds. Suppose the second condition of Theorem \ref{Non_ortho_Pre} holds. Then, with probability at least $1-\delta$,  $\lim_{n\rightarrow\infty}\eta_{1}=0$. 
\end{lemma}

\begin{proof}[Proof of Lemma \ref{sufficient_3}]
By the definition of $\eta_{1}$, we have
\begin{equation}
    \eta_{1}=\sqrt{(1+\varepsilon_{1})^{2}\frac{n}{R_{\|\cdot\|}(\Xi_{z})}}\|\Xi_{z}^{1/2}\Sigma_{u}^{+}\corrcoef\|_{2}.
\end{equation}
By (\ref{Variant_72}), for sufficiently large effective rank, it holds that $\left(E\|\Xi_z^{1/2}H\|_{2}\right)^{2}\gtrsim \trace(\Xi_z)$ and so
\begin{equation*}
    (1+\varepsilon_{1})^{2}\frac{n}{R_{\|\cdot\|_{2}}(\Xi_z)}=n\frac{(1+\varepsilon_{1})^{2}(E\|v^{*}\|_{\Xi_z})^{2}}{\left(E\|\Xi_z^{1/2}H\|_{2}\right)^{2}}\lesssim n\log(4/\delta)\frac{\trace(\Xi_z^{2})}{\trace(\Xi_z)^{2}}=\frac{n\log(4/\delta)}{R(\Xi_z)}
\end{equation*}
As $n/R(\Xi_z)$ and $\corrcoef^\top\Sigma_{u}^{+}\Xi_{z}\Sigma_{u}^{+}$ converge to zero, we have $\lim_{n\rightarrow\infty}\eta_{1}=0$.
\end{proof}

\begin{lemma}\label{sufficient_4}
Suppose $ \lim_{n\rightarrow\infty}\corrcoef^\top\Sigma_{u}^{+}\Xi_{z}\Sigma_{u}^{+}\corrcoef=0$ holds. The fourth condition of Theorem \ref{Non_ortho_Pre} implies $\lim_{n\rightarrow\infty}\eta_{2}=0$. 
\end{lemma}

\begin{proof}[Proof of Lemma \ref{sufficient_4}]
By the definition of $\eta_{2}$, we have
\begin{align*}
    \eta_{2}:=\sqrt{\frac{(\mathbb{E}\|\Xi_z^{1/2}H\|_{2})^{2}}{n}\left(1+\sqrt{\frac{2\log(8/\delta)}{r_{\|\cdot\|_{2}}(\Xi_z)}}\right)\|\Sigma_{u}^{+}\corrcoef\|^{2}_{2}+\|\Xi_{z}^{1/2} \Sigma_{u}^{+}\corrcoef\|^{2}_{2}}.
\end{align*}
By trivial calculation, for any covariance matrix $\Sigma$, it holds that 
\begin{align*}
    \trace(\Sigma)=\Ep\|\Sigma^{1/2}H\|^{2}_{2}&=(\Ep\|\Sigma^{1/2}H\|_{2})^{2}+\mathrm{Var}(\|\Sigma^{1/2}H\|_{2}) \geq (\Ep\|\Sigma^{1/2}H\|_{2})^{2}.
\end{align*}
From the result of (\ref{Variant_74}), we have
\begin{equation*}
    \eta_{2}\lesssim \sqrt{\frac{\trace(\Xi_{z})}{n}\left(1+\sqrt{\frac{\log(1/\delta)}{r(\Xi_z)}}\right)\|\Sigma_{u}^{+}\corrcoef\|^{2}_{2}+\|\Xi_{z}^{1/2} \Sigma_{u}^{+}\corrcoef\|^{2}_{2}}.
\end{equation*}
As $R(\Xi_z)\rightarrow\infty$ implies $r(\Xi_z)\rightarrow\infty$, we have the following result:
\begin{equation*}
    \left(1+\sqrt{\frac{\log(1/\delta)}{r(\Xi_z)}}\right)\rightarrow1.
\end{equation*}
Moreover, the conditions
\begin{equation*}
   \lim_{n\rightarrow\infty}\frac{\|\Sigma_{u}^{+}\corrcoef\|_{2}}{\tilde{\sigma}}\sqrt{\frac{\trace(\Xi_{z})}{n}}=0 \quad \text{and} \quad \lim_{n\rightarrow\infty}\corrcoef^\top\Sigma_{u}^{+}\Xi_{z}\Sigma_{u}^{+}\corrcoef=0
\end{equation*}
lead to the conclusion $\lim_{n\rightarrow\infty}\eta_{2}=0$.
\end{proof}

\textbf{Theorem \ref{Non_ortho_Variant_Theorem12}} (Sufficient conditions: Non-Orthogonal Case) \textit{Under Assumption \ref{asmp:Gaussianity}, let $\hat{\theta}$ be the ridgeless estimator. Suppose that as $n$ goes to $\infty$, the covariance splitting $\Sigma_x=\Xi_z + \Sigma_{u}$ satisfies  the following conditions:
\begin{enumerate}
    \item[(i)] (Small large-variance dimension.)
    \begin{equation*}
        \lim_{n\rightarrow\infty}\frac{\rank(\Sigma_{u})}{n}=0.
    \end{equation*}
    \item[(ii)] (Large effective dimension.)
    \begin{equation*}
        \lim_{n\rightarrow\infty}\frac{n}{R(\Xi_z)}=0.
    \end{equation*}
    \item[(iii)] (No aliasing condition.)
    \begin{equation*}
        \lim_{n\rightarrow\infty}\|\theta_{0}\|_{2}\sqrt{\frac{\trace(\Xi_z)}{n}}=0.
    \end{equation*}
    \item[(iv)] (Condition for the minimal interpolation of instrumental variable in the non-orthogonal case)
    \begin{equation*}
        \lim_{n\rightarrow\infty}\frac{\|\Sigma_{u}^{+}\corrcoef\|_{2}}{\tilde{\sigma}}\sqrt{\frac{\trace(\Xi_{z})}{n}}=0.
    \end{equation*}
    \item[(v)] (Non-orthogonality)
    \begin{enumerate}
        \item \begin{equation*}
            \lim_{n\rightarrow\infty}\frac{n}{R(\Xi_z)}\frac{\trace(\Sigma_{u}\Xi_{z})}{\trace(\Xi_{z}^{2})}=0.
        \end{equation*}
        \item \begin{equation*}
             \lim_{n\rightarrow\infty}\corrcoef^\top\Sigma_{u}^{+}\Xi_{z}\Sigma_{u}^{+}\corrcoef=0.
        \end{equation*}
    \end{enumerate}
\end{enumerate}
Then, $ \|\hat{\theta}-\theta_{0}\|_{\Xi_z}^{2}$ converges to $0$ in probability.
}

\begin{proof}[Proof of Theorem \ref{Non_ortho_Variant_Theorem12}]
By Lemma \ref{sufficient_3}, it holds that from conditions (ii) and (v)(b) that $\lim_{n\rightarrow\infty}\eta_{1}=0$. We also have $\lim_{n\rightarrow\infty}\eta_{2}=0$ by Lemma \ref{sufficient_4}. Therefore, from Theorem \ref{Non_ortho_Pre}, $ \|\hat{\theta}-\theta_{0}\|_{\Xi_z}^{2}$ converges to $0$ in probability.
\end{proof}

\section{Supportive Result} \label{app_sec:support}

\begin{proposition} \label{prop:pi_0}
    The necessary condition for $\lim_{n\rightarrow\infty}\rank(\Sigma_{u})/n=0$ in Theorem \ref{Variant_Theorem12} is 
\begin{align*}
    \rank(\Sigma_u) \geq p - \min\{\mathrm{rank}(\Sigma_{z}),\mathrm{rank}(\Pi_{0})\}.
\end{align*}
\end{proposition}
\begin{proof}[Proof of Proposition \ref{prop:pi_0}]
As we have $\rank(A)=\rank(A^\top)=\rank(AA^\top)=\rank(A^\top A)$ for any matrix $A$, we have $\rank(\Pi_{0}\Sigma_{z}^{1/2})=\rank(\Xi_{z})$. From the property of the matrix, we have
\begin{equation*}
    \rank(\Pi_{0}\Sigma_{z}^{1/2})\leq \min\{\rank(\Sigma_{z}^{1/2}),rank(\Pi_{0})\}.
\end{equation*}
Under Assumption \ref{asmp:orthogonal}, we have
\begin{equation*}
    \rank(\Xi_{z})+\rank(\Sigma_{u})=p.
\end{equation*}
Therefore, we obtain the statement.
\end{proof}

\begin{lemma}\label{lem:Covariance_IV}
Assume Assumption \ref{asmp:orthogonal} holds. 
Then, we obtain the following covariance splitting:
\begin{equation*}
    \Sigma_x=\Xi_z + \Sigma_{u},
\end{equation*}
where $\Xi_z$ and $\Sigma_{u}$ are positive semidefinite matrices and subspaces generated from $\Xi_z and \Sigma_{u}$, which are orthogonal. 
\end{lemma}

\begin{proof}[Proof of Lemma \ref{lem:Covariance_IV}]
By construction, we have
\begin{equation*}
    X_{i}=\Pi_{0}Z_{i}+u_{i},
\end{equation*}
where $E[u_{i}|Z_{i}]=0$. Therefore, we have
\begin{align*}
    \Sigma_x=E[X_{i}X_{i}^\top ]&=\Pi_{0}E[Z_{i}Z_{i}^\top ]\Pi_{0}^\top +E[u_{i}u_{i}^\top ] =\Xi_z+\Sigma_{u}.
\end{align*}
By construction, clearly $\Xi_z,\Sigma_{u}$ are positive semidefinite. By Assumption \ref{asmp:orthogonal}, subspaces generated by $\Xi_z,\Sigma_{u}$ are orthogonal. Therefore, we have $\Sigma_x=\Xi_z + \Sigma_{u}$.
\end{proof}

\begin{lemma}\label{rho-sigma_inq}
Under Assumptions \ref{asmp:Gaussianity} and \ref{asmp:orthogonal}, we have the following inequality:
\begin{equation*}
     \|\rho\|_{2}^{2} = \|\corrcoef\|_{\Sigma_u^{+}}^{2}\leq\sigma^{2}.
\end{equation*}
Furthermore, the covariance matrix of $(W_{1,i}, W_{2,i}, \xi_{i})^{\top}$ in \eqref{def:cov_wwxi} is positive semi-definite.
\end{lemma}

\begin{proof}[Proof of Lemma \ref{rho-sigma_inq}]
First, we clarify the necessary and sufficient condition of positive semi-definiteness of the covariance matrix of $(X_{i},\xi_{i})^\top$. The definition of positive semi-definiteness is 
\begin{equation*}
(a,b)
\begin{pmatrix}
\Sigma_{x} & \Sigma_{u}^{1/2}\mathbf{\rho} \\
\mathbf{\rho}^\top\Sigma_{u}^{1/2}  & \sigma^{2}
\end{pmatrix}
(a,b)^\top\geq0,
\end{equation*}
for any $a\in\mathbb{R}^{p}$ and $b\in\mathbb{R}$. By trivial calculation, we have
\begin{equation}\label{co-variancematrix_X_xi}
(a,b)
\begin{pmatrix}
\Sigma_{x} & \Sigma_{u}^{1/2}\mathbf{\rho} \\
\mathbf{\rho}^\top\Sigma_{u}^{1/2}  & \sigma^{2}
\end{pmatrix}
(a,b)^\top=a^\top \Sigma_{x} a+2ba^\top \Sigma_{u}^{1/2}\rho+b^{2}\sigma^{2} .
\end{equation}
By solving the first order condition of (\ref{co-variancematrix_X_xi}) with respect to $a\in\mathbb{R}^{p}$, we have the solution $a^{*}=-b\Sigma^{-1}_{x}\Sigma_{u}^{1/2}\rho$. By substituting $a^{*}$ into (\ref{co-variancematrix_X_xi}), it holds that
\begin{equation*}
   b^{2}\rho^\top \Sigma_{u}^{1/2}\Sigma^{-1}_{x}\Sigma_{u}^{1/2}\rho-2b^{2}\rho^\top \Sigma_{u}^{1/2}\Sigma^{-1}_{x}\Sigma_{u}^{1/2}\rho+b^{2}\sigma^{2}=b^{2}(\sigma^{2}-\rho^\top \Sigma_{u}^{1/2}\Sigma^{-1}_{x}\Sigma_{u}^{1/2}\rho).
\end{equation*}
Therefore, $(\sigma^{2}-\rho^\top \Sigma_{u}^{1/2}\Sigma^{-1}_{x}\Sigma_{u}^{1/2}\rho)\geq 0$ is the sufficient and necessary condition for positive semi-definiteness. Likewise, we obtain  $(\sigma^{2}-\rho^\top \rho)\geq 0$ as the sufficient and necessary condition for positive semi-definiteness for the covariance matrix of $(W_{1,i}, W_{2,i}, \xi_{i})^{\top}$.

Finally, under Assumptions \ref{asmp:Gaussianity} and \ref{asmp:orthogonal}, we show that positive semi-definiteness of the covariance matrix of $(X_{i},\xi_{i})^\top$ implies that the covariance matrix of $(W_{1,i}, W_{2,i}, \xi_{i})^\top$ is positive semi-definite. 
As $\rho$ is defined as the one that has the minimum norm subject to $\corrcoef = \Sigma_u^{1/2} \rho$,
$    \rho := \argmin\{\|b\|: \corrcoef = \Sigma_u^{1/2} b\}$.
By the property of the generalized inverse matrix, we have
\begin{equation*}
    \rho=(\Sigma_u^{1/2})^{+}\corrcoef.
\end{equation*}
Under Assumption \ref{asmp:orthogonal}, $\Sigma_{u}\Xi_{z}=0$. As $\Sigma_{u}$ and $\Xi_{z}$ are symmetric, we have 
\begin{align*}
\Sigma_{x}(\Sigma_{u}^{+}+\Xi_{z}^{+})&=(\Sigma_{u}+\Xi_{z})(\Sigma_{u}^{+}+\Xi_{z}^{+}) \\
    &=\Sigma_{u}\Sigma_{u}^{+}+\Sigma_{u}\Xi_{z}^{+}+\Xi_{z}\Sigma_{u}^{+}+\Xi_{z}\Xi_{z}^{+} \\
    &=I_{u}+\Sigma_{u}\Xi_{z}^{\top}(\Xi_{z}\Xi_{z}^\top )^{+}+\Xi_{z}\Sigma_{u}^{\top}(\Sigma_{u}\Sigma_{u}^\top )^{+}+(I_{p}-I_{u}) \\
    &=I_{p}.
\end{align*}
Hence, $\Sigma_{x}^{-1}=\Sigma_{u}^{+}+\Xi_{z}^{+}$. Then, we have
\begin{align*}
\rho^\top \Sigma_{u}^{1/2}\Sigma_{x}^{-1}\Sigma_{u}^{1/2}\rho&=\corrcoef^\top \Sigma_{x}^{-1}\corrcoef =\corrcoef^\top \Sigma_{u}^{+}\corrcoef+\corrcoef^\top\Xi_{z}^{+}\corrcoef = \corrcoef^\top \Sigma_{u}^{+}\corrcoef =\rho^\top\rho.
\end{align*}
The third equality holds because $\corrcoef=\Sigma_{u}^{1/2}\rho$ and $\Sigma_{u}$ is orthogonal to $\Xi_{z}$. The above discussion suggests
\begin{equation*}
    \sigma^{2}- \rho^\top \Sigma_{u}^{1/2}\Sigma^{-1}_{x}\Sigma_{u}^{1/2}\rho= \sigma^{2}-\rho^\top \rho.
\end{equation*}
Therefore, by the positive semi-definiteness of the covariance matrix of $(X_{i},\xi_{i})^\top$, we have
\begin{equation*}
    \sigma^{2}-\rho^\top \rho=\sigma^{2}- \corrcoef^\top \Sigma^{-1}_{x}\corrcoef\geq 0.
\end{equation*}
\end{proof}

\begin{proof}[Proof of Proposition \ref{example_1}]
 By Theorem 2 (1) in \citet{bartlett2020benign}, the first and second conditions of Definition \ref{cond:basic} are satisfied.  As $\trace(\Xi_{z})$ converges to a finite value, the third condition also holds under the assumption $\|\theta_{0}\|_{2}=o(\sqrt{n})$. 

We show the condition in Theorem \ref{Variant_Theorem12} is satisfied under the setting of Proposition \ref{example_1}. By the setting, we have
 \begin{align*}
    \frac{1}{n}\|\Sigma_{u}^{+}\corrcoef\|^{2}_{2}&=\frac{1}{n}\sum_{i=1}^{k^{*}_{n}}i^{2}\log^{2\beta}(i+1)\cdot(U\corrcoef)_{i}^{2} \lesssim \frac{1}{n}\sum_{i=1}^{k^{*}_{n}}\frac{i^{2}\log^{2\beta}(i+1)}{i^{2}\log^{2\beta}(i+1)} =  \frac{k^{*}_{n}}{n}.
\end{align*}
As $\lim_{n\rightarrow\infty}k^{*}_{n}/n=0$ holds, we have $\lim_{n\rightarrow\infty}\|\Sigma_{u}^{+}\corrcoef\|^{2}_{2}/n=0$.
\end{proof}

\begin{proof}[Proof of Proposition \ref{example_2}]
By Theorem 2 (2) in \citet{bartlett2020benign}, the first and second conditions of Definition \ref{cond:basic} are satisfied.

We prove $\theta_{0}$ and $\trace(\Xi_{z})$ satisfy the third condition stated in Definition \ref{cond:basic}. By the definition of the matrices, we have
\begin{align}\label{proof_6-1}
    \trace(\Xi_{z})&\leq \trace (\Sigma_{x})  =\sum_{i=1}^{p}(\gamma_{i}+\varepsilon_{n})  =\sum_{i=1}^{p}\gamma_{i}+p\varepsilon_{n}.
\end{align}
As $p\varepsilon_{n}$ is equal to $ne^{-o(n)}$, the second term of (\ref{proof_6-1}) converges to zero. It also holds from the definition that
\begin{align*}
    \sum_{i=1}^{p}\gamma_{i}\lesssim \sum_{i=1}^{p}\exp(-i/\tau).
\end{align*}
 As $\sum_{i=1}^{\infty}\exp(-i/\tau)$ is finite, $\trace(\Xi_{z})$ is also finite. Therefore, $\lim_{n\rightarrow\infty}\|\theta_{0}\|_{2}\sqrt{\trace(\Xi_{z})/n}=0$
 holds.

Finally, we show the setting in Proposition \ref{example_2} satisfies the condition stated in Theorem \ref{Variant_Theorem12}. By the assumption of Proposition \ref{example_2}, we have
\begin{align*}
    \frac{1}{n}\|\Sigma_{u}^{+}\corrcoef\|^{2}_{2}&=\frac{1}{n}\sum_{i=1}^{k^{*}_{n}}\frac{(U\corrcoef)_{i}^{2}}{(\gamma_{i}+\varepsilon_{n})^{2}} \leq \frac{1}{n}\sum_{i=1}^{k^{*}_{n}}\frac{(U\corrcoef)_{i}^{2}}{\gamma_{i}^{2}} \lesssim \frac{1}{n}\sum_{i=1}^{k^{*}_{n}}\frac{\exp(-2i/\tau)}{\exp(-2i/\tau)} =  \frac{k^{*}_{n}}{n}.
\end{align*}
As $k_{n}^{*}/n$ goes to zero, $\lim_{n\rightarrow\infty}\|\Sigma_{u}^{+}\corrcoef\|^{2}_{2}/n=0$ holds.
\end{proof}

\begin{proof}[Proof of Proposition \ref{non-ortho_ex}]

$\rank(\Sigma_{u})$ in Proposition \ref{non-ortho_ex} is the same as $\rank(\Sigma_{u})$ defined in Proposition \ref{example_1}. Hence, by Theorem 2 (1) in \citet{bartlett2020benign}, the first condition of Definition \ref{cond:basic} is satisfied.

To satisfy the second condition of Definition \ref{cond:basic}, we prove $n/R(\Xi_{z})$ goes to zero as $n$ goes to infinity. By the definition of $R(\Xi_{z})$, we have
\begin{align}\label{9-1}
    \frac{n}{R(\Xi_{z})}&=n\frac{\trace(\Xi_{z}^{2})}{(\trace(\Xi_{z}))^{2}} \notag \\
    &=n\frac{\frac{1}{n^{2\alpha}}\sum_{i=1}^{k_{n}^{*}}\lambda_{i}^{2}+\sum_{i=k_{n}^{*}+1}^{p}\lambda_{i}^{2}}{\left(\frac{1}{n^{\alpha}}\sum_{i=1}^{k_{n}^{*}}\lambda_{i}+\sum_{i=k_{n}^{*}+1}^{p}\lambda_{i}\right)^{2}} \notag \\
    &=\frac{\frac{1}{n^{2\alpha}}\sum_{i=1}^{k_{n}^{*}}\lambda_{i}^{2}+\sum_{i=k_{n}^{*}+1}^{p}\lambda_{i}^{2}}{\sum_{i=k_{n}^{*}+1}^{p}\lambda_{i}^{2}}\cdot n\frac{\sum_{i=k_{n}^{*}+1}^{p}\lambda_{i}^{2}}{\left(\sum_{i=k_{n}^{*}+1}^{p}\lambda_{i}\right)^{2}} \cdot \frac{\left(\sum_{i=k_{n}^{*}+1}^{p}\lambda_{i}\right)^{2}}{\left(\frac{1}{n^{\alpha}}\sum_{i=1}^{k_{n}^{*}}\lambda_{i}+\sum_{i=k_{n}^{*}+1}^{p}\lambda_{i}\right)^{2}} \notag \\
    &=\left(\frac{\sum_{i=1}^{k_{n}^{*}}\lambda_{i}^{2}}{n^{2\alpha}\sum_{i=k_{n}^{*}+1}^{p}\lambda_{i}^{2}}+1\right)
    \cdot  n\frac{\sum_{i=k_{n}^{*}+1}^{p}\lambda_{i}^{2}}{\left(\sum_{i=k_{n}^{*}+1}^{p}\lambda_{i}\right)^{2}}\cdot \left(\frac{\sum_{i=1}^{k_{n}^{*}}\lambda_{i}}{n^{^{\alpha}}\sum_{i=k_{n}^{*}+1}^{p}\lambda_{i}}+1\right)^{-2}.
\end{align}
As $\lambda_{i}=Ci^{-1}\log^{-\beta}(i+1)$ where $\beta>1$, it holds that
\begin{align}\label{9-2}
    \lim_{n\rightarrow\infty}\sum_{i=1}^{k_{n}^{*}}\lambda_{i}^{2}&\leq  \lim_{n\rightarrow\infty}\left(\sum_{i=1}^{k_{n}^{*}}\lambda_{i}\right)^{2}=\left(\lim_{n\rightarrow\infty}\sum_{i=1}^{k_{n}^{*}}\lambda_{i}\right)^{2}<\infty.
\end{align}
 Moreover, we have
\begin{align}\label{9-3}
    n^{2\alpha}\sum_{i=k_{n}^{*}+1}^{p}\lambda_{i}^{2}&\gtrsim n^{2\alpha}(p-k_{n}^{*})\left(\frac{1}{p\log^{\beta}(p+1)}\right)^{2} =\frac{1}{q}\left(1-\frac{1}{q}\frac{k_{n}^{*}}{n}\right)\left(\frac{n^{\frac{2\alpha-1}{2\beta}}}{\log(qn+1)}\right)^{2\beta}. 
\end{align}
As $\left(n^{\frac{2\alpha-1}{\beta}}/\log(qn+1)\right)^{2\beta}$ diverges to infinity and $k_{n}^{*}/n$ converges to zero as $n$ goes to infinity, it holds from \eqref{9-3} that $n^{2\alpha}\sum_{i=k_{n}^{*}+1}^{p}\lambda_{i}^{2}$ diverges to infinity. Therefore, we have
\begin{align}\label{9-4}
    \lim_{n\rightarrow\infty}\left(\frac{\sum_{i=1}^{k_{n}^{*}}\lambda_{i}^{2}}{n^{2\alpha}\sum_{i=k_{n}^{*}+1}^{p}\lambda_{i}^{2}}+1\right)=1.
\end{align}
In a similar way to (\ref{9-2}) and (\ref{9-3}), we obtain
\begin{align*}
    \lim_{n\rightarrow\infty}\sum_{i=1}^{k_{n}^{*}}\lambda_{i}&<\infty, \mbox{~and~}
   n^{{\alpha}}\sum_{i=k_{n}^{*}+1}^{p}\lambda_{i}\gtrsim \left(1-\frac{1}{q}\frac{k_{n}^{*}}{n}\right)\left(\frac{n^{\frac{\alpha}{\beta}}}{\log(qn+1)}\right)^{\beta}.
\end{align*}
Hence, we have
\begin{align}\label{9-5}
    \lim_{n\rightarrow\infty}\left(\frac{\sum_{i=1}^{k_{n}^{*}}\lambda_{i}}{n^{^{\alpha}}\sum_{i=k_{n}^{*}+1}^{p}\lambda_{i}}+1\right)^{-2}=1.
\end{align}
Combining the result in the proof of Proposition \ref{example_1}, $\lim_{n\rightarrow\infty}n\sum_{i=k_{n}^{*}+1}^{p}\lambda_{i}^{2}/\left(\sum_{i=k_{n}^{*}+1}^{p}\lambda_{i}\right)^{2}=0$, with (\ref{9-4}) and (\ref{9-5}), we have
\begin{equation*}
    \lim_{n\rightarrow\infty}\frac{n}{R(\Xi_{z})}=0.
\end{equation*}

We show the third condition in Definition \ref{cond:basic} is satisfied under the setting of Proposition \ref{non-ortho_ex}.
By the definition of $\lambda_{i}$, we have
\begin{align*}
    \trace(\Xi_{z})\leq \trace(\Sigma_{x})<\infty.
\end{align*}
 Hence, the third condition of Definition \ref{cond:basic} clearly holds under the assumption $\|\theta_{0}\|_{2}=o(\sqrt{n})$.

We prove $\Sigma_{u}$ and $\omega$  satisfy the first condition stated in Theorem \ref{Non_ortho_Variant_Theorem12}. As $\lim_{n\rightarrow\infty}  ( \sigma^{2}-\|\corrcoef\|^{2}_{\Sigma_{u}^{+}})>0$ holds by assumption, it is sufficient to show $\lim_{n\rightarrow\infty}\|\Sigma_{u}^{+}\corrcoef\|_{2}\sqrt{\trace(\Xi_z)/n}=0$. By the assumption of Proposition \ref{non-ortho_ex}, we have
\begin{align*}
    \frac{1}{n}\|\Sigma_{u}^{+}\corrcoef\|^{2}_{2}&= \frac{1}{n}\sum_{i=1}^{k^{*}_{n}}\frac{n^{2\alpha}}{(n^{\alpha}-1)^{2}}i^{2}\log(i+1))^{2\beta}(U\corrcoef)_{i}^{2}\\
    &\lesssim  \frac{1}{n}\frac{1}{1-2/n^{\alpha}+1/n^{2\alpha}}\sum_{i=1}^{k^{*}_{n}}\frac{i^{2}\log(i+1))^{2\beta}}{i^{2}\log(i+1))^{2\beta}}.
\end{align*}
From the discussion in the proof of Proposition \ref{example_1}, $\lim_{n\rightarrow \infty }\|\Sigma_{u}^{+}\corrcoef\|^{2}_{2}/n=0$ holds.

For the second condition of Theorem \ref{Non_ortho_Variant_Theorem12}, we need to show $\trace(\Sigma_{u}\Xi_{z})/\trace(\Xi_{z}^{2})$ is finite. 
By the definition of $\Sigma_{u}$ and $\Xi_{z}$, we have
\begin{align*}
    \frac{\trace(\Sigma_{u}\Xi_{z})}{\trace(\Xi_{z}^{2})}&=\frac{\left(\frac{n^{\alpha}-1}{n^{2\alpha}}\sum_{i=1}^{k_{n}^{*}}\frac{1}{i^{2}\log(i+1)^{2\beta}}\right)}{\frac{1}{n^{2\alpha}}\sum_{i=1}^{k_{n}^{*}}\frac{1}{i^{2}\log(i+1)^{2\beta}}+\sum_{i=k_{n}^{*}+1}^{p}\frac{1}{i^{2}\log(i+1)^{2\beta}}} \\
    &=\left(\frac{1}{n^{\alpha}-1}+\frac{\frac{n^{2\alpha}}{n^{\alpha}-1}\sum_{i=k_{n}^{*}+1}^{p}\frac{1}{i^{2}\log(i+1)^{2\beta}}}{\left(\sum_{i=1}^{k_{n}^{*}}\frac{1}{i^{2}\log(i+1)^{2\beta}}\right)}\right)^{-1}.
\end{align*}
By trivial calculation, we have
\begin{align*}
    \frac{n^{2\alpha}}{n^{\alpha}-1}\sum_{i=k_{n}^{*}+1}^{p}\frac{1}{i^{2}\log(i+1)^{2\beta}}&=\frac{n^{\alpha}}{n^{\alpha}-1}n^{\alpha}\sum_{i=k_{n}^{*}+1}^{p}\frac{1}{i^{2}\log(i+1)^{2\beta}}\\
    &\gtrsim \frac{n^{\alpha}}{n^{\alpha}-1}\left(1-\frac{1}{q}\frac{k_{n}^{*}}{n}\right)\left(\frac{n^{\frac{\alpha-1}{2\beta}}}{\log(qn+1)}\right)^{2\beta}.
\end{align*}
As $\left(n^{\frac{\alpha-1}{\beta}}/\log(qn+1)\right)^{2\beta}$ diverges to infinity and $k_{n}^{*}/n$ converges to zero as $n$ goes to infinity, it holds that $\lim_{n\rightarrow\infty}\trace(\Sigma_{u}\Xi_{z})/\trace(\Xi_{z}^{2})=0$.

Finally, we show the setting in Proposition \ref{non-ortho_ex} satisfies the last condition stated in Theorem \ref{Non_ortho_Variant_Theorem12}. By trivial calculation, we have
\begin{align*}
    \corrcoef^\top \Sigma_{u}^{+}\Xi_{z} \Sigma_{u}^{+}\corrcoef&\lesssim\frac{1}{n^{\alpha}-1}\sum_{i=1}^{k^{*}_{n}}\frac{1}{i\log^{\beta}(i+1))} \lesssim\frac{k^{*}_{n}}{n^{\alpha}}.
\end{align*}
Therefore, $\lim_{n\rightarrow\infty}\corrcoef^\top \Sigma_{u}^{+}\Xi_{z} \Sigma_{u}^{+}\corrcoef=0$ holds. 
\end{proof}

\begin{lemma}[Application of Theorem 5.1.4  in  \citet{vershynin2018high}]\label{Variant_Lemma1}
Assume $S$ is a subspace of dimension $d$ in $\mathbb{R}^{n}$ where $n\geq4$. Let $P_{S}$ denote the orthogonal projection onto $S$ and let $V$ denote a spherically symmetric random variable. Then, with at least $1-\delta$ probability, we have 
\begin{equation}\label{Variant_25}
    \frac{\|P_{S}V\|_{2}}{\|V\|_{2}}\leq\sqrt{\frac{d}{n}}+2\sqrt{\frac{\log(2/\delta)}{n}}.
\end{equation}
From this inequality, we also have
\begin{equation}\label{Variant_26}
    |\langle s, V\rangle |=|\langle s, P_{S}V\rangle |\leq \|s\|_{2}\|P_{S}V\|_{2}\leq  \|s\|_{2}\|V\|_{2}\left(\sqrt{\frac{d}{n}}+2\sqrt{\frac{\log(2/\delta)}{n}}\right).
\end{equation}
\end{lemma}

\begin{lemma}[Theorem 3.1.1 in \citet{vershynin2018high}]\label{Variant_Lemma2}
Suppose that $Z\sim N(0,I_{n})$. Then,
\begin{equation*}
    \Pr(|\|Z\|_{2}-\sqrt{n}|\geq t)\leq 4e^{-t^{2}/4}.
\end{equation*}
\end{lemma}

\begin{lemma}[Lemma 9 in \citet{koehler2022uniform}]\label{Variant_Lemma9}
Let $H$ be normally distributed with mean zero and variance $I_{d}$, that is, $H\sim N(0,I_{d})$. For any covariance matrix $\Sigma$, it holds that
\begin{equation}\label{Variant_72}
    \left(E\|\Sigma^{1/2}H\|_{2}\right)^{2}\geq\left(1-\frac{1}{r(\Sigma)}\right)\trace(\Sigma)
\end{equation}
and
\begin{equation*}
    \frac{1}{\trace(\Sigma)}\geq\left(1-\sqrt{\frac{8}{r(\Sigma)}}\right)E\left[\frac{1}{H^{T}\Sigma H}\right].
\end{equation*}
Consequently, it holds that
\begin{equation}\label{Variant_74}
    r(\Sigma)-1\leq r_{\|\cdot\|_{2}}(\Sigma)\leq r(\Sigma)
\end{equation}
and
\begin{equation*}
    1-\frac{4}{\sqrt{r(\Sigma)}}\leq\frac{R_{\|\cdot\|_{2}}(\Sigma)}{R(\Sigma)}\leq\left(1-\sqrt{\frac{8}{r(\Sigma^{2})}}\right)^{-1},
\end{equation*}
where we define
\begin{equation*}
    r(\Sigma)=\frac{\trace(\Sigma)}{\|\Sigma\|_{op}},\ \ \ R(\Sigma)=\frac{\trace(\Sigma)^{2}}{\trace(\Sigma^{2})},\ \ \
    r_{\|\cdot\|}(\Sigma)=\left(\frac{E\|\Sigma^{1/2}H\|_{*}}{\sup_{\|u\|\leq1}\|u\|_{\Sigma}}\right)^{2},\ \ \ 
\end{equation*}
and
\begin{equation*}
    R_{\|\cdot\|}(\Sigma)=\left(\frac{E\|\Sigma^{1/2}H\|_{*}}{E\|v^{*}\|_{\Sigma}}\right)^{2}.
\end{equation*}
\end{lemma}

\begin{lemma}[Lemma 10 in \citet{koehler2022uniform}]
Let $H$ be normally distributed with mean zero and variance $I_{d}$, that is, $H\sim N(0,I_{d})$. For any covariance matrix $\Sigma$, it holds that with probability at least $1-\delta$
\begin{equation}\label{Variant_76}
    1-\frac{\|\Sigma^{1/2}H\|_{2}^{2}}{\trace(\Sigma)}\lesssim\frac{\log(4/\delta)}{\sqrt{R(\Sigma)}}
\end{equation}
and
\begin{equation*}
\|\Sigma H\|_{2}^{2}\lesssim \log(4/\delta)\trace(\Sigma^{2}).
\end{equation*}
Therefore, provided that $R(\Sigma)\gtrsim\log(4/\delta)^{2}$, it holds that
\begin{equation}\label{Variant_78}
    \left(\frac{\|\Sigma H\|_{2}}{\|\Sigma^{1/2} H\|_{2}}\right)^{2}\lesssim \log(4/\delta)\frac{\trace(\Sigma^{2})}{\trace(\Sigma)}.
\end{equation}
\end{lemma}

\begin{theorem}[Theorem 3.25 in \citet{van2014probability}]\label{Variant_Theorem6}
Assume $f$ is $L$-Lipschitz continuous with respect to the Euclidean norm with $L>0$, that is, for $f:\mathbb{R}^{n}\rightarrow \mathbb{R}$, 
\begin{equation*}
    |f(x)-f(y)|\leq L\|x-y\|_{2},
\end{equation*}
for all $x,y\in\mathbb{R}^{n}$. Then, we have
\begin{equation}\label{Variant_23}
    \Pr(|f (Z)-E[f(Z)]|\geq t)\leq 2e^{-t^{2}/2L^{2}},
\end{equation}
where  $Z\sim N(0, I_{n})$.
\end{theorem}

\bibliographystyle{apecon}
\bibliography{main}

\end{document}